\documentclass{article}

\usepackage{microtype}
\usepackage{graphicx}
\usepackage{subfigure}
\usepackage{booktabs} 
\usepackage{caption}

\usepackage{hyperref}

\usepackage[accepted]{icml2020}
\usepackage{amsthm}
\usepackage{amsmath}
\usepackage{amssymb}
\usepackage{algorithm}
\usepackage{paralist}
\usepackage{multirow}

\theoremstyle{plain}
\newtheorem{theorem}{Theorem}[section]
\newtheorem{corollary}{Corollary}[section]
\newtheorem{lemma}{Lemma}[section]
\theoremstyle{definition}
\newtheorem{definition}{Definition}[section]
\newtheorem{assumption}{Assumption}[section]
\theoremstyle{remark}
\newtheorem{remark}{Remark}[section]

\newcommand{\R}{\mathbb{R}}
\newcommand{\Rext}{\R\cup\{+\infty\}}
\newcommand{\set}[1]{\left\{#1\right\}}
\newcommand{\sets}[1]{\{#1\}}
\newcommand{\norm}[1]{\left\Vert#1\right\Vert}
\newcommand{\norms}[1]{\Vert#1\Vert}

\newcommand{\prox}{\mathrm{prox}}

\newcommand{\argmin}{\mathrm{arg}\!\displaystyle\min}

\newcommand{\Exps}[2]{\mathbb{E}_{#1}\left[#2\right]}
\newcommand{\Prob}[1]{\mathbf{Prob}\left(#1\right)}
\newcommand{\xb}{x}
\newcommand{\Hb}{\mathbf{H}}

\newcommand{\Bc}{\mathcal{B}}
\newcommand{\Xc}{\mathcal{X}}

\newcommand{\Lc}{\mathcal{L}}
\newcommand{\Qc}{\mathcal{Q}}
\newcommand{\Tc}{\mathcal{T}}
\newcommand{\Fb}{\mathbf{F}}
\newcommand{\Fc}{\mathcal{F}}

\newcommand{\iprods}[1]{\langle #1\rangle}
\newcommand{\Exp}[1]{\mathbb{E}\left[#1\right]}
\newcommand{\dist}[1]{\mathrm{dist}\left(#1\right)}
\newcommand{\BigO}[1]{\mathcal{O}\left(#1\right)}

\newcommand{\beforesubsec}{\vspace{0ex}}
\newcommand{\aftersubsec}{\vspace{0ex}}
\newcommand{\beforesec}{\vspace{0ex}}
\newcommand{\aftersec}{\vspace{0ex}}

\newcommand{\myeq}[2]{
\begin{equation}\label{#1}
{#2}
\end{equation}
}

\icmltitlerunning{Stochastic Gauss-Newton Algorithms for Nonconvex Compositional Optimization}

\begin{document}

\twocolumn[
\icmltitle{Stochastic Gauss-Newton Algorithms for Nonconvex Compositional Optimization}



\icmlsetsymbol{equal}{*}

\begin{icmlauthorlist}
\icmlauthor{Quoc Tran-Dinh}{1} \quad\qquad
\icmlauthor{Nhan H. Pham}{1}  \quad\qquad
\icmlauthor{Lam M. Nguyen}{2}
\end{icmlauthorlist}

\icmlaffiliation{1}{Department of Statistics and Operations Research, The University of North Carolina at Chapel Hill, NC, USA.}
\icmlaffiliation{2}{IBM Research, Thomas J. Watson Research Center, NY, USA}

\icmlcorrespondingauthor{Quoc Tran-Dinh}{quoctd@email.unc.edu}

\icmlkeywords{Machine Learning, ICML}

\vskip 0.3in
]



\printAffiliationsAndNotice{The first version was online on Feb 17, 2020 on Arxiv. This is the second version (ICML 2020).\newline}  

\begin{abstract}
We develop two new stochastic Gauss-Newton algorithms for solving a class of non-convex  stochastic compositional optimization problems frequently arising in practice. 
We consider both the expectation and finite-sum settings under standard assumptions, and use both classical stochastic and SARAH estimators for approximating function values and Jacobians. 
In the expectation case, we establish $\BigO{\varepsilon^{-2}}$ iteration-complexity to achieve a stationary point in expectation and estimate the total number of stochastic oracle calls for both function value and its Jacobian, where $\varepsilon$ is a desired accuracy. 
In the finite sum case, we also estimate $\BigO{\varepsilon^{-2}}$ iteration-complexity and the total oracle calls with high probability. 
To our best knowledge, this is the first time such global stochastic oracle complexity is established for stochastic Gauss-Newton methods.
Finally, we illustrate our theoretical results via two numerical examples on both synthetic and real datasets.
\end{abstract}

\beforesec
\section{Introduction}\label{sec:intro}
\aftersec
We consider the following nonconvex stochastic compositional nonconvex optimization problem:
\begin{equation}\label{eq:nl_least_squares}
\min_{\xb\in\R^p}\Big\{ \Psi(\xb) := \phi(F(x))  \equiv \phi\Big(\Exps{\xi}{\Fb(x,\xi)}\Big) \Big\},
\end{equation}
where $\Fb : \R^p\times\Omega\to \R^q$ is a stochastic function defined on a probability space $(\Omega, \mathbb{P})$, $\phi : \R^q\to\Rext$ is a proper, closed, and convex, but not necessarily smooth function, and $F$ is the expectation of $\Fb$ w.r.t. to $\xi$.

As a special case, if $\Omega$ is finite, i.e. $\Omega := \set{\xi_1,\cdots, \xi_n}$ and $\mathbb{P}(\xi=\xi_i) = \mathbf{p}_i > 0$ for $i \in [n] := \set{1,\cdots, n}$ and $\sum_{i=1}^n\mathbf{p}_i=1$, then by introducting $F_i(x) := n\mathbf{p}_i\Fb(x,\xi_i)$, $F(x)$ can be written into a finite-sum $F(x) := \frac{1}{n}\sum_{i=1}^nF_i(x)$, and \eqref{eq:nl_least_squares} reduces to
\myeq{eq:finite_sum}{
\min_{\xb\in\R^p}\bigg\{ \Psi(x) := \phi(F(x)) \equiv \phi\Big(\frac{1}{n}\sum_{i=1}^nF_i(x)\Big)\bigg\}.
}
This expression can also be viewed as a stochastic average approximation of $F(x) := \Exps{\xi}{\Fb(x,\xi)}$ in \eqref{eq:nl_least_squares}.
Note that the setting  \eqref{eq:nl_least_squares} is completely different from $\min_x\set{\Psi(x) := \Exps{\xi}{\phi(\Fb(x,\xi),\xi)}}$ in \citet{davis2017proximally,davis2019stochastic,duchi2018stochastic}.

Problem \eqref{eq:nl_least_squares} or its special form \eqref{eq:finite_sum} covers various applications in different domains (both deterministic and stochastic) such as penalized problems for constrained optimization, parameter estimation, nonlinear least-squares, system identification, statistical learning, dynamic programming, and minimax problems \cite{drusvyatskiy2019efficiency,duchi2018stochastic,Lewis2008,Nesterov2007g,Tran-Dinh2011,wang2017stochastic}.
Note that both \eqref{eq:nl_least_squares} and \eqref{eq:finite_sum} cover the composite form 
\myeq{eq:composite_form}{
\min_{x\in\R^p}\Big\{ \Psi(x) := \phi(F(x)) + g(x) \Big\},
}
 for a given convex function $g$ if we  introduce $\hat{\phi}(\cdot) := \phi(\cdot) + g(\cdot)$ and $\hat{F}(x) := [F(x); x]$ to reformulate it into \eqref{eq:nl_least_squares} or \eqref{eq:finite_sum}.
This formulation, on the other hand, is an extension of \eqref{eq:nl_least_squares}.
 We will also show how to handle \eqref{eq:composite_form} in Subsection~\ref{subsec:composite}.

Our goal in this paper is to develop novel stochastic methods to solve \eqref{eq:nl_least_squares} and \eqref{eq:finite_sum} based on the following assumptions:
\begin{assumption}\label{as:A1}
The function $\Psi$ of \eqref{eq:nl_least_squares} is bounded from below on its domain, i.e. $\Psi^{\star} := \inf_{x}\Psi(x) > -\infty$.
The function $\phi$ is $M_{\phi}$-Lipschitz continuous, and $F$ is $L_F$-average smooth, i.e., there exist $M_{\phi}, L_F\in (0, +\infty)$ such that
\myeq{eq:lips}{
\hspace{-1.5ex}
\arraycolsep=0.1em
\left\{\begin{array}{ll}
& \vert\phi(u) - \phi(v)\vert \leq M_{\phi}\norms{u-v},~\forall u, v \in \R^q, \vspace{1ex}\\
&\Exps{\xi}{\norms{\Fb'(x,\xi) - \Fb'(y,\xi)}^2} \leq L_F^2\norms{x-y}^2,~\forall x, y.
\end{array}\right.
\hspace{-2ex}
}
For the finite-sum case \eqref{eq:finite_sum}, we impose a stronger assumption that $\norms{F_i'(x) - F_i'(y)} \leq L_F\norms{x-y}$ for all $x,y\in\R^p$ and all $i\in [n]$.
Here, we use spectral norm for Jacobian.
\end{assumption}
\begin{assumption}\label{as:A2}
There exist  $\sigma_F, \sigma_D \in [0,+\infty)$ such that the variance of $F$ and $F'$ is uniformly bounded, i.e., $\Exps{\xi}{\norms{\Fb(x,\xi) - F(x)}^2} \leq \sigma_F^2$ and $\Exps{\xi}{\norms{\Fb'(x,\xi) - F'(x)}^2} \leq \sigma_D^2$, respectively.
In the finite sum case \eqref{eq:finite_sum}, we again impose stronger conditions $\norms{F_i(x) - F(x)} \leq \sigma_F$ and $\norms{F_i'(x) - F'(x)} \leq \sigma_D$ for all $x\in\R^p$ and for all $i\in [n]$. 
\end{assumption}
Assumptions \ref{as:A1} and \ref{as:A2} are standard and cover a wide class of models in practice as opposed to existing works.
The stronger assumptions imposed on \eqref{eq:finite_sum} allow us to develop adaptive subsampling schemes later.

\noindent\textbf{Related work.}
Problem \eqref{eq:nl_least_squares} or \eqref{eq:finite_sum} has been widely studied in the literature under both deterministic (including  the finite-sum \eqref{eq:finite_sum} and $n=1$) and stochastic settings, see, e.g.,   \cite{drusvyatskiy2019efficiency,duchi2018stochastic,Lewis2008,Nesterov2007g,Tran-Dinh2011,wang2017stochastic}.
If  $q=1$ and $\phi(u) = u$, then \eqref{eq:nl_least_squares} reduces to the standard stochastic optimization model studied in, e.g. \citet{ghadimi2016accelerated,Pham2019}.
In the deterministic setting, the common method to solve \eqref{eq:nl_least_squares} is the \textit{Gauss-Newton} (GN) scheme, which is also known as the \textit{prox-linear method}.
This method has been studied in several papers, including \citet{drusvyatskiy2019efficiency,duchi2018stochastic,Lewis2008,Nesterov2007g,Tran-Dinh2011}.
In such settings, GN only requires Assumption~\ref{as:A1} to have global convergence guarantees \cite{drusvyatskiy2019efficiency,Nesterov2007g}.

In the stochastic setting of the form \eqref{eq:nl_least_squares}, \citet{wang2017stochastic,wang2017accelerating} proposed stochastic compositional gradient descent methods to solve more general forms than \eqref{eq:nl_least_squares}, but they required a set of stronger assumptions than Assumptions~\ref{as:A1} and \ref{as:A2}, including the smoothness of $\phi$.
These methods eventually belong to a gradient-based class.
Other works in this direction include \citet{lian2017finite,yu2017fast,yang2019multilevel,liu2017variance,xu2019katyusha}, which also rely on a similar approach.
Together with algorithms, convergence guarantees and stochastic oracle complexity bounds have also been estimated.
For instance, \citet{wang2017stochastic} estimates $\BigO{\varepsilon^{-8}}$ oracle complexity for solving \eqref{eq:nl_least_squares}, while it is improved to $\BigO{\varepsilon^{-4.5}}$ in \citet{wang2017accelerating}.
Recent works such as \citet{zhang2019multi} further improve the complexity to $\BigO{\varepsilon^{-3}}$.
However, these methods are completely different from GN and require much stronger assumptions, including the smoothness of $\phi$ and $F$.

One main challenge to design algorithms for solving \eqref{eq:nl_least_squares} is the bias of stochastic estimators.
Some researchers have tried to remedy this issue by proposing more sophisticated sampling schemes, see, e.g., \citet{blanchet2017unbiased}.
Other works relies on biased estimators but using variance reduction techniques, e.g., \citet{zhang2019multi}.

\noindent\textbf{Challenges.}
The stochastic formulation \eqref{eq:nl_least_squares} creates several challenges for developing numerical methods.
First, it is often nonconvex. 
Many papers consider special cases when $\Psi$ is convex. 
This only holds if $\phi$ is convex and $F$ is linear, or $\phi$ is convex and monotone and $F$ is convex or concave.
Clearly, such a setting is almost unrealistic or very limited. 
One can assume weak convexity of $\Psi$ and add a regularizer to make the resulting problem convex but this completely changes the model.
Second, $\phi$ is often non-smooth such as norm, penalty, or gauge functions.
This prevents the use of gradient-based methods.
Third, even when both $\phi$ and $F$ are smooth, to guarantee Lipschitz continuity of $\nabla\Psi$, it requires simultaneously $F$, $F'$, $\phi$, and $\nabla{\phi}$ to be Lipschitz continuous.
This condition is very restrictive and often requires additional bounded constraints or bounded domain assumption.
Otherwise, it  fails  to hold even for bilinear functions.
Finally, in stochastic settings, it is very challenging to form unbiased estimate for gradients or subgradients of $\Psi$, making classical stochastic-based method inapplicable.

\noindent\textbf{Our approach and contribution.}
Our main motivation is to overcome the above challenges by following a different approach.\footnote{When this paper was under review, \citet{zhang2020stochastic} was brought to our attention, which presents similar methods.}
We extend the GN method from the deterministic setting \cite{Lewis2008,Nesterov2007g} to the stochastic setting \eqref{eq:nl_least_squares}.
Our methods can be viewed as inexact variants of GN using stochastic estimators for both function values $F(x)$ and its Jacobian $F'(x)$.
This approach allows us to cover a wide class of \eqref{eq:nl_least_squares}, while only requires standard assumptions as Assumptions~\ref{as:A1} and \ref{as:A2}.
Our contribution can be summarized as follows: 

\begin{compactitem}
\vspace{-1ex}
\item[(a)] We develop an inexact GN framework to solve \eqref{eq:nl_least_squares} and \eqref{eq:finite_sum} using inexact estimations of $F$ and its Jacobian $F'$.
This framework is independent of approximation schemes for generating approximate estimators.
We characterize approximate stationary points of \eqref{eq:nl_least_squares} and \eqref{eq:finite_sum} via prox-linear gradient mappings.
Then, we prove global convergence guarantee of our method to a stationary point under appropriate inexact computation.

\item[(b)] We analyze stochastic oracle complexity of our GN algorithm when  mini-batch stochastic estimators are used. 
We separate our analysis into two cases. 
The first variant is to solve \eqref{eq:nl_least_squares}, where we obtain convergence guarantee in expectation.
The second variant is to solve \eqref{eq:finite_sum}, where we use adaptive mini-batches and obtain convergence guarantee with high probability.

\item[(c)] We also provide oracle complexity of this algorithm when  mini-batch SARAH estimators in \citet{nguyen2017sarah,Nguyen2019_SARAH} are used for both \eqref{eq:nl_least_squares} and \eqref{eq:finite_sum}.  
Under an additional mild assumption,  this estimator significantly improves the oracle complexity  by an order of $\varepsilon$ compared to the  mini-batch stochastic one.
\vspace{-1ex}
\end{compactitem}
We believe that our methods are the first ones to achieve global convergence rates and stochastic oracle complexity for solving \eqref{eq:nl_least_squares} and \eqref{eq:finite_sum} under standard assumptions.
It is completely different from existing works such as  \citet{wang2017stochastic,wang2017accelerating,lian2017finite,yu2017fast,yang2019multilevel,zhang2019multi}, where we only use Assumptions~\ref{as:A1} and \ref{as:A2}, while not imposing any special structure on $\phi$ and $F$, including smoothness.
When using SARAH estimators, we impose the Lipschitz continuity of $F$ to achieve better oracle complexity.
This additional assumption is still much weaker than the ones used in existing works.
However, without this assumption, our GN scheme with SARAH estimators still converges (see Remark~\ref{re:convergnece2}). 

\noindent\textbf{Content.}
Section~\ref{sec:math_tools} recalls some mathematical tools.
Section~\ref{sec:inexact_gn_method} develops an inexact GN framework.
Sections~\ref{sec:sgn_methods} analyzes convergence and complexity of the two stochastic GN variants  using different stochastic estimators.
Numerical examples are given in Section~\ref{sec:num_exp}.
All the proofs and discussion are deferred to Supplementary Document (Supp. Doc.).

\beforesec
\section{Background and Mathematical Tools}\label{sec:math_tools}
\aftersec
We first characterize the optimality condition of \eqref{eq:nl_least_squares}.
Next, we recall the prox-linear mapping of the compositional function $\Psi(x) := \phi(F(x))$ and its properties.

\noindent\textbf{\textit{Basic notation.}}
We work with Euclidean spaces $\R^p$ and $\R^q$.
Given a convex set $\Xc$, $\dist{u,\Xc} := \inf_{x\in\Xc}\norm{u - x}$ denotes the Euclidean distance from $u$ to $\Xc$.
For a convex function $f$, we denote $\partial{f}$ its subdifferential, $\nabla{f}$ its gradient, and $f^{*}$ its Fenchel conjugate.
For a smooth function $F:\R^p\to\R^q$, $F'(\cdot)$ denotes its Jacobian.
For vectors, we use Euclidean norms, while for matrices, we use spectral norms, i.e., $\norms{X} := \sigma_{\max}(X)$.
$\lfloor\cdot\rfloor$ stands for number rounding.

\beforesubsec
\subsection{Exact and Approximate Stationary Points}\label{subsec:opt_cond}
\aftersubsec
The optimality condition of \eqref{eq:nl_least_squares} can be written as
\begin{equation}\label{eq:opt_cond}
\begin{array}{ll}
&0 \in \partial{\Psi}(x^{\star}) \equiv F'(x^{\star})^{\top}\partial{\phi}(F(x^{\star})), \vspace{1ex}\\
&\text{or equivalently}~~\dist{0, \partial{\Psi}(x^{\star})} = 0.
\end{array}
\end{equation}
Any  $x^{\star}$ satisfying \eqref{eq:opt_cond} is called a stationary point of \eqref{eq:nl_least_squares} or \eqref{eq:finite_sum}.{\!\!\!}

Since $\phi$ is convex, let $\phi^{*}$ be its Fenchel conjugate and $y^{\star} \in \partial{\phi}(F(x^{\star}))$.
Then, \eqref{eq:opt_cond} can be rewritten as
\begin{equation}\label{eq:opt_cond2}
0 = F'(x^{\star})^{\top}y^{\star}~~\text{and}~~ 0 \in - F(x^{\star}) +  \partial{\phi^{*}}(y^{\star}).
\end{equation}
Now, if we define 
\begin{equation}\label{eq:residual}
\mathcal{E}(x, y) := \norms{F'(x)^{\top}y} + \dist{0, - F(x) +  \partial{\phi^{*}}(y)},
\end{equation}
then the optimality condition \eqref{eq:opt_cond} of \eqref{eq:nl_least_squares} or \eqref{eq:finite_sum} becomes 
\begin{equation}\label{eq:opt_cond1}
\mathcal{E}(x^{\star}, y^{\star}) = 0.
\end{equation}
Note that once a stationary point $x^{\star}$ is available, we can compute $y^{\star}$ as any element $y^{\star}\in\partial{\phi}(F(x^{\star}))$ of $\phi\circ F$.

In practice, we can only find an approximate stationary point $\hat{x}$ and its dual $\hat{y}$ such that $(\hat{x}, \hat{y})$ approximates $(x^{\star}, y^{\star})$  of \eqref{eq:nl_least_squares} or \eqref{eq:finite_sum} up to a given accuracy $\varepsilon \geq 0$ as follows:
\begin{definition}\label{de:approx_sol}
Given $\varepsilon > 0$, we call $\hat{x} \in \R^p$ an $\varepsilon$-stationary point of \eqref{eq:nl_least_squares} if there exists $\hat{y}\in\R^q$ such that
\begin{equation}\label{eq:approx_opt_cond1}
\mathcal{E}(\hat{x}, \hat{y}) \leq \varepsilon,
\end{equation}
where $\mathcal{E}(\cdot)$ is defined by \eqref{eq:residual}.
This condition can be characterized in expectation, where $\Exp{\cdot}$ is taken over all the randomness generated by the problem and the corresponding stochastic algorithm, or with high probability $1-\delta$.
Such guarantees will be specified in the sequel.
\end{definition}

\beforesubsec
\subsection{Prox-Linear Operator and Its Properties}\label{subsec:prox_linear_map}
\aftersubsec
\noindent\textbf{(a)~\textit{Prox-linear operator.}}
Since we assume that the Jacobian $F'(\cdot)$ of $F$ is Lipschitz continuous with a Lipschitz constant $L_F \in (0, +\infty)$, and $\phi$ is $M_{\phi}$-Lipschitz continuous as in Assumption~\ref{as:A1}, we have (see Supp. Doc. \ref{apdx:sec:math_tools}):
\begin{equation}\label{eq:key_est1}
\phi(F(z)) \leq \phi(F(x) + F'(x)(z - x)) + \frac{M_{\phi}L_F}{2}\norms{z-x}^2,
\end{equation}
for all $z, x\in\R^p$.
Given $x \in \R^p$, let $\widetilde{F}(x) \approx F(x)$ and $\widetilde{J}(x) \approx F'(x)$ be  a deterministic or stochastic approximation of $F(x)$ and its Jacobian $F'(x)$, respectively.
We consider the following approximate prox-linear model:
\begin{align}\label{eq:surg_model}
\widetilde{T}_M(x) &:= \argmin_{z \in\R^p}\Big\{\widetilde{\Qc}_M(z; x) \!:=\!  \phi(\widetilde{F}(x) \!+\! \widetilde{J}(x)(z \!-\! x)) \nonumber\\
&\hspace{12ex} +  \tfrac{M}{2}\norms{z - x}^2\Big\},
\end{align}
where $M > 0$ is a given constant.
As usual, if $\widetilde{F}(x) = F(x)$ and $\widetilde{J}(x) = F'(x)$, then
\begin{align}\label{eq:exact_surg_model}
T_M(x) &:= \argmin_{z \in\R^p}\Big\{ \Qc_M(z; x) \!:=\!  \phi(F(x) \!+\! F'(x)(z \!-\! x)) \nonumber\\
&\hspace{12ex} +\tfrac{M}{2}\norms{z - x}^2\Big\}
\end{align}
is the exact prox-linear operator of $\Psi$.
In this context, we also call $\widetilde{T}_M(\cdot)$ an approximate prox-linear operator of $\Psi$.

\noindent\textbf{(b)~\textit{Prox-gradient mapping.}}
We also define the prox-gradient mapping and its approximation, respectively as 
\begin{equation}\label{eq:grad_map1}
\left\{\begin{array}{ll}
G_M(x) &:= M(x - T_M(x)), \vspace{1ex}\\
\widetilde{G}_M(x) &:= M(x - \widetilde{T}_M(x)).
\end{array}\right.
\end{equation}
Clearly if $\norms{G_M(x)} = 0$, then $x = T_M(x)$ and $x$ is a stationary point of \eqref{eq:nl_least_squares}.
In our context, we can only compute $\widetilde{G}_M(x)$ as an approximation of $G_M(x)$.

\noindent\textbf{(c)~\textit{Characterizing approximate stationary points.}}
The following lemma bounds the optimality error $\mathcal{E}(\cdot)$ defined by \eqref{eq:residual} via the approximate prox-gradient mapping $\widetilde{G}_M(x)$.

\begin{lemma}\label{le:aprox_opt_cond}
Let $\widetilde{T}_M(x)$ be computed by \eqref{eq:surg_model} and $\widetilde{G}_M(x)$ be defined by \eqref{eq:grad_map1}.
Then, $\mathcal{E}(\widetilde{T}_M(x), y)$ of \eqref{eq:nl_least_squares} defined by \eqref{eq:residual} with $y\in \partial{\phi}(F(\widetilde{T}_M(x)))$ is bounded by
\begin{align}\label{eq:approx_opt_cond}
\hspace{-1ex}
&\mathcal{E}(\widetilde{T}_M(x), y) 
 \leq   \left(\! 1  \! + \! \tfrac{M_{\phi}L_F}{M}\! \right) \norms{\widetilde{G}_M(x)} \! + \!  \tfrac{(1 \! + \! L_F)}{2M^2} \norms{\widetilde{G}_M(x)}^2 \nonumber\\
&\hspace{8ex}+ \norms{\widetilde{F}(x) - F(x)} +  \tfrac{1}{2}\norms{\widetilde{J}(x) - F'(x)}^2.
\hspace{-1ex}
\end{align}
\end{lemma}
Clearly, if we use exact oracles $\widetilde{F}(x) = F(x)$ and $\widetilde{J}(x) = F'(x)$, then $\mathcal{E}(\widetilde{T}_M(x), y)$ is reduced to
\begin{equation*}
\hspace{-1ex}
\begin{array}{ll}
\mathcal{E}(T_M(x), y) \leq \left(1  \!+\! \frac{M_{\phi}L_F}{M}\right)\norms{G_M(x)} \!+\!  \frac{(1\!+\! L_F)}{2M^2}\norms{G_M(x)}^2.
\end{array}
\hspace{-2ex}
\end{equation*}
Alternatively, from \eqref{eq:approx_opt_cond}, if we can guarantee $ \norms{\widetilde{F}(x) - F(x)} \leq \BigO{\varepsilon}$, $\norms{\widetilde{J}(x) - F'(x)} \leq \BigO{\sqrt{\varepsilon}}$, and $\norms{\widetilde{G}_M(x)} \leq \BigO{\varepsilon}$, then
\begin{equation*}
{\!\!\!\!\!}\begin{array}{ll}
\mathcal{E}(\widetilde{T}_M(x), y) {\!\!\!\!}&\leq \BigO{\varepsilon}, 
\end{array}{\!\!\!\!}
\end{equation*}
which  shows that $\widetilde{T}_M(x)$ is a $\BigO{\varepsilon}$-stationary point of \eqref{eq:nl_least_squares} in the sense of Definition~\ref{de:approx_sol}.
Our goal is to approximate $F$ and $F'$ and compute $\widetilde{G}_M(x)$ to guarantee these conditions.

\beforesec
\section{Inexact Gauss-Newton Framework}\label{sec:inexact_gn_method}
\aftersec
In this section, we develop a conceptual inexact Gauss-Newton (iGN) framework for solving \eqref{eq:nl_least_squares} and \eqref{eq:finite_sum}.

\beforesubsec
\subsection{Descent Property and Approximate Conditions}
\aftersubsec
Lemma~\ref{le:descent_property} provides a key bound regarding \eqref{eq:surg_model}, which will be used for convergence analysis of our algorithms.

\begin{lemma}\label{le:descent_property}
Let Assumption~\ref{as:A1} hold and $\widetilde{T}_M(x)$ be computed by \eqref{eq:surg_model}.
Then, for any $\beta_d > 0$, we also have
\begin{equation}\label{eq:key_est5}
\hspace{-1ex}
\arraycolsep=0.1em
\begin{array}{lcl}
\phi(F(\widetilde{T}_M(x))) 
&\leq & \phi(F(x)) + 2L_\phi\norms{F(x) - \widetilde{F}(x)}  \vspace{1ex}\\
&& + {~} \frac{M_{\phi}}{2\beta_d}\Vert F'(x) - \widetilde{J}(x)\Vert_F^2  \vspace{1ex}\\
&& - {~} \frac{(2M - M_{\phi}L_F - \beta_d L_\phi)}{2}\norms{\widetilde{T}_M(x) - x}^2.
\end{array}
\hspace{-2ex}
\end{equation}
\end{lemma}
Since we approximate both $F$ and its Jacobian $F'$ in our prox-linear model \eqref{eq:surg_model}, we assume that this approximation satisfies one of the following two conditions:
\begin{compactitem}
\item\textbf{Condition 1:}
Given a tolerance $\varepsilon > 0$ and $M > \frac{1}{2}M_{\phi}(L_F + \beta_d)$, at each iterate $x_t\in\R^p$, it holds that
\begin{equation}\label{eq:a_cond1}
\arraycolsep=0.2em
\left\{\begin{array}{lcl}
\norms{\widetilde{F}(x_t) - F(x_t)} &\leq & \frac{C_g\varepsilon^2}{16M_{\phi}M^2}, \vspace{1ex}\\
\norms{\widetilde{J}(x_t) - F'(x_t)} &\leq & \frac{\sqrt{\beta_dC_g}\varepsilon}{M\sqrt{2M_{\phi}}},
\end{array}\right.
\end{equation}
where $C_g := 2M - M_{\phi}(L_F + \beta_d) > 0$.

\item\textbf{Condition 2:} 
Given $C_f > 0$, $C_d > 0$, and $\beta_d > 0$, let $C_g := 2M - M_{\phi}(L_F + \beta_d)$ and $C_a := 2M - M_{\phi}\big(L_F + \beta_d + 2\sqrt{C_f }+ \frac{C_d}{2\beta_d}\big)$ such that $C_a > 0$.
For $x_0\in\R^p$, we assume that
\begin{equation}\label{eq:a_cond2b}
\arraycolsep=0.2em
\left\{\begin{array}{lcl}
\norms{\widetilde{F}(x_0) - F(x_0)}  & \leq & \frac{C_a\varepsilon^2}{16M_{\phi} M^2} \vspace{1ex}\\
\norms{\widetilde{J}(x_0) - F'(x_0)} &\leq  & \frac{\sqrt{\beta_dC_a}\varepsilon}{M\sqrt{2M_{\phi}}},
\end{array}\right.
\hspace{-1ex}
\end{equation}
while, for any iterate $x_t\in\R^p$ ($t\geq 1$), we assume that
\begin{equation}\label{eq:a_cond2}
\hspace{0ex}
\arraycolsep=0.2em
\left\{\begin{array}{llcl}
&\norms{\widetilde{F}(x_t) - F(x_t)}    &\leq &  \sqrt{C_f}\norm{x_t - x_{t-1}}^2, \vspace{1ex}\\
&\norms{\widetilde{J}(x_t) - F'(x_t)}    &\leq & \sqrt{C_d} \norm{x_t - x_{t-1}}.
\end{array}\right. 
\hspace{-1ex}
\end{equation}
\end{compactitem}
The condition \eqref{eq:a_cond1} assumes that both $\widetilde{F}$ and $\widetilde{J}$ should respectively well approximate $F$ and $F'$ up to a given accuracy $\varepsilon$.
Here, the function value $F$ must have higher accuracy than its Jacobian $F'$.
The condition  \eqref{eq:a_cond2} is adaptive, which depends on the norm $ \norm{x_t - x_{t-1}}$ of the iterates $x_t$ and $x_{t-1}$.
This condition is less conservative than  \eqref{eq:a_cond1}.

\beforesubsec
\subsection{The Inexact Gauss-Newton Algorithm}
\aftersubsec
We first present a conceptual stochastic Gauss-Newton method as described in Algorithm~\ref{alg:A1}.

\begin{algorithm}[ht!]\caption{(Inexact Gauss-Newton~\textbf{(iGN)})}\label{alg:A1}
\normalsize
\begin{algorithmic}[1]
   \STATE{\bfseries Initialization:}  Choose $x_0 \in\R^p$ and $M > 0$.
   \STATE\hspace{0ex}\label{step:A1_o1}{\bfseries For $t := 0,\cdots, T$ do}
   \STATE\hspace{2ex}\label{step:A1_o2} Form $\widetilde{F}(x_t)$ and $\widetilde{J}(x_t)$ satisfying either \eqref{eq:a_cond1} or \eqref{eq:a_cond2}.
   \STATE\hspace{2ex}\label{step:A1_o3} Update $x_{t+1}  := \widetilde{T}_M(x_t)$ based on \eqref{eq:surg_model}.
   \STATE\hspace{0ex}{\bfseries End For}
\end{algorithmic}
\end{algorithm}

Algorithm~\ref{alg:A1} remains conceptual since we have not specified how to form $\widetilde{F}(x_t)$ and $\widetilde{J}(x_t)$.

\beforesubsec
\subsection{Convergence Analysis}\label{subsec:convergence1}
\aftersubsec
Let us first state the convergence of Algorithm~\ref{alg:A1} under \textbf{Conditon 1} or \textbf{Condition 2} in the following theorem.

\begin{theorem}\label{th:convergence1}
Assume that Assumptions~\ref{as:A1} and \ref{as:A2} are satisfied.
Let $\set{x_t}$ be generated by Algorithm~\ref{alg:A1} to solve either \eqref{eq:nl_least_squares} or \eqref{eq:finite_sum}.
Then, the following statements hold:

$\mathrm{(a)}$~If \eqref{eq:a_cond1} holds for some $\varepsilon \geq 0$, then 
\begin{equation}\label{eq:convergence_bound1}
\hspace{-0.25ex}
\frac{1}{(T\!+\!1)}\!\displaystyle\sum_{t=0}^T \norms{\widetilde{G}_M(x_t)}^2  \leq  \frac{2M^2\left[\Psi(x_0) - \Psi^{\star}\right]}{C_g(T+1)} + \frac{\varepsilon^2}{2},
\hspace{-0.5ex}
\end{equation}
where $C_g := 2M - M_{\phi}(L_F \!+\! \beta_d)$ with $M > \frac{1}{2}M_{\phi}(L_F \!+\! \beta_d)$.

$\mathrm{(b)}$~If \eqref{eq:a_cond2b} and \eqref{eq:a_cond2} hold for given $C_a > 0$, then
\begin{equation}\label{eq:convergence_bound2}
\hspace{-0.25ex}
\frac{1}{(T \! + \!1)}{\!}\displaystyle\sum_{t=0}^T\norms{\widetilde{G}_M(x_t)}^2 \leq  \dfrac{2M^2\left[\Psi(x_0) - \Psi^{\star}\right]}{C_a(T+1)} + \dfrac{\varepsilon^2}{2}.
\hspace{-0.5ex}
\end{equation}
Consequently, the total number of iterations $T$ to achieve $\frac{1}{(T  + 1)}\sum_{t=0}^T\norms{\widetilde{G}_M(x_t)}^2 \leq \varepsilon^2$ is at most 
\begin{equation*}
T := \left\lfloor \frac{4M^2\left[ \Psi(x_0) - \Psi^{\star} \right]}{D\varepsilon^2} \right\rfloor = \BigO{\frac{1}{\varepsilon^{2}}}, 
\end{equation*}
where $D := C_g$ for $\mathrm{(a)}$ and $D := C_a$ for $\mathrm{(b)}$.
\end{theorem}

\begin{remark}\label{re:E_vs_E2}
The guarantee  $\frac{1}{(T  + 1)}\sum_{t=0}^T\norms{\widetilde{G}_M(x_t)}^2 \leq \varepsilon^2$  implies that $\liminf_{t\to\infty,\varepsilon\downarrow 0^{+}}\norms{\widetilde{G}_M(x_t)} = 0$.
That is there exists subsequence $x_{t_k}$ of $\set{x_t}$ such that $\norms{\widetilde{G}_M(x_{t_k})} \to 0$ as $k\to +\infty$ and $\varepsilon\to 0$.
\end{remark}

\beforesec
\section{Stochastic Gauss-Newton Methods}\label{sec:sgn_methods}
\aftersec
\subsection{SGN with Mini-Batch Stochastic Estimators}
\aftersubsec
As a natural instance of Algorithm~\ref{alg:A1}, we propose to approximate $F(x_t)$ and $F'(x_t)$ in Algorithm~\ref{alg:A1} by mini-batch stochastic estimators as:
\begin{equation}\label{eq:msgd_estimators}
\left\{\begin{array}{lcl}
\widetilde{F}(x_t) &:= & \frac{1}{b_t}\sum_{\xi_i\in\Bc_t}\Fb(x_t, \xi_i), \vspace{1ex}\\
\widetilde{J}(x_t) &:= & \frac{1}{\hat{b}_t}\sum_{\zeta_j\in \hat{\Bc}_t}\Fb'(x_t, \zeta_j),
\end{array}\right.
\end{equation}
where the mini-batches $\Bc_t$ and $\hat{\Bc}_t$ are not necessarily independent, $b_t := \vert\Bc_t\vert$, and $\hat{b}_t := \vert \hat{\Bc}_t\vert$.
Using \eqref{eq:msgd_estimators} we prove our first result in expectation on stochastic oracle complexity of Algorithm~\ref{alg:A1} for solving \eqref{eq:nl_least_squares}.

In practice, we may not need to explicitly form $\widetilde{J}(x_t)$, but its matrix-vector product $\widetilde{J}(x_t)d$ for some vector $d$, when evaluating the prox-linear operator $\widetilde{\Tc}_M(x_t)$.
This requires $\Fb'(x_t, \zeta_j)d$, which can be evaluated efficiently by using, e.g., automatic differentiation techniques. 

\begin{theorem}\label{th:sgd_complexity1}
Suppose that Assumptions~\ref{as:A1} and \ref{as:A2} hold for \eqref{eq:nl_least_squares}.
Let $\widetilde{F}_t$ and $\widetilde{J}_t$ defined by \eqref{eq:msgd_estimators} be mini-batch stochastic estimators  of $F(x_t)$ and $F'(x_t)$, respectively.
Let $\set{x_t}$ be generated by Algorithm~\ref{alg:A1} $($called \textbf{SGN}$)$ to solve \eqref{eq:nl_least_squares}.
Assume that $b_t$ and $\hat{b}_t$ in \eqref{eq:msgd_estimators} are chosen as 
\begin{equation}\label{eq:bt_size1}
\left\{\begin{array}{lclcl}
b_t &:= & \left\lfloor \frac{256M_{\phi}^2M^4\sigma^2_F}{C_g^2\varepsilon^4} \right\rfloor  & = & \BigO{\frac{\sigma^2_F}{\varepsilon^4}} \vspace{1ex}\\
\hat{b}_t &:= &  \left\lfloor \frac{2M_{\phi}M^2\sigma_D^2}{\beta_dC_g\varepsilon^2} \right\rfloor & = & \BigO{\frac{\sigma^2_D}{\varepsilon^2}},
\end{array}\right.
\end{equation}
for some constant $C_f > 0$ and $C_d > 0$.
Furthermore, let $\widehat{x}_T$ be chosen uniformly at random in $\set{x_t}_{t=0}^T$ as the output of Algorithm~\ref{alg:A1} after $T$ iterations.
Then
\begin{equation}\label{eq:convergence_bound1_sgd}
\Exp{\norms{\widetilde{G}_M(\widehat{x}_T)}^2}  
 \leq  \dfrac{2M^2\left[\Psi(x_0) - \Psi^{\star}\right]}{C_g(T+1)} + \dfrac{\varepsilon^2}{2},
\end{equation}
where $C_g := 2M - M_{\phi}(L_F + \beta_d) > 0$. 

Moreover, the number $\Tc_f$ of function evaluations $\Fb(x_t,\xi)$ and the number $\Tc_d$ of Jacobian evaluations $\Fb'(x_t,\zeta)$ to achieve $\Exp{\norms{\widetilde{G}_M(\widehat{x}_T)}^2} \leq \varepsilon^2$ do not exceed 
\begin{equation}\label{eq:complexity_case1}
\hspace{-0ex}
\arraycolsep=0.2em
\left\{\begin{array}{lclcl}
\Tc_f  &:= &  \left\lfloor \frac{1024M^6M_{\phi}^2\sigma_F^2\left[\Psi(x_0) - \Psi^{\star}\right]}{C_g^3\varepsilon^6} \right\rfloor &= & \BigO{\frac{\sigma_F^2}{\varepsilon^6}}, \vspace{1ex}\\
\Tc_d  &:=  & \left\lfloor \frac{8M^4M_{\phi}\sigma_D^2\left[\Psi(x_0) - \Psi^{\star}\right]}{\beta_dC_g^2\varepsilon^4} \right\rfloor &= & \BigO{\frac{\sigma_D^2}{\varepsilon^4}}.
\end{array}\right.
\hspace{-1ex}
\end{equation}
\end{theorem}

Note that if we replace $b_t$ and $\hat{b}_t$ in \eqref{eq:bt_size1} by $\min\set{b_t, n}$ and $\min\sets{\hat{b}_t, n}$, respectively, then the result of Theorem~\ref{th:sgd_complexity1} still holds for \eqref{eq:finite_sum} since it is a special case of \eqref{eq:nl_least_squares}.

Now, we derive the convergence result of Algorithm~\ref{alg:A1} for solving \eqref{eq:finite_sum} using adaptive mini-batches.
However, our convergence guarantee is obtained with high probability.
 
\begin{theorem}\label{th:sgd_complexity2}
Suppose that Assumptions~\ref{as:A1} and \ref{as:A2} hold for \eqref{eq:finite_sum}.
Let $\widetilde{F}_t$ and $\widetilde{J}_t$ defined by \eqref{eq:msgd_estimators} be mini-batch stochastic estimators  to approximate $F(x_t)$ and $F'(x_t)$, respectively.
Let $\set{x_t}$ be generated by Algorithm~\ref{alg:A1} for solving \eqref{eq:finite_sum}.
Assume that $b_t$ and $\hat{b}_t$ in \eqref{eq:msgd_estimators} are chosen such that $b_t := \min\sets{n, \bar{b}_t}$ and $\hat{b}_t := \min\sets{n, \hat{\bar{b}}_t}$ for $t\geq 0$, with
\begin{equation}\label{eq:bt_size2}
\hspace{-1ex}
\arraycolsep=0.05em
\left\{\begin{array}{lcl}
\bar{b}_0 &:= & \left\lfloor {\!\!} \frac{32M_{\phi}M^2\sigma_F\left(48\sigma_FM_{\phi}M^2 +  C_a\varepsilon^2\right)}{3C_a^2\varepsilon^4} \cdot \log\left(\frac{p+1}{\delta}\right) \right\rfloor \vspace{1ex}\\
\hat{\bar{b}}_0 &:= & \left\lfloor {\!\!} \frac{4M\sqrt{2M_{\phi}}\sigma_D\left(3M\sqrt{2M_{\phi}}\sigma_D + \sqrt{\beta_dC_a}\varepsilon\right)}{\beta_dC_a\varepsilon^2} \!\cdot\! \log\left(\frac{p+q}{\delta}\right) {\!\!}\right\rfloor,\vspace{1ex}\\
\bar{b}_t  &:=  & \left\lfloor  {\!\!} \frac{\big(6\sigma_F^2 + 2\sigma_F\sqrt{C_f}\norms{x_t - x_{t-1}}^2\big)}{3C_f^2\norms{x_t - x_{t-1}}^4} \cdot \log\left(\frac{p+1}{\delta}\right) \right\rfloor ~~(t\geq 1)\vspace{1ex}\\
\hat{\bar{b}}_t  &:= & \left\lfloor {\!\!} \frac{\left(6\sigma_D^2 + 2\sigma_D\sqrt{C_d}\norms{x_t - x_{t-1}}\right)}{3C_d\norms{x_t - x_{t-1}}^2}\cdot \log\left(\frac{p+q}{\delta}\right) \right\rfloor ~~(t\geq 1),
\end{array}\right.
\hspace{-1ex}
\end{equation}
for $\delta \in (0, 1)$, and $C_f$, $C_d$, and $C_a$ given in \textbf{Condition 2}.
Then, with probability at least $1-\delta$, the bound \eqref{eq:convergence_bound2} in Theorem \ref{th:convergence1} still holds.

Moreover, the total number $\Tc_f$ of stochastic function evaluations $\Fb(x_t,\xi)$ and the total number $\Tc_d$ of stochastic Jacobian evaluations $\Fb'(x_t,\zeta)$ to guarantee $\frac{1}{(T+1)}\sum_{t=0}^T\norms{\widetilde{G}_M(x_t)}^2 \leq \varepsilon^2$ do not exceed 
\begin{equation}\label{eq:complexity_case2}
\left\{\begin{array}{lcl}
\Tc_f  &:= & \BigO{\frac{\sigma_F^2\left[\Psi(x_0) - \Psi^{\star}\right]}{\varepsilon^6} \cdot \log\left(\frac{p+1}{\delta}\right)}, \vspace{1ex}\\
\Tc_d &:= &  \BigO{\frac{\sigma_D^2\left[\Psi(x_0) - \Psi^{\star}\right]}{\varepsilon^4}  \cdot \log\left(\frac{p+q}{\delta}\right)}.
\end{array}\right. 
\end{equation}
\end{theorem}

To the best of our knowledge, the oracle complexity bounds stated in Theorems~\ref{th:sgd_complexity1} and \ref{th:sgd_complexity2} are the first results for the stochastic Gauss-Newton methods described in Algorithm~\ref{alg:A1} under Assumptions~\ref{as:A1} and \ref{as:A2}.
Whereas there   exist several methods for solving \eqref{eq:nl_least_squares}, these algorithms are either not in the form of GN schemes as ours or rely on a different set of assumptions.
For instance, \citet{duchi2018stochastic,AdaGrad} considers a different model and uses stochastic subgradient methods, while \citet{zhang2019multi,zhang2019stochastic} directly applies a variance reduction gradient descent method and requires a stronger set of assumptions.

\beforesubsec
\subsection{SGN with SARAH Estimators}\label{subsec:SGN_with_SARAH}
\aftersubsec
Algorithm~\ref{alg:A1} with mini-batch stochastic estimators \eqref{eq:msgd_estimators} has high oracle complexity bounds when $\varepsilon$ is sufficiently small, especially for  function evaluations $\Fb(\cdot, \xi)$.
We attempt to reduce this complexity by exploiting a biased estimator called SARAH in  \citet{nguyen2017sarah} in this subsection.

More concretely, we approximate $F(x_t)$ and $F'(x_t)$ by using the following SARAH estimators, respectively:
\begin{equation}\label{eq:SARAH_estimators}
\hspace{-0.25ex}
\arraycolsep=0.1em
\left\{\begin{array}{ll}
\widetilde{F}_t := \widetilde{F}_{t\!-\!1} + \frac{1}{b_t}\sum_{\xi_j\in\Bc_t}\left(\Fb(x_t, \xi_j) \! - \! \Fb(x_{t\!-\!1}, \xi_j)\right), \vspace{1ex}\\
\widetilde{J}_t := \widetilde{J}_{t\!-\!1} + \frac{1}{\hat{b}_t}\sum_{\xi_i\in\hat{\Bc}_t}\left(\Fb'(x_t, \zeta_i) \! - \! \Fb'(x_{t\!-\!1}, \zeta_i)\right), 
\end{array}\right.
\hspace{-2ex}
\end{equation}
where the snapshots $ \widetilde{F}_0$ and $ \widetilde{J}_0$ are given, and $\Bc_t$ and $\hat{\Bc}_t$ are two mini-batches of size $b_t := \vert \Bc_t\vert$ and $\hat{b}_t := \vert\hat{\Bc}_t\vert$.

Using both the standard stochastic estimators \eqref{eq:msgd_estimators} and these SARAH estimators \eqref{eq:SARAH_estimators}, we modify Algorithm~\ref{alg:A1} to obtain the following double-loop variant  as in Algorithm~\ref{alg:A2}.

\begin{algorithm}[ht!]\caption{(SGN with SARAH estimators~\textbf{(SGN2)})}\label{alg:A2}
\normalsize
\begin{algorithmic}[1]
   \STATE{\bfseries Initialization:} Choose $\widetilde{x}^0 \in\R^p$ and $M > 0$.
   \vspace{0.5ex}  
   \STATE\hspace{0ex}\label{step:o1}{\bfseries For $s := 1,\cdots, S$ do}
   \vspace{0.5ex}   
   \STATE\hspace{2ex}\label{step:o2}  Generate mini-batches $\Bc_s$ (size $b_s$) and $\hat{\Bc}_s$  (size $\hat{b}_s$).
   \STATE\hspace{2ex}\label{step:o2}  Evaluate  $F_{0}^{(s)}$ and $J_0^{(s)}$ at $x_0^{(s)} := \widetilde{x}^{s-1}$ from \eqref{eq:msgd_estimators}.
   \vspace{0.65ex}   
   \STATE\hspace{2ex}\label{step:o3} Update $x_{1}^{(s)} := \widetilde{T}_M(x^{(s)}_0)$ based on \eqref{eq:surg_model}.
   \vspace{0.65ex}  
   \STATE\hspace{2ex}\label{step:o4} \textbf{Inner Loop: For $t := 1,\cdots, m$ do}
   \STATE\hspace{5ex}\label{step:i1} Generate mini-batches $\Bc_t^{(s)}$ and $\hat{\Bc}_t^{(s)}$.
   \STATE\hspace{5ex}\label{step:i2} Evaluate $F^{(s)}_t$ and $J_t^{(s)}$ from \eqref{eq:SARAH_estimators}.
   \STATE\hspace{5ex}\label{step:i3} Update $x_{t+1}^{(s)} := \widetilde{T}_M(x_t^{(s)})$ based on \eqref{eq:surg_model}.
   \vspace{0.65ex}  
   \STATE\hspace{2ex}\label{step:o5} \textbf{End of Inner Loop}
   \vspace{0.5ex}   
   \STATE\hspace{2ex}\label{step:o6} Set $\widetilde{x}^s := x_{m+1}^{(s)}$.
   \vspace{0.5ex}   
   \STATE\hspace{0ex}{\bfseries End For}
\end{algorithmic}
\end{algorithm}

In Algorithm~\ref{alg:A2}, every outer iteration $s$, we take a snapshot $\widetilde{x}^s$ using  \eqref{eq:msgd_estimators}.
Then, we run Algorithm~\ref{alg:A2} up to $m$ iterations in the inner loop $t$ but using SARAH estimators \eqref{eq:SARAH_estimators}.
Unlike \eqref{eq:msgd_estimators}, we are unable to exploit matrix-vector products for $\widetilde{J}_t$ in \eqref{eq:SARAH_estimators} due to its dependence on $\widetilde{J}_{t-1}$.

Let us first prove convergence and oracle complexity estimates in expectation of Algorithm~\ref{alg:A2} for solving \eqref{eq:nl_least_squares}.
However, we require an additional assumption for this case:

\begin{assumption}\label{as:A3}
$\Fb$ is $M_F$-average Lipschitz continuous, i.e. $\Exps{\xi}{\norms{\Fb(x,\xi) - \Fb(y,\xi)}^2} \leq M_F^2\norms{x - y}^2$ for all $x, y$.
\end{assumption}
Though Assumption~\ref{as:A3} is relatively strong, it has been used in several models, including neural network training under a bounded weight assumption.

Given a tolerance $\varepsilon > 0$ and $C > 0$, we first choose $M > 0$, $\delta_d > 0$, and two constants $\gamma_1 > 0$ and $\gamma_2 > 0$ such that
\begin{equation}\label{eq:rho_quantity}
\hspace{-0.25ex}
\arraycolsep=0.1em
\left\{\begin{array}{ll}
\theta_F &:= 2M - M_{\phi}(L_F + \delta_d) - \gamma_1M_F^2 - \gamma_2L_F^2 > 0, \vspace{1ex}\\
m  &:= \Big\lfloor \frac{8\left[\Psi(\widetilde{x}^{0}) - \Psi^{\star}\right]}{\theta_F C\varepsilon} \Big\rfloor.
\end{array}\right.
\hspace{-3ex}
\end{equation}
Next, we choose the mini-batch sizes of $\Bc_s$, $\hat{\Bc}_s$, $\Bc_t^{(s)}$, and $\hat{\Bc}_t^{(s)}$, respectively as follows:
\begin{equation}\label{eq:mini_batch}
\hspace{-0ex}
\arraycolsep=0.2em
\left\{ \begin{array}{llcllcl}
&b_s & :=  &\frac{2CM_{\phi}^2\sigma_F^2}{\theta_F^2\varepsilon^3} &\quad \hat{b}_s &:= &  \frac{4CM_{\phi}\sigma_D^2}{\theta_F \delta_D\varepsilon} \vspace{1ex}\\  
&b_t^{(s)} & := &  \frac{8M_{\phi}^2(m+1 - t)}{\theta_F \gamma_1\varepsilon^2} &\quad \hat{b}_t^{(s)} &:= & \frac{M_{\phi}(m+1-t)}{\gamma_2\delta_d}.
\end{array}\right.
\hspace{-1ex}
\end{equation}
Then, the following theorem states the convergence and oracle complexity bounds of Algorithm~\ref{alg:A2}.

\begin{theorem}\label{th:convergence_of_Sarah_GN}
Suppose that Assumptions~\ref{as:A1} and \ref{as:A2}, and \ref{as:A3} are satisfied for \eqref{eq:nl_least_squares}.
Let $\sets{x_t^{(s)}}_{t=0\to m}^{s=1\to S}$ be  generated by Algorithm~\ref{alg:A2} to solve \eqref{eq:nl_least_squares}.
Let $\theta_F$ and $m$ be chosen by \eqref{eq:rho_quantity}, and  the mini-batches $b_s$, $\hat{b}_s$, $b_t^{(s)}$, and $\hat{b}_t^{(s)}$ be set as in \eqref{eq:mini_batch}.
Assume that the output $\widehat{x}_T$ of Algorithm~\ref{alg:A2} is chosen uniformly at random in $\sets{x_t^{(s)}}_{t=0\to m}^{s=1\to S}$.
Then:

$\mathrm{(a)}$~The following bound holds
\begin{align}\label{eq:convergence2_bound1}
\frac{1}{S(m+1)}\displaystyle\sum_{s=1}^S\sum_{t=0}^m \Exp{\norms{\widetilde{G}_M(x_t)}^2}  \leq  \varepsilon^2.
\end{align}
$\mathrm{(b)}$~The total number of iterations $T$ to obtain $\Exp{\norms{\widetilde{G}_M(\widehat{x}_T)}^2} \leq \varepsilon^2$ is at most 
\begin{equation*}
T := S(m+1) =  \left\lfloor \frac{8M^2\left[\Psi(\widetilde{x}^0) - \Psi^{\star}\right]}{\theta_F\varepsilon^2} \right\rfloor = \BigO{\frac{1}{\varepsilon^{2}}}.
\end{equation*}
Moreover, the total stochastic oracle calls $\Tc_f$ and $\Tc_d$ for evaluating stochastic estimators of $\Fb(x_t,\xi)$ and its Jacobian $\Fb'(x_t,\zeta)$, respectively do not exceed:
\begin{equation}\label{eq:total_or_complexity}
\left\{\begin{array}{ll}
\Tc_f &:= \BigO{\frac{M_{\phi}^2\sigma_F^2}{\theta_F^2\varepsilon^4} + \frac{M^4M_{\phi}^2\left[\Psi(\widetilde{x}^{0}) - \Psi^{\star}\right]}{\theta_F^2\varepsilon^5}} \vspace{1ex}\\
\Tc_d &:= \BigO{\frac{M_{\phi}\sigma_D^2}{\theta_F\varepsilon^2}  + \frac{M^2M_{\phi} \left[\Psi(\widetilde{x}^{0}) - \Psi^{\star}\right]}{\theta_F \varepsilon^3}}.
\end{array}\right.
\vspace{-2ex}
\end{equation}
\end{theorem}

Finally, we show that $x_t$ computed by our methods is indeed an approximate stationary point of \eqref{eq:nl_least_squares} or \eqref{eq:finite_sum}.
\begin{corollary}\label{re:stationary_point}
If $x_t$ satisfies $\norms{\widetilde{G}_M(x_t)} \leq \varepsilon$ for given $\varepsilon > 0$, then under either \textbf{Condition~1} or \textbf{Condition~2}, and for any $y_t\in\partial{\phi}(F(x_{t+1}))$, we have $\mathcal{E}(x_{t+1}, y_t) \leq \BigO{\varepsilon}$. Consequently, $x_{t+1}$ is a $\BigO{\varepsilon}$-stationary point of \eqref{eq:nl_least_squares} or \eqref{eq:finite_sum}.
\end{corollary}

\begin{proof}
From Lemma~\ref{le:aprox_opt_cond}, we have
\begin{equation*}
\hspace{-1ex}
\arraycolsep=0.1em
\begin{array}{lcl}
\mathcal{E}(x_{t+1}, y_t)  &\leq& \left(1+ \frac{M_{\phi}L_F}{M}\right)\norms{\widetilde{G}_M(x_t)} + \frac{1}{2}\norms{\widetilde{J}_t - F'(x_t)}^2\vspace{1ex}\\
&& + {~} \frac{(1+L_F)}{2M^2}\norms{\widetilde{G}_M(x_t)}^2 + M_{\phi}\norms{\widetilde{F}_t - F(x_t)}.
\end{array}
\hspace{-1ex}
\end{equation*}
Under either \textbf{Condition~1} or \textbf{Condition~2}, we  have $\norms{\widetilde{F}_t - F(x_t)}\leq\BigO{\varepsilon^2}$ and  $\norms{\widetilde{J}_t - F'(x_t)}^2 \leq\BigO{\varepsilon^2}$.
Hence, if $\norms{\widetilde{G}_M(x_t)} \leq \varepsilon$, then using these three bounds into the last estimate, one can show that $\mathcal{E}(x_{t+1}, y_t) \leq \BigO{\varepsilon}$.
Consequently, $x_{t+1}$ is a $\BigO{\varepsilon}$-stationary point of \eqref{eq:nl_least_squares} or \eqref{eq:finite_sum}.
\end{proof}

\begin{remark}[\textbf{Algorithm~\ref{alg:A2} without Assumption~\ref{as:A3}}]\label{re:convergnece2}
We claim that Algorithm~\ref{alg:A2} still converges without Assumption~\ref{as:A3}.
However, its oracle complexity remains $\BigO{\sigma_F^2\varepsilon^{-6}}$ for $F$ and $\BigO{\sigma_D^2\varepsilon^{-4}}$ for $F'$ as in Algorithm~\ref{alg:A1}.
We therefore omit the proof of this statement.
\end{remark}

Another main step of both Algorithms~\ref{alg:A1} and \ref{alg:A2} is to compute $\widetilde{T}_M(x_t)$.
We will provide different routines in Sup. Doc.~\ref{apdx:sec:subsolver} to efficiently compute $\widetilde{T}_M(x_t)$.

\beforesubsec
\subsection{Extension to The Regularization Setting \eqref{eq:composite_form}}\label{subsec:composite}
\aftersubsec
It is straight forward to extend our methods to handle a regularizer $g$ as in \eqref{eq:composite_form}.
If $g$ is nonsmooth and convex, then we can modify \eqref{eq:surg_model} as follows:
\begin{align}\label{eq:surg_model2} 
\hspace{-0.5ex}
\widetilde{T}_M(x_t) &{\!} := {\!} \argmin_{z \in\R^p}\Big\{\widetilde{\Qc}_M(z; x_t) \!:=\!  \phi(\widetilde{F}(x_t) \!+\! \widetilde{J}(x_t)(z \!-\! x_t)) \nonumber\\
&\hspace{12ex} + {~} g(z) +  \tfrac{M}{2}\norms{z - x_t}^2\Big\}.
\hspace{-0.25ex}
\end{align}
Then, we obtain variants of Algorithms~\ref{alg:A1} and \ref{alg:A2} for solving \eqref{eq:composite_form}, where our theoretical guarantees in this paper remain preserved.
This subproblem can efficiently be solved by primal-dual methods as presented in Supp. Doc. \ref{apdx:sec:subsolver}.
If $g$ is $L_g$-smooth, then we can replace $g$ in \eqref{eq:surg_model} by its quadratic surrogate $g(x_t) + \iprods{\nabla{g}(x_t), z - x_t} + \frac{L_g}{2}\norms{z - x_t}^2$.

\beforesec
\section{Numerical Experiments}\label{sec:num_exp}
\aftersec
We conduct two numerical experiments to evaluate the performance of Algorithm~\ref{alg:A1} (SGN) and Algorithm~\ref{alg:A2} (SGN2).
Further details of our experiments are in Supp. Doc. \ref{sec:apdx:add_experiments}.

\beforesubsec
\subsection{Stochastic Nonlinear Equations}\label{subsec:parameter_estimation}
\aftersubsec
We consider a nonlinear equation: $\Exps{\xi}{\Fb(x,\xi)} = 0$ as the expectation of a stochastic function $\Fb : \R^p\times\Omega\to\R^q$.
This equation can be viewed as a natural extension of nonlinear equations from a deterministic setting to a stochastic setting, including stochastic dynamical systems and PDEs.
It can also present as the first-order optimality condition $\Exps{\xi}{\nabla{\mathbf{G}}(x,\xi)} = 0$ of a stochastic optimization problem $\min_x\Exps{\xi}{\mathbf{G}(x,\xi)}$.
Moreover, it can be considered as a special case of stochastic variational inequality in the literature, see, e.g., \citet{rockafellar2017stochastic}.

Instead of directly solving $\Exps{\xi}{\Fb(x,\xi)} = 0$, we  can formulate it into the following minimization problem:
\begin{equation}\label{eq:exp_1}
\min_{x\in\R^p}\Big\{ \Psi(x) := \norm{\Exps{\xi}{\Fb(x,\xi)}} \Big\},
\end{equation}
where $\Fb(x,\xi) := (\Fb_1(x,\xi),\Fb_2(x,\xi),\cdots, \Fb_q(x,\xi))^{\top}$ such that $F_j : \R^p \to \R$ is the expectation of $\Fb_j(\cdot,\xi)$, i.e., $F_j(x) := \Exps{\xi}{\Fb_j(x,\xi)}$ for $j = 1,\cdots, q$, and $\norm{\cdot}$ is a given norm (e.g., $\ell_2$-norm or $\ell_1$-norm). 

Assume that we take average approximation of $\Exps{\xi}{\Fb(x,\xi)}$ to obtain a finite sum $F(x) = \frac{1}{n}\sum_{i=1}^nF(x, \xi_i)$ for sufficiently large $n$.
In the following experiments, we choose $q = 4$, and for $i=1,\cdots, n$, we choose $\Fb_j(x,\xi_i)$ as
\begin{equation*}
\arraycolsep=0.1em
\left\{\begin{array}{lcl}
\Fb_1(x,\xi_i) & := &  (1 - \tanh(y_i(a_i^{\top}x + b_i)),  \vspace{0.5ex}\\
\Fb_2(x,\xi_i) & := &  \left(1 - (1 + \exp(- y_i(a_i^{\top}x + b_i)))^{-1}\right)^2, \vspace{0.5ex}\\
\Fb_3(x,\xi_i) & := & \log(1 + \exp({-y_i(a_i^{\top}x + b_i)})) \vspace{0.5ex}\\
&& - {~} \log(1 + \exp({-y_i(a_i^{\top}x + b_i)-1})), \vspace{0.5ex}\\
\Fb_4(x,\xi_i) & := & \log(1 +(y_i(a_i^{\top}x + b_i) - 1)^2),
\end{array}\right.
\end{equation*}
where  $a_i$ is the $i$-th row  of an input matrix $A \in\R^{n\times p}$, and $y \in \set{-1,1}^n$, $b\in\R^n$ are two input vectors, and $\xi_i := (a_i, b_i, y_i)$.
These functions were used in binary classification involving nonconvex losses, e.g., \citet{zhao2010convex}.
Since they are nonnegative, if we use the $\ell_1$-norm, then \eqref{eq:exp_1} can be viewed  as a model average of $4$ different losses in binary classification (see Supp. Doc. \ref{sec:apdx:add_experiments}).

We implement both Algorithms~\ref{alg:A1} (SGN) and \ref{alg:A2} (SGN2) to solve \eqref{eq:exp_1}.
We also compare them with the baseline using the full samples instead of calculating $\widetilde{J}$ and $\widetilde{F}$ as in \eqref{eq:msgd_estimators} and \eqref{eq:SARAH_estimators}.
We call it the deterministic GN scheme (GN).

\noindent\textbf{Experiment setup.}
We test three algorithms on four standard datasets: \texttt{w8a}, \texttt{ijcnn1}, \texttt{covtype}, and \texttt{url\_combined} from LIBSVM\footnote{Available online at \href{https://www.csie.ntu.edu.tw/~cjlin/libsvm/}{https://www.csie.ntu.edu.tw/{$\sim$}cjlin/libsvm/}}. 
Further information about these dataset is described in Supp. Doc. \ref{sec:apdx:add_experiments}.

\begin{figure}[hpt!]
\begin{center}
    \includegraphics[width = 0.43\textwidth]{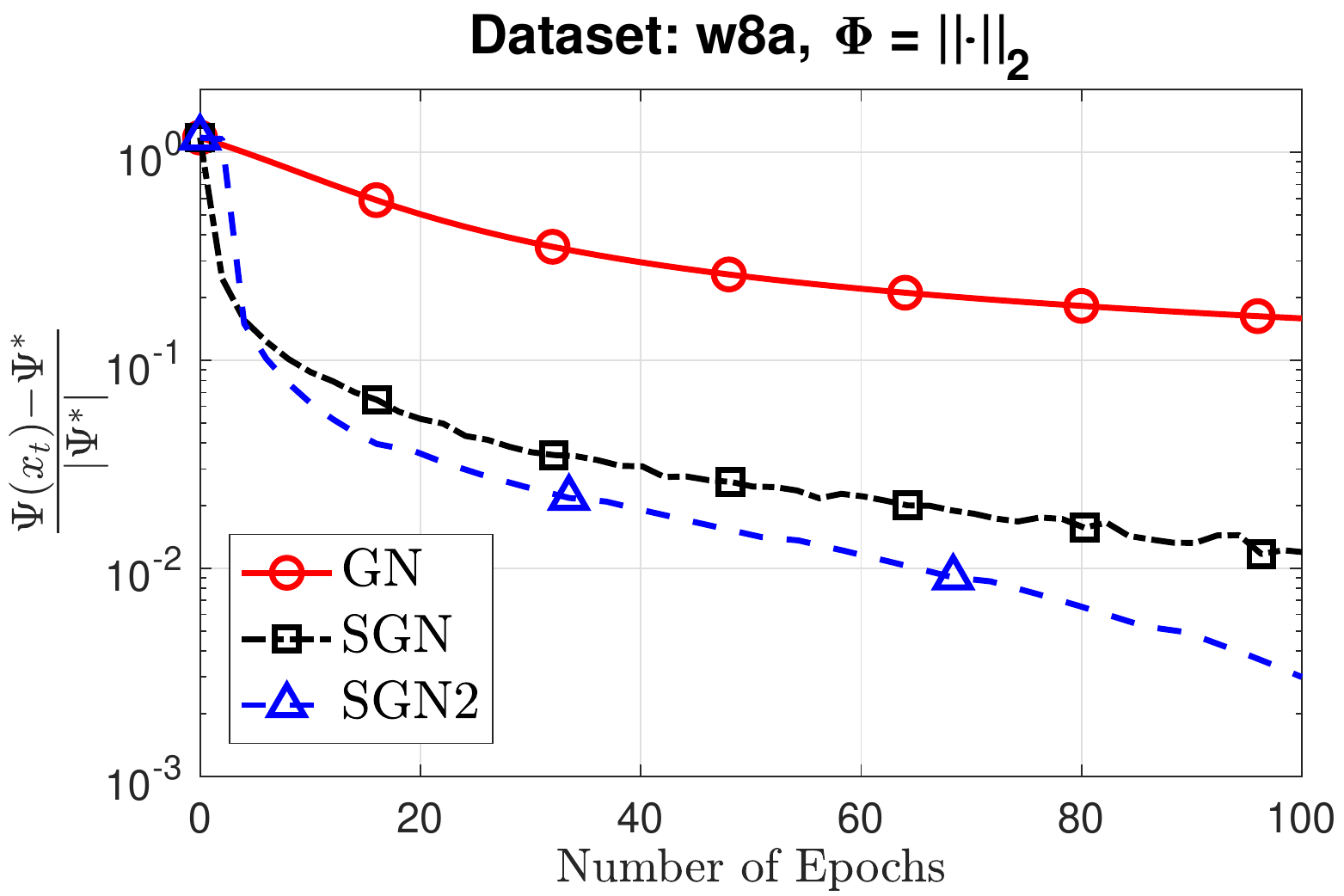}
    \includegraphics[width = 0.43\textwidth]{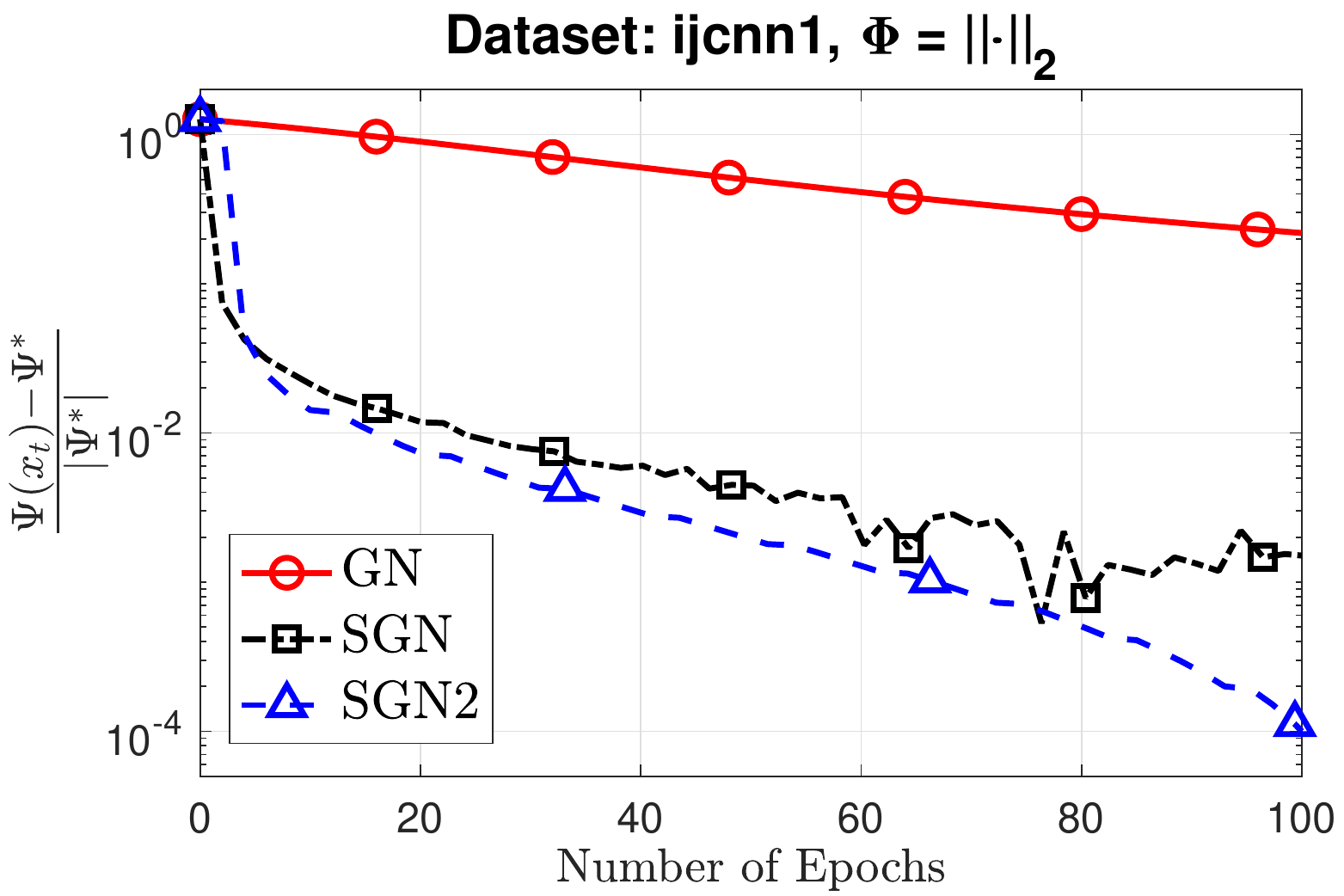}
    \vspace{-1ex}
    \caption{The performance of $3$ algorithms on the \texttt{w8a} and \texttt{ijcnn1}.}\label{fig:exp_1}
\end{center}
\vspace{-3ex}
\end{figure}

\begin{figure}[ht!]
\begin{center}
    \includegraphics[width = 0.43\textwidth]{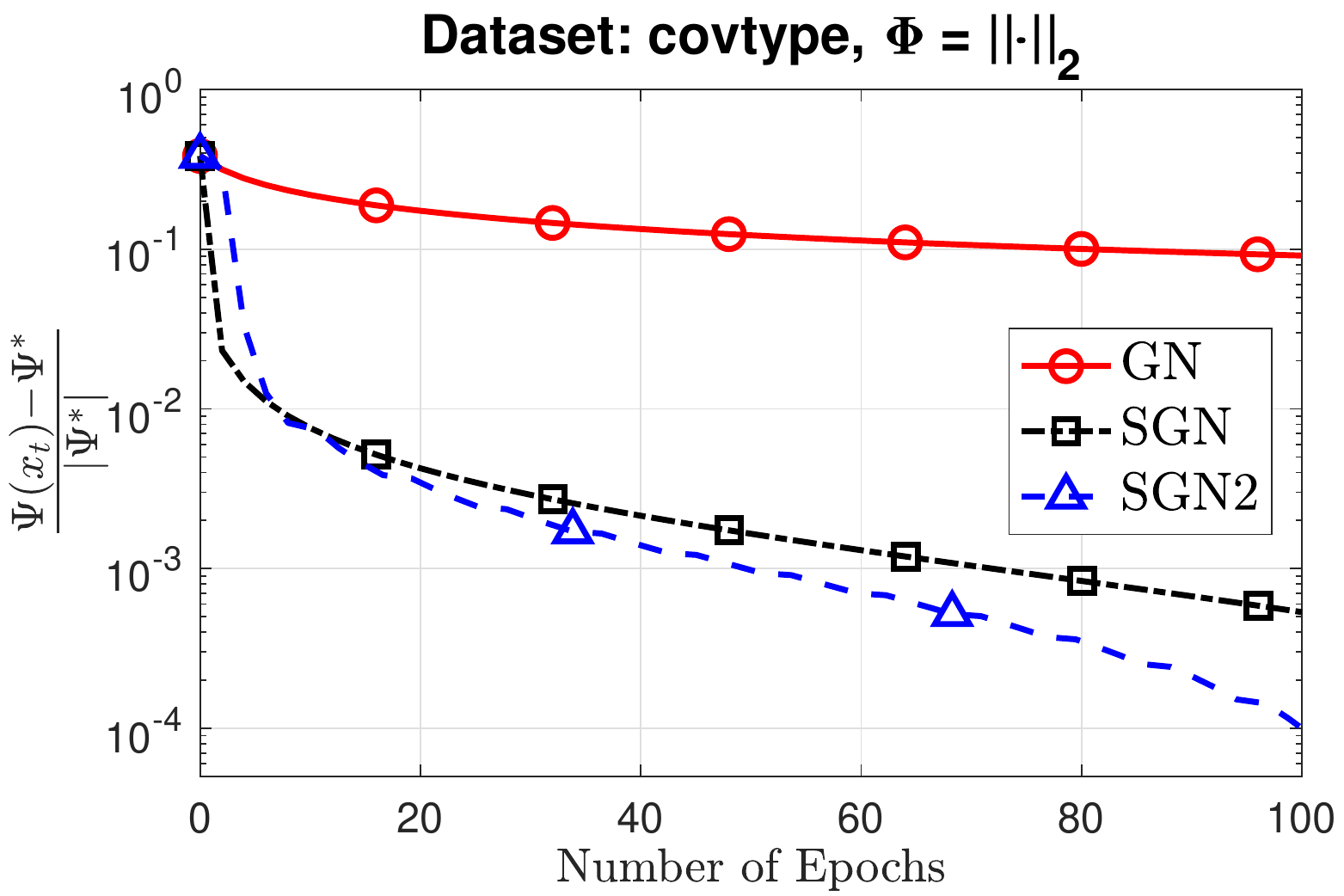}
    \includegraphics[width = 0.43\textwidth]{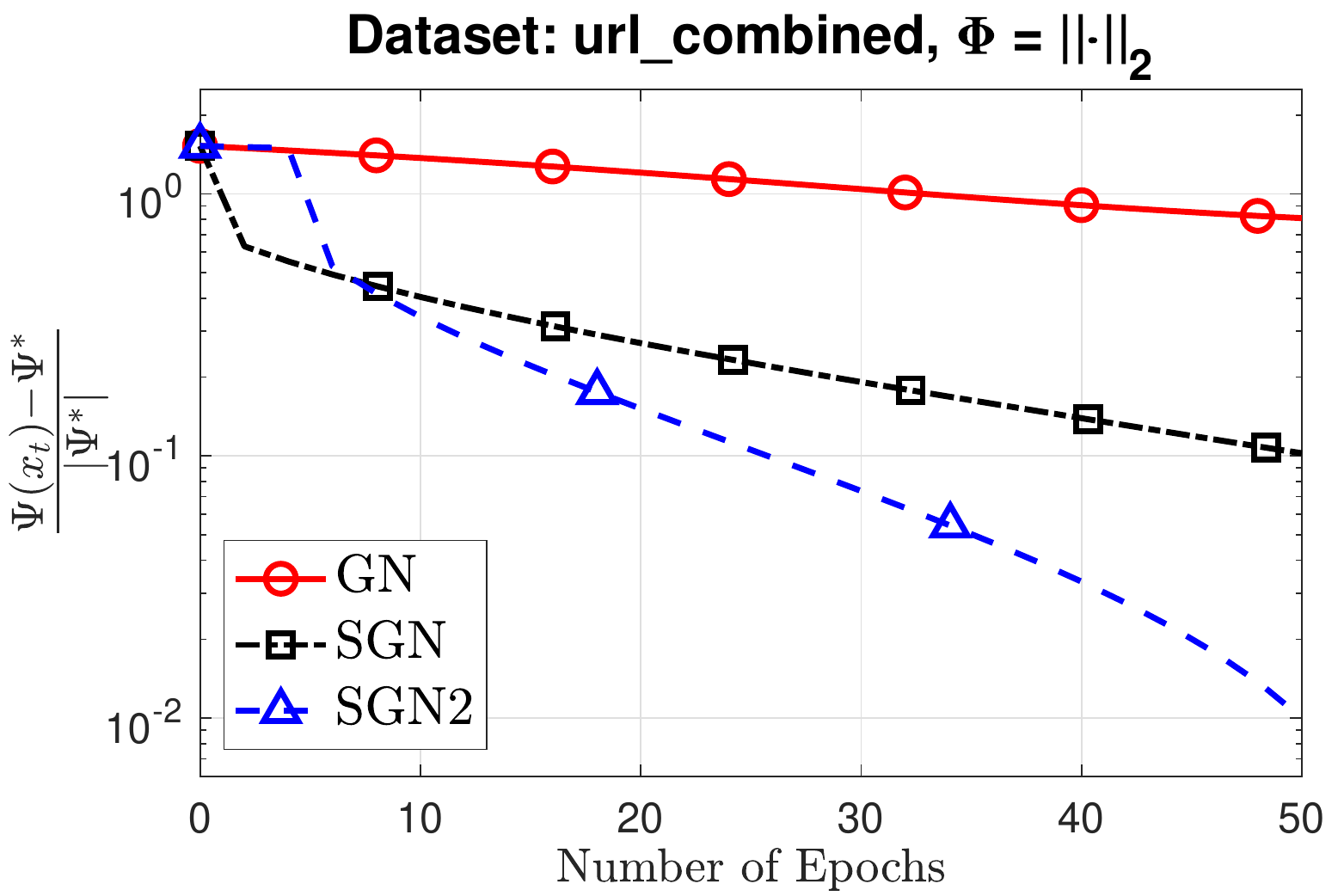}
    \vspace{-1ex}
    \caption{The performance of $3$ algorithms on \texttt{covtype} and \texttt{url\_combined}.}\label{fig:exp_2}
\end{center}
\vspace{-4ex}
\end{figure}

To find appropriate batch sizes for $\widetilde{J}$ and $\widetilde{F}$, we perform a grid search over different combinations of mini-batch sizes to select the best ones. More information about this process can be found in Supp. Doc. \ref{sec:apdx:add_experiments}.

We evaluate these algorithms on instances of \eqref{eq:exp_1} using $\phi(\cdot) = \norm{\cdot}_2$. We use $M:=1$ and $\rho:=1$ for all datasets.
The performance of three algorithms is shown in Figure~\ref{fig:exp_1} for the \texttt{w8a} and \texttt{ijcnn1} datasets. 
This figure depicts the relative objective residuals $\frac{\Psi(x_t) - \Psi^{\star}}{\vert \Psi^{\star}\vert}$ over the number of epochs, where $\Psi^{\star}$ is the lowest objective value obtained when running three algorithms until the relative residuals falls below $10^{-6}$.
In both cases, SGN2 works best while SGN is still much better than the baseline GN in terms of sample efficiency.

For \texttt{covtype} and \texttt{url\_combined} datasets, we obverse similar behavior as shown in Figure~\ref{fig:exp_2}, where SGN2 is more efficient than SGN, and both SGN schemes outperform GN. This experiment shows that both SGN algorithms are indeed much more sample efficient than the baseline GN algorithm. 

In order to compare with existing algorithms, we use a smooth objective function in \eqref{eq:exp_1} with a Huber loss, $\phi(u) = \frac{1}{2} u^2~\text{for}~\vert u \vert \leq \delta$ and $\phi(u) = \delta(\vert u\vert - \frac{1}{2}\delta)$ otherwise, and $\delta := 1.0$. We implement the nested SPIDER method in \citet[Algorithm 3]{zhang2019multi}, denoted as N-SPIDER, and the stochastic compositional gradient descent in \citet[Algorithm 1]{wang2017stochastic}, denoted as SCGD.

We run 5 algorithms: GN, SGN, SGN2, N-SPIDER, and SCGD on 4 datasets as in the previous test. 
We choose $M:= 1$ and $\rho := 1$ for all datasets. 
We tune the learning rate for both N-SPIDER and SCGD and finally obtain $\eta := 1.0$ for both algorithms. 
We also set $\varepsilon=10^{-1}$ for N-SPIDER, see \citet[Algorithm 3]{zhang2019multi}. In addition, we conduct similar grid search as before to choose the suitable parameters for these algorithms. 
The chosen parameters are presented in Supp. Doc. \ref{sec:apdx:add_experiments}. The results on these datasets are depicted in Figure~\ref{fig:exp_1_smooth1} and Figure~\ref{fig:exp_1_smooth2}.

\begin{figure}[hpt!]
\vspace{-1ex}
\begin{center}
    \includegraphics[width = 0.43\textwidth]{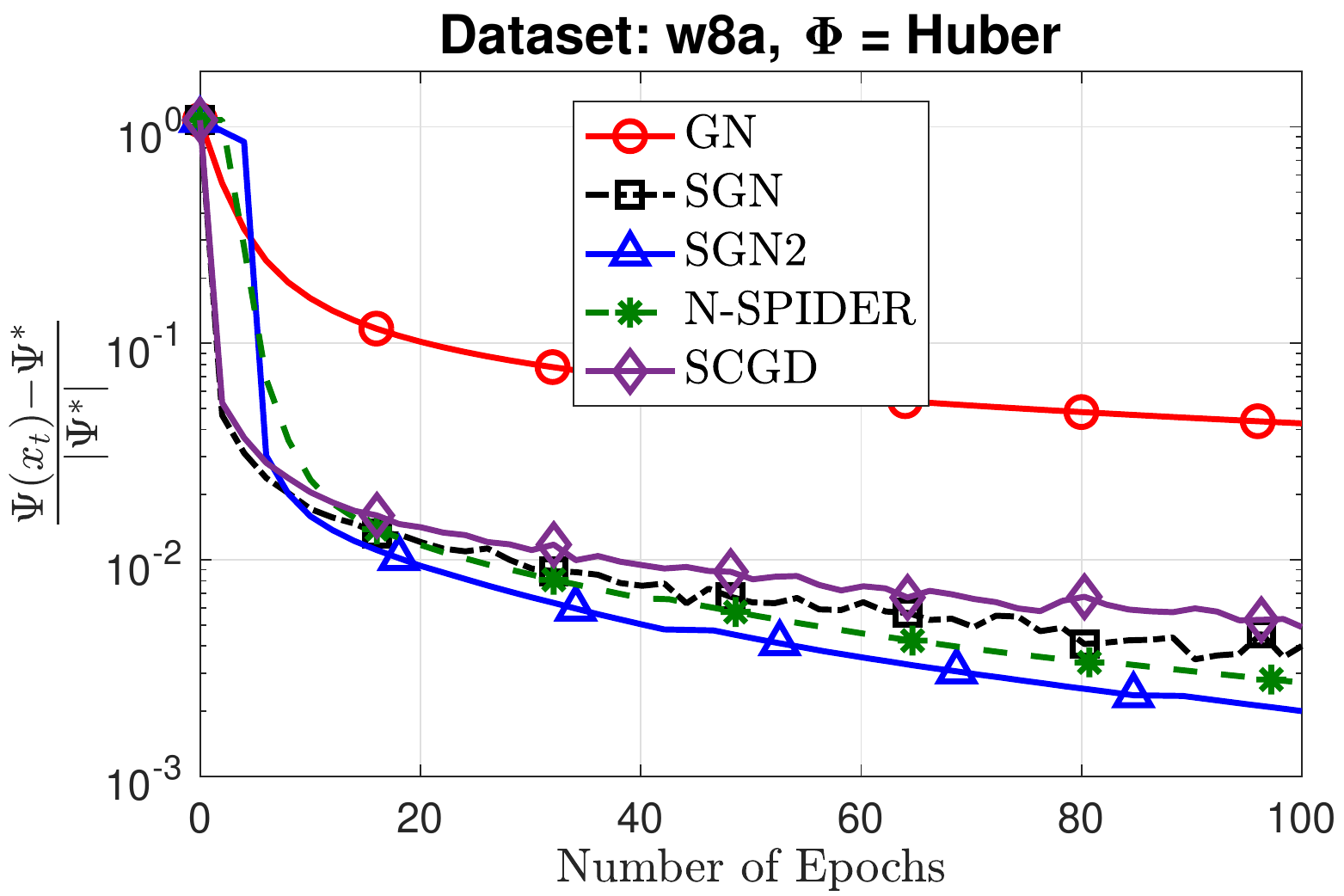}
    \includegraphics[width = 0.43\textwidth]{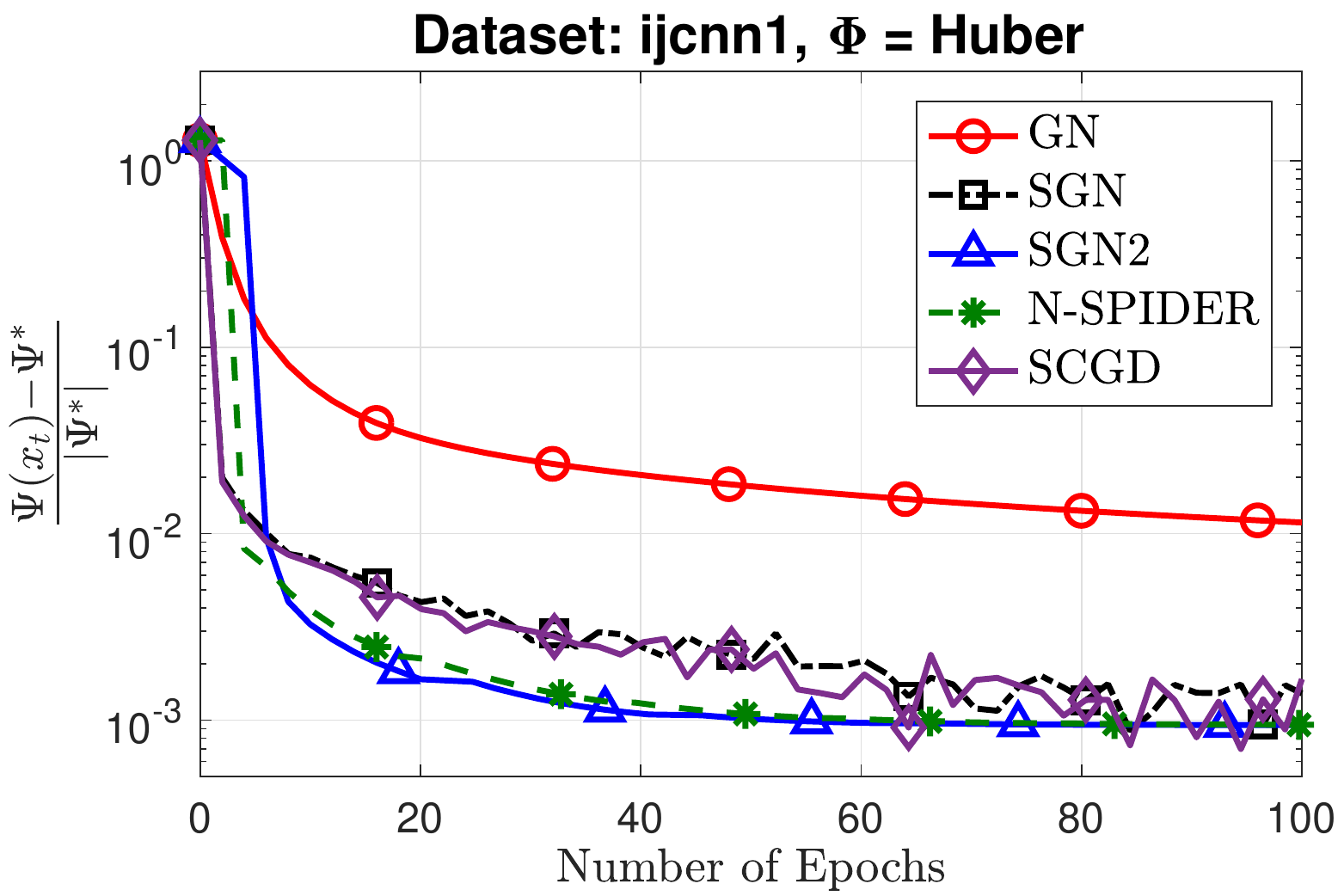}
    \vspace{-1ex}
    \caption{The performance of $5$ algorithms on \texttt{w8a} and \texttt{ijcnn1} datasets.}\label{fig:exp_1_smooth1}
\end{center}
\vspace{-2ex}
\end{figure}

\begin{figure}[hpt!]
\vspace{0ex}
\begin{center}
    \includegraphics[width = 0.43\textwidth]{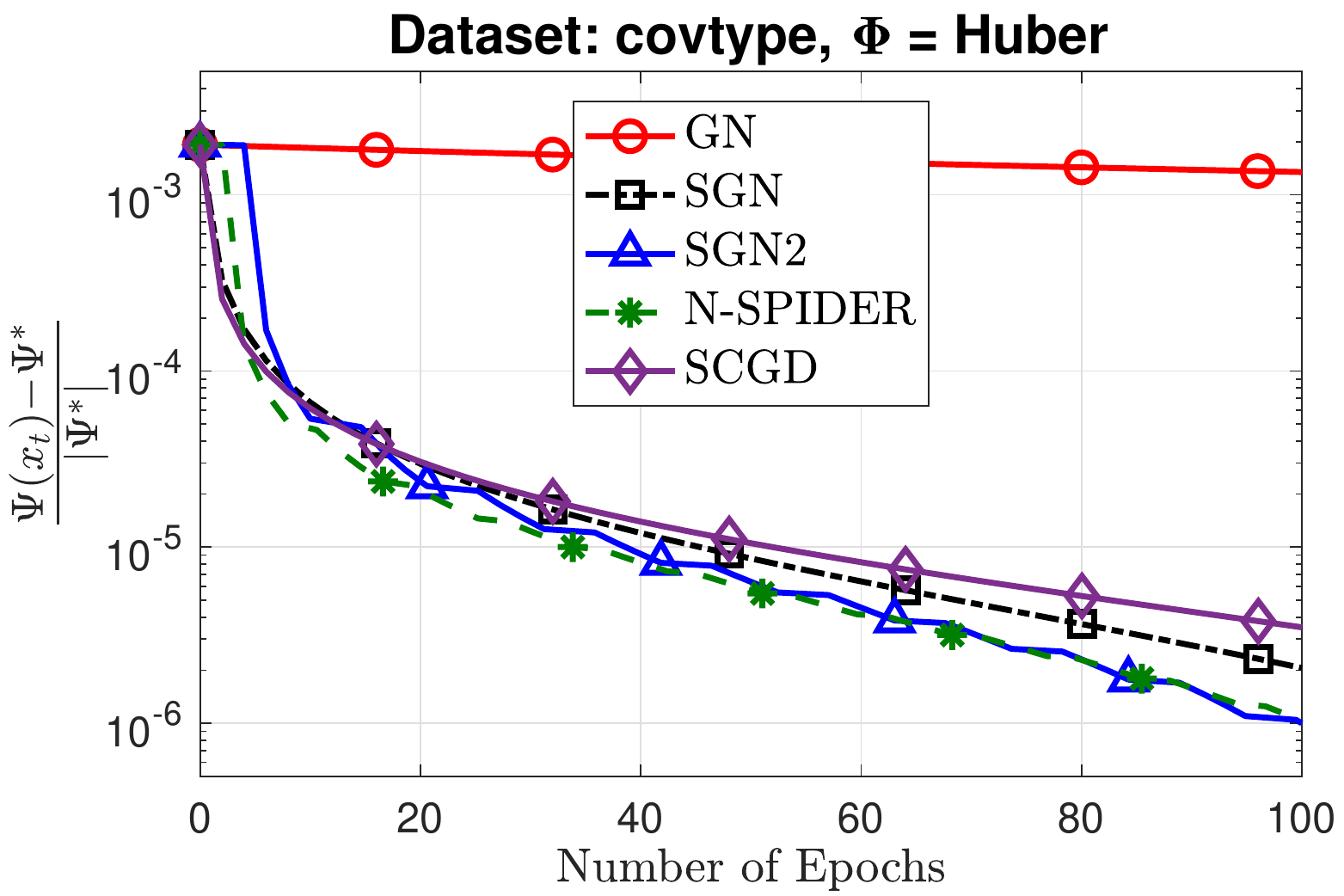}
    \includegraphics[width = 0.43\textwidth]{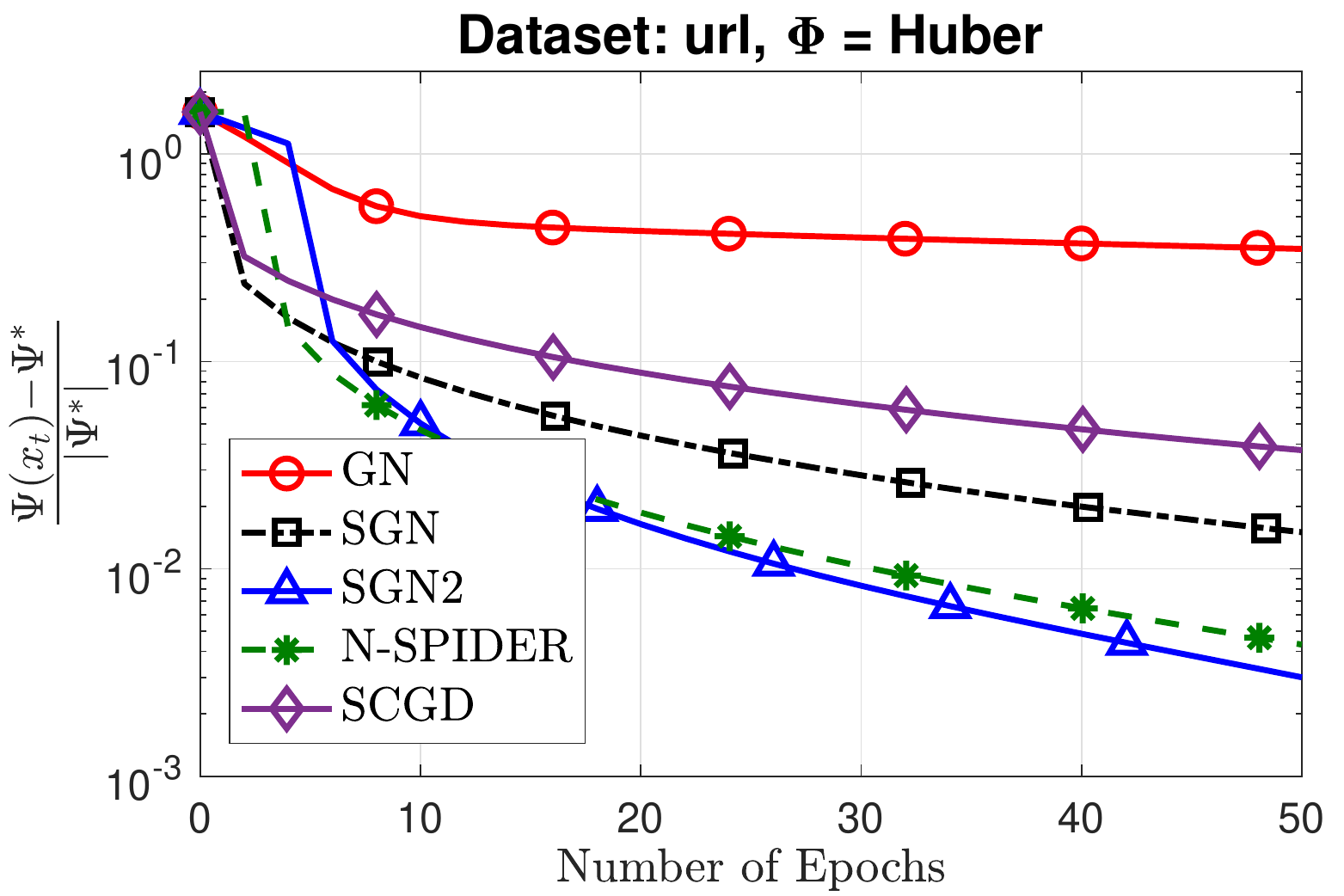}
    \vspace{-1ex}
    \caption{The performance of $5$ algorithms on \texttt{covtype} and \texttt{url\_combined} datasets.}\label{fig:exp_1_smooth2}
\end{center}
\vspace{-3ex}
\end{figure}

From both figures, SGN2 seems to perform best in all datasets. N-SPIDER is better than SGN and comparable with SGN in \texttt{ijcnn1} and \texttt{covtye}. SGN is comparable with SCGD in \texttt{ijcnn1} dataset while having better performance in the remaining ones. GN still perform poorly in these cases since it use full samples to compute $\tilde{F}$ and $\tilde{J}$.

\beforesubsec
\subsection{Optimization Involving Expectation Constraints}
\aftersubsec
We consider the following optimization problem:
\myeq{eq:min_with_e_constr}{
\min_{x\in\R^p}\Big\{ g(x) ~~\mathrm{s.t.}~~\Exps{\xi}{\Fb(x, \xi)} \leq 0 \Big\},
}
where $g : \R^p\to\Rext$ is a convex function, possibly nonsmooth, and $\Fb : \R^p\times\Omega\to\R^q$ is a smooth stochastic function.
This problem has various applications such as optimization with conditional value at risk (CVaR) constraints and metric learning \cite{lan2016algorithms} among others.
Let us consider an exact penalty formulation of \eqref{eq:min_with_e_constr} as
\begin{equation}\label{eq:penalty_form}
\min_{x\in\R^p}\Big\{ \Psi(x) := g(x) + \phi(\Exps{\xi}{\Fb(x, \xi)}) \Big\},
\end{equation}
where $\phi(u) := \rho\sum_{i=1}^q[u_i]_{+} $ with  $[u]_{+} :=\max\set{0, u}$ and $\rho > 0$ is a given penalty parameter.
Clearly, \eqref{eq:penalty_form} coincides with \eqref{eq:composite_form}, an extension of \eqref{eq:nl_least_squares}.

We evaluate $3$ algorithms on the asset allocation problem \cite{rockafellar2000optimization} as an instance of \eqref{eq:min_with_e_constr}:
\myeq{eq:cvar}{
\hspace{-1.5ex}
\arraycolsep=0.01em
\left\{\begin{array}{ll}
\displaystyle\min_{\tau \in [\underline{\tau}, \bar{\tau}], z\in\R^p}  &\hspace{-1ex} -c^{\top}z + \phi\left( \tau + \frac{1}{\beta n}\sum_{i=1}^n[-\xi_i^{\top}z - \tau]_{+}\right) \vspace{1ex}\\
\mathrm{s.t}~&z \in \Delta_p := \set{ \hat{z} \in \R^p_{+} \mid \sum_{i=1}^p\hat{z}_i = 1}.
\end{array}\right.
\hspace{-2ex}
}
To apply our methods, we need to smooth $[u]_{+}$ by $\frac{1}{2}\big[ u + (u^2 + \gamma^2)^{1/2} - \gamma \big]$ for a sufficiently small value $\gamma > 0$.
If we introduce $x := (z, \tau)$, $F(x,\xi_i) := \tau +  \frac{1}{2\beta}\left( [ (\xi_i^{\top}z + \tau)^2+\gamma^2]^{1/2} - \xi_i^{\top}z - \tau - \gamma\right)$ for $i=1,\cdots, n$, and $g(x) = -c^{\top}z + \delta_{\Delta_p}(x)$, then we can reformulate the smoothed approximation of \eqref{eq:cvar} into \eqref{eq:composite_form}, where $\delta_{\Delta_p\times[\underline{\tau},\bar{\tau}]}$ is the indicator of $\Delta_p\times[\underline{\tau},\bar{\tau}]$.
Note that $F'(\cdot,\xi_i)$ is Lipschitz continuous with the Lipschitz constant $L_{i} := \frac{\norms{\xi_i}^2}{2\beta\gamma}$.
In our experiments, we choose $[\underline{\tau},\bar{\tau}]$ to be $[0, 1]$, $\beta := 0.1$, and  $\gamma := 10^{-3}$.
We were experimenting different $\rho$ and $M$, and eventually set $\rho := 5$ and $M := 5$.

We test three algorithms: GN, SGN, and SGN2 on both synthetic and real datasets. 
We follow the procedures from \citet{lan2012validation} to generate synthetic data with $n = 10^5$ and $p \in \set{300, 500, 700}$. 
We also obtain real datasets of US stock prices for $889$, $865$, and $500$ types of stocks described, e.g., in \citet{SunTran2017gsc} then bootstrap them to obtain different datasets of sizes $n = 10^5$.
The details and additional results are given in Supp. Doc. \ref{sec:apdx:add_experiments}.

\begin{figure}[hpt!]
\vspace{-1ex}
\begin{center}
    \includegraphics[width = 0.43\textwidth]{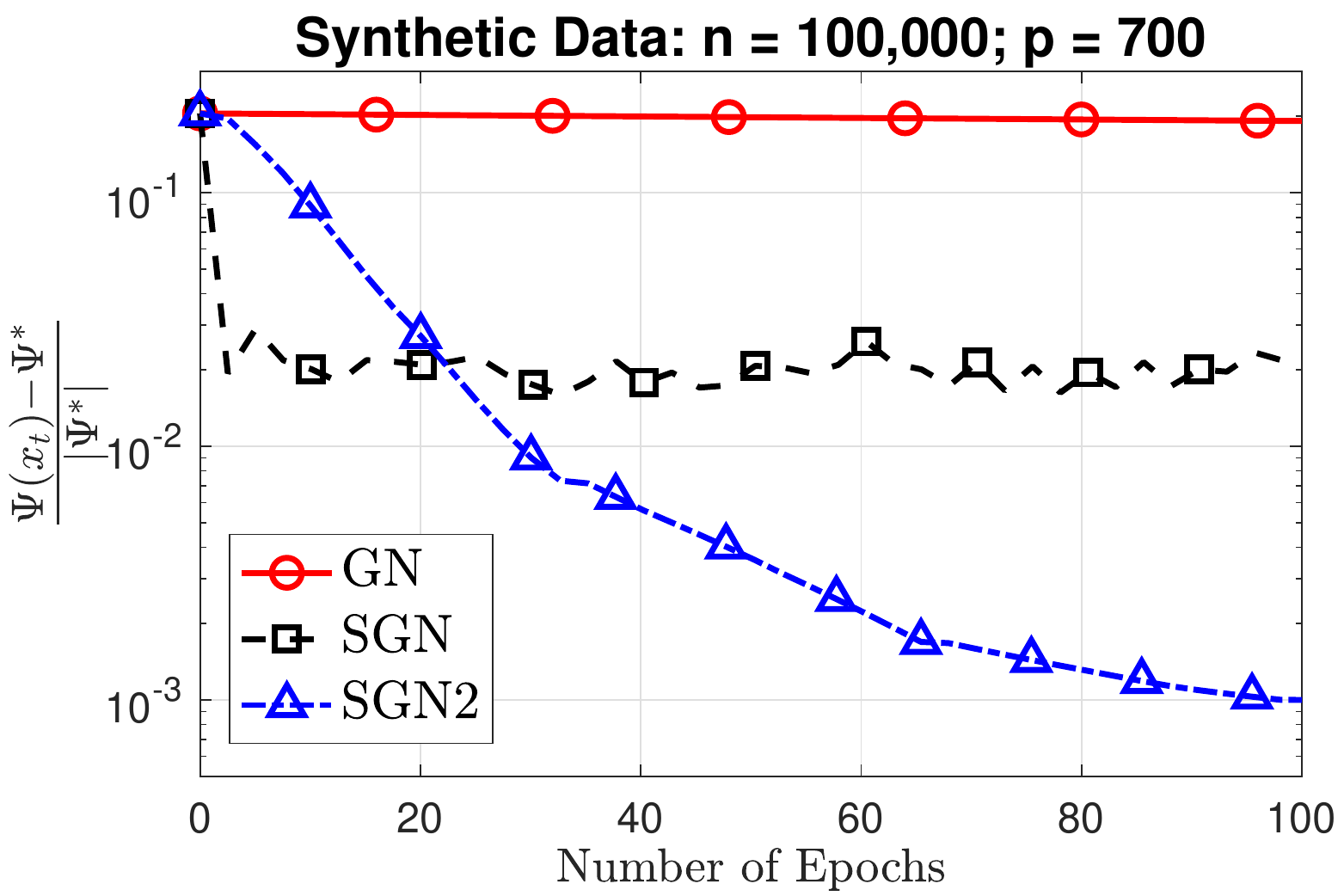}
    \includegraphics[width = 0.43\textwidth]{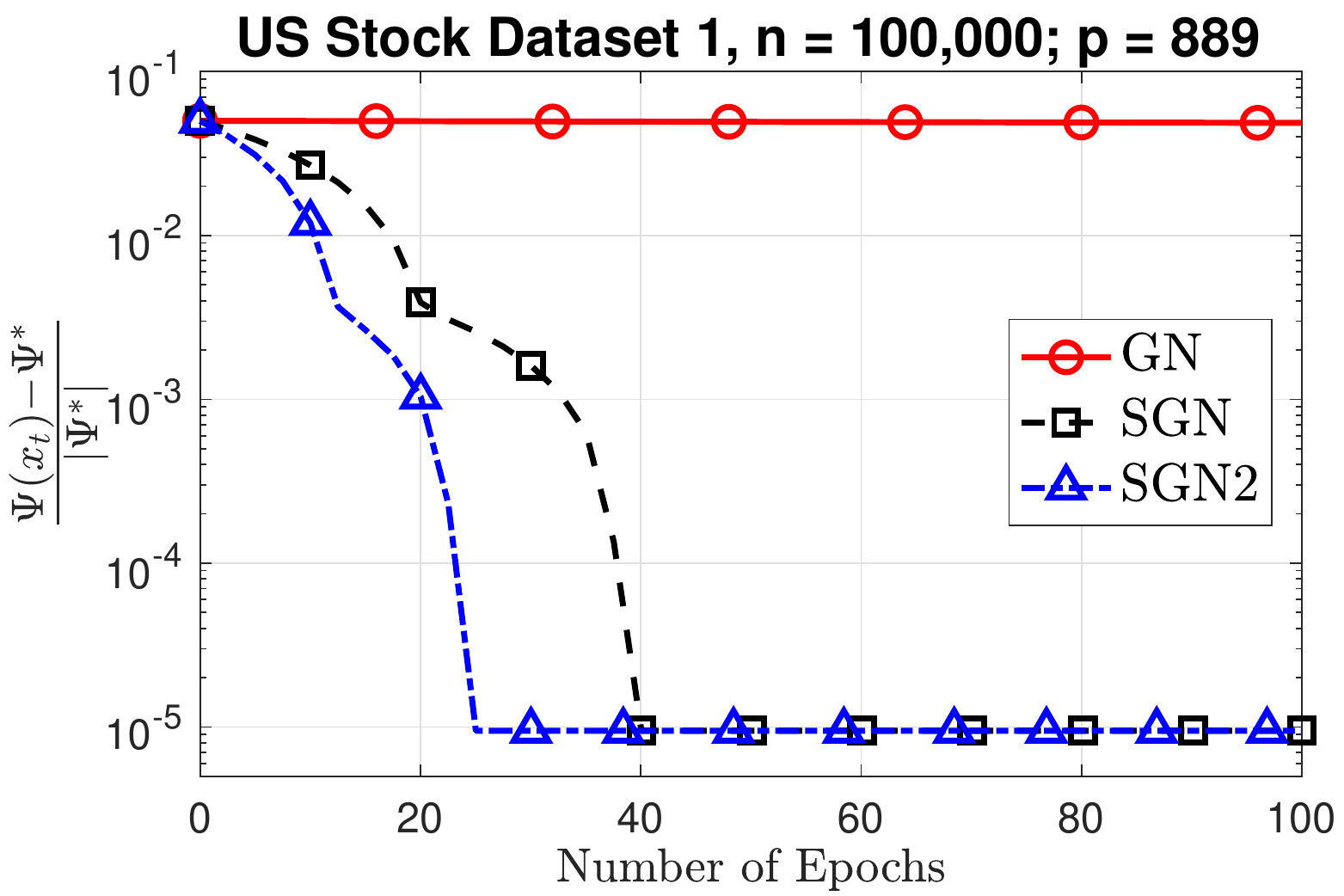}
    \vspace{-1ex}
    \caption{The performance of 3 algorithms on two datasets.}\label{fig:exp_4}
\end{center}
\vspace{-2ex}
\end{figure}

The performance of three algorithms on these datasets is depicted in Figure~\ref{fig:exp_4}. 
SGN is still much better than GN in both experiments while SGN2 is the best among three. 
With the large amount of samples per iteration, GN performs poorly in these experiments.

Numerical results have confirmed the advantages of SGN and SGN2 which well align with our theoretical analysis.

\section*{Acknowledgements}
The work of Q. Tran-Dinh has partially been supported by the National Science Foundation (NSF), award No. DMS-1619884  and the Office of Naval Research (ONR), grant No. N00014-20-1-2088 (2020--2023).
The authors are thankful to Deyi Liu for providing some parts of Python codes used in the experiment section.

\cleardoublepage
\bibliographystyle{icml2020}

\cleardoublepage
\onecolumn
\appendix
\begin{center}
\textsc{\large Supplementary Document}

\textbf{\Large Stochastic Gauss-Newton Algorithms for Nonconvex Compositional Optimization}
\end{center}
\vspace{2ex}
\beforesec
\section{The Proof of Technical Results in Section~\ref{sec:math_tools}: Mathematical Tools}\label{apdx:sec:math_tools}
\aftersec
This section provides the full proof of  technical results in Section~\ref{sec:math_tools}.
Let us first recall the bound \eqref{eq:key_est1}.
The proof of this bound can be found, e.g., in \citet{Nesterov2007g}.
However, for completeness, we prove it here.

\begin{proof}[\textbf{The proof of \eqref{eq:key_est1}}]
Since $F'$ is $L_F$-Lipschitz continuous with a Lipschitz constant $L_F$, we have $\norms{F(y) - F(x) - F'(x)(y-x)} \leq \frac{L_F}{2}\norms{y-x}^2$ for any $x, y\in\R^p$.
On the other hand, since $\phi$ is $M_{\phi}$-Lipschitz continuous, we have $\phi(u) \leq \phi(v) + M_{\phi}\norms{u-v}$ for any $u, v\in\R^q$.
Hence, we have
\begin{equation*}
\begin{array}{lcl}
\phi(F(y)) & \leq & \phi(F(x) + F'(x)(y - x)) + M_{\phi}\norms{F(y) - F(x) - F'(x)(y-x)} \vspace{1ex}\\
&  \leq & \phi(F(x) + F'(x)(y - x)) + \frac{M_{\phi}L_F}{2}\norms{y-x}^2,
\end{array}
\end{equation*}
which proves \eqref{eq:key_est1}.
\end{proof}

\beforesubsec
\subsection{The Proof of Lemma~\ref{le:aprox_opt_cond}: Approximate Optimality Condition}\label{apdx:le:aprox_opt_cond}
\aftersubsec
\textbf{Lemma}.~\ref{le:aprox_opt_cond}.
\textit{
Suppose that Assumption~\ref{as:A1} holds.
Let $\widetilde{T}_M(x)$ be computed by \eqref{eq:surg_model} and $\widetilde{G}_M(x)$ be defined by \eqref{eq:grad_map1}.
Then, $\mathcal{E}(\widetilde{T}_M(x), y)$ of \eqref{eq:nl_least_squares} or \eqref{eq:finite_sum} defined by \eqref{eq:residual} with $y\in \partial{\phi}(F(\widetilde{T}_M(x)))$ is bounded by
\begin{equation*} 
\begin{array}{lcl}
\mathcal{E}(\widetilde{T}_M(x), y)  &:= & \dist{0, -F(\widetilde{T}_M(x))  + \partial{\phi^{*}}(y)}  +  \norms{ F'(\widetilde{T}_M(x))^{\top}y }  \vspace{1ex}\\
& \leq & \left(1  + \frac{M_{\phi}L_F}{M}\right) \norms{\widetilde{G}_M(x)}  +  \frac{(1+L_F)}{2M^2} \norms{\widetilde{G}_M(x)}^2 + \norms{\widetilde{F}(x) - F(x)} +  \frac{1}{2}\norms{\widetilde{J}(x) - F'(x)}^2.
\end{array}\tag{\ref{eq:approx_opt_cond}}
\end{equation*}
}
\begin{proof}
First, the optimality condition of  \eqref{eq:surg_model} becomes
\begin{equation}\label{eq:opt_cond_k}
0 \in \widetilde{J}(x)^{\top}\partial{\phi}(\widetilde{F}(x) + \widetilde{J}(x)(\widetilde{T}_M(x) - x)) +  M(\widetilde{T}_M(x) - x).
\end{equation}
We can rewrite this optimality condition as  
\begin{equation*} 
r_F(x) = F'(\widetilde{T}_M(x))^{\top}y \qquad\text{and}\qquad r_D(x) \in -F(\widetilde{T}_M(x))  + \partial{\phi^{*}}(y),
\end{equation*}
where 
\begin{equation*} 
\left\{\begin{array}{ll}
r_F(x) &:= M(x - \widetilde{T}_M(x)) + (F'(\widetilde{T}_M(x)) -  \widetilde{J}(x))^{\top}y, \vspace{1ex}\\
r_D(x) &:= \widetilde{F}(x) + \widetilde{J}(x)(\widetilde{T}_M(x) - x) - F(\widetilde{T}_M(x)).
\end{array}\right.
\end{equation*}
Next, since $y\in\partial{\phi}(\widetilde{F}(x) + \widetilde{J}(x)(\widetilde{T}_M(x) - x))$ and $\phi$ is $M_{\phi}$-Lipschitz continuous, we can bound $y$ as $\norms{y} \leq M_{\phi}$.
Now, we need to bound $r_F$ as follows:
\begin{equation*}
\begin{array}{lcl}
\norms{r_F(x)} &= &  \norms{M(x - \widetilde{T}_M(x)) + (F'(\widetilde{T}_M(x)) -  \widetilde{J}(x))^{\top}y} \vspace{1ex}\\
&= &  \norms{M(x - \widetilde{T}_M(x)) + (F'(\widetilde{T}_M(x)) -  F'(x))^{\top}y + (F'(x) - \widetilde{J}(x))^{\top}y}  \vspace{1ex}\\
&\leq & M\norms{x - \widetilde{T}_M(x)} +  \norms{F'(\widetilde{T}_M(x)) -  F'(x)}_F\norms{y} +  \norms{F'(x) - \widetilde{J}(x)}_F\norms{y} \vspace{1ex}\\
& \leq & \norms{\widetilde{G}_M(x)} + M_{\phi} \norms{F'(\widetilde{T}_M(x)) -  F'(x)}_F  +  M_{\phi} \norms{F'(x) - \widetilde{J}(x)}  \vspace{1ex}\\
& \leq & \left(1  + \frac{M_{\phi}L_F}{M}\right) \norms{\widetilde{G}_M(x)} + M_{\phi} \norms{F'(x) - \widetilde{J}(x)}.
\end{array}
\end{equation*}
Similarly, we can also bound $r_D$ as
\begin{equation*}
\begin{array}{lcl}
\norms{r_D(x)}  & = &  \norms{\widetilde{F}(x) + \widetilde{J}(x)(\widetilde{T}_M(x) - x) - F(\widetilde{T}_M(x))} \vspace{1ex}\\
& = & \Vert\widetilde{F}(x) - F(x) + F(x) + F'(x)(\widetilde{T}_M(x) - x) - F(\widetilde{T}_M(x)) +  [\widetilde{J}(x) - F'(x)](\widetilde{T}_M(x) - x)\Vert \vspace{1ex}\\
& \leq &  \norms{\widetilde{F}(x) - F(x)} +  \norms{F(x) + F'(x)(\widetilde{T}_M(x) - x) - F(\widetilde{T}_M(x))}  +  \norms{[\widetilde{J}(x) - F'(x)](\widetilde{T}_M(x) - x)} \vspace{1ex}\\
& \leq & \norms{\widetilde{F}(x) - F(x)} + \frac{L_F}{2} \norms{\widetilde{T}_M(x) - x}^2  + \frac{1}{2}\norms{F'(x) - \widetilde{J}(x)}^2 + \frac{1}{2}\norms{\widetilde{T}_M(x) - x}^2 \vspace{1ex}\\
& = & \norms{\widetilde{F}(x) - F(x)} + \frac{1}{2}\norms{F'(x) - \widetilde{J}(x)}^2 +  \frac{(1 + L_F)}{2M^2} \norms{\widetilde{G}_M(x)}^2.
\end{array}
\end{equation*}
Combining these bounds, we can show that
\begin{equation*} 
\begin{array}{lcl}
\mathcal{E}(\widetilde{T}_M(x), y) & := & \norms{ F'(\widetilde{T}_M(x))^{\top}y }  +  \dist{0, -F(\widetilde{T}_M(x))  + \partial{\phi^{*}}(y)} \vspace{1ex}\\
&\leq & \left(1  + \frac{M_{\phi}L_F}{M}\right) \norms{\widetilde{G}_M(x)}  + \frac{(1+L_F)}{2M^2}\norms{\widetilde{G}_M(x)}^2  +  \norms{\widetilde{F}(x) - F(x)} + \frac{1}{2}\norms{F'(x) - \widetilde{J}(x)}^2,
\end{array}
\end{equation*}
which is exactly \eqref{eq:approx_opt_cond}.
\end{proof}

\beforesec
\section{The Proof of Technical Results in Section~\ref{sec:inexact_gn_method}: Convergence of Inexact GN Framework}\label{apdx:sec:inexact_gn_method}
\aftersec
This appendix provides the full proof of technical results in Section~\ref{sec:inexact_gn_method} on convergence of the inexact Gauss-Newton framework, Algorithm~\ref{alg:A1}.

\beforesubsec
\subsection{The Proof of Lemma~\ref{le:descent_property}: Descent Property}\label{apdx:le:descent_property}
\aftersubsec
\textbf{Lemma}.~\ref{le:descent_property}.
\textit{
Let Assumption~\ref{as:A1} hold, $\widetilde{T}_M(x)$ be computed by \eqref{eq:surg_model}, and $\widetilde{G}_M(x) := M(x - \widetilde{T}_M(x))$ be the prox-gradient mapping of $F$.
Then, for any $z \in \R^p$, we have
\begin{equation}\label{eq:key_est3}
\phi(\widetilde{F}(x) + \widetilde{J}(x)(\widetilde{T}_M(x)  - x)) \leq \phi(\widetilde{F}(x)  +  \widetilde{J}(x)(z - x))  - \iprods{\widetilde{G}_M(x), z - x}   - \tfrac{1}{M}\norms{\widetilde{G}_M(x)}^2.
\end{equation}
For any $\beta_d > 0$, we also have
\begin{equation*}
\hspace{-0.5ex}
\arraycolsep=0.15em
\begin{array}{lcl}
\phi(F(\widetilde{T}_M(x)))  &\leq & \phi(F(x)) + 2L_\phi\norms{F(x) - \widetilde{F}(x)}   + M_{\phi}\Vert F'(x) - \widetilde{J}(x)\Vert\norms{x - \widetilde{T}_M(x)} -  \frac{(2M - M_{\phi}L_F)}{2}\norms{\widetilde{T}_M(x) - x}^2 \vspace{1ex}\\
& \leq & \phi(F(x)) + 2L_\phi\norms{F(x) - \widetilde{F}(x)}   +  \frac{M_{\phi}}{2\beta_d}\Vert F'(x) - \widetilde{J}(x)\Vert_F^2   - \frac{(2M - M_{\phi}L_F - \beta_d L_\phi)}{2M^2}\norms{\widetilde{G}_M(x)}^2.
\end{array}
\tag{\ref{eq:key_est5}}
\hspace{-4ex}
\end{equation*}
}
\begin{proof}
The optimality condition \eqref{eq:opt_cond_k} can be written as 
\begin{equation*}
\widetilde{J}(x)^{\top}y = M(x - \widetilde{T}_M(x))~~~\text{and}~~~y \in \partial{\phi}(\widetilde{F}(x) + \widetilde{J}(x)(\widetilde{T}_M(x) - x)).
\end{equation*}
By convexity of $\phi$, using the above relations, we have
\begin{equation*}
\arraycolsep=0.15em
\begin{array}{lcl}
\phi(\widetilde{F}(x) + \widetilde{J}(x)(z - x)) & \geq & \phi(\widetilde{F}(x) + \widetilde{J}(x)(\widetilde{T}_M(x) - x))  +  \iprods{y, \widetilde{F}(x) + \widetilde{J}(x)(z - x) - (\widetilde{F}(x) + \widetilde{J}(x)(\widetilde{T}_M(x) - x))} \vspace{1ex}\\
& \geq &  \phi(\widetilde{F}(x) + \widetilde{J}(x)(\widetilde{T}_M(x) - x))  +  \iprods{\widetilde{J}(x)^{\top}y, z -  \widetilde{T}_M(x)} \vspace{1ex}\\
& = & \phi(\widetilde{F}(x) + \widetilde{J}(x)(\widetilde{T}_M(x) - x))  + M\iprods{z - \widetilde{T}_M(x), x - \widetilde{T}_M(x)} \vspace{1ex}\\
& = & \phi(\widetilde{F}(x) + \widetilde{J}(x)(\widetilde{T}_M(x) - x)) + M\iprods{x - \widetilde{T}_M(x), z - x} + M\norms{x - \widetilde{T}_M(x)}^2 \vspace{1ex}\\
& = & \phi(\widetilde{F}(x) + \widetilde{J}(x)(\widetilde{T}_M(x) - x)) + \iprods{\widetilde{G}_M(x), z - x} + \frac{1}{M}\norms{\widetilde{G}_M(x)}^2,
\end{array}
\end{equation*}
which implies \eqref{eq:key_est3}.

Now, combining \eqref{eq:key_est1} and \eqref{eq:key_est3}, we can show that
\begin{equation*} 
\arraycolsep=0.2em
\begin{array}{lcl}
\phi(F(\widetilde{T}_M(x))) & \overset{\tiny \eqref{eq:key_est1}}{\leq} & \phi(F(x) + F'(x)(\widetilde{T}_M(x) - x)) + \frac{M_{\phi}L_F}{2}\norms{\widetilde{T}_M(x) - x}^2  \vspace{1ex}\\
& \leq & \phi(\widetilde{F}(x) + \widetilde{J}(x)(\widetilde{T}_M(x) - x)) +  \frac{M_{\phi}L_F}{2}\norms{\widetilde{T}_M(x) - x}^2   \vspace{1ex}\\
&&  + {~}  \vert \phi(F(x) + F'(x)(\widetilde{T}_M(x) - x)) - \phi(\widetilde{F}(x) + \widetilde{J}(x)(\widetilde{T}_M(x) - x))\vert\vspace{1ex}\\
& \overset{\tiny\eqref{eq:key_est3}}{\leq} &  \phi(\widetilde{F}(x) + \widetilde{J}(x)(z - x)) - M\iprods{x - \widetilde{T}_M(x), z - x} -  \frac{(2M - M_{\phi}L_F)}{2}\norms{\widetilde{T}_M(x) - x}^2 \vspace{1ex}\\
&& + {~} M_{\phi}\norms{F(x) - \widetilde{F}(x) + [F'(x) - \widetilde{J}(x)](\widetilde{T}_M(x) - x)} \vspace{1ex}\\
& \leq & \phi(F(x)) - \frac{(2M - M_{\phi}L_F)}{2}\norms{\widetilde{T}_M(x) - x}^2 + M_{\phi}\Vert F(x) - \widetilde{F}(x)\Vert  \vspace{1ex}\\
&& + {~}  M_{\phi}\norms{F(x) - \widetilde{F}(x) - \widetilde{J}(x)(z - x)} - M\iprods{x - \widetilde{T}_M(x), z - x} \vspace{1ex} \\
&&  + {~}  M_{\phi}\Vert (F'(x) -  \widetilde{J}(x))(\widetilde{T}_M(x) - x)\Vert.
\end{array}
\end{equation*}
Substituting $z = x$ into this estimate, we obtain
\begin{equation}\label{eq:key_est4}
\arraycolsep=0.2em
\begin{array}{lcl}
\phi(F(\widetilde{T}_M(x))) & \leq & \phi(F(x)) - \frac{(2M - M_{\phi}L_F)}{2}\norms{\widetilde{T}_M(x) - x}^2 + 2M_{\phi}\Vert F(x) - \widetilde{F}(x)\Vert \vspace{1ex}\\
&& + {~} M_{\phi}\Vert (F'(x) -  \widetilde{J}(x))(\widetilde{T}_M(x) - x)\Vert.
\end{array}
\end{equation}
Using the Cauchy-Schwarz inequality, we have
\begin{equation*}
\begin{array}{ll}
\Vert (F'(x) -  \widetilde{J}(x))(\widetilde{T}_M(x) - x)\Vert \le \Vert F'(x) - \widetilde{J}(x)\Vert \Vert \widetilde{T}_M(x) - x \Vert.
\end{array}
\end{equation*}
Next, applying Young's inequality to the right hand side of this inequality, for any $\beta_d > 0$, we obtain
\begin{equation}\label{eq:eq2}
\Vert (F'(x) -  \widetilde{J}(x))(\widetilde{T}_M(x) - x)\Vert \leq \Vert F'(x) - \widetilde{J}(x)\Vert_F \Vert \widetilde{T}_M(x) - x \Vert  \leq \frac{1}{2\beta_d}\Vert F'(x) - \widetilde{J}(x)\Vert^2 + \frac{\beta_d}{2} \Vert \widetilde{T}_M(x) - x \Vert^2.
\end{equation}
Finally, plugging \eqref{eq:eq2}  into \eqref{eq:key_est4}, we have
\begin{equation*} 
\arraycolsep=0.2em
\begin{array}{lcl}
\phi(F(\widetilde{T}_M(x))) &\leq & \phi(F(x)) - \frac{(2M - M_{\phi}L_F)}{2}\norms{\widetilde{T}_M(x) - x}^2 + 2L_\phi\norms{F(x) - \widetilde{F}(x)}  +  M_{\phi}\Vert F'(x) - \widetilde{J}(x)\Vert\Vert \widetilde{T}_M(x) - x \Vert \vspace{1ex}\\
&\leq &  \phi(F(x)) - \frac{(2M - M_{\phi}L_F - \beta_d L_\phi)}{2}\norms{\widetilde{T}_M(x) - x}^2 + 2L_\phi\norms{F(x) - \widetilde{F}(x)}  + \frac{M_{\phi}}{2\beta_d}\Vert F'(x) - \widetilde{J}(x)\Vert^2,
\end{array}
\end{equation*}
for any $\beta_d > 0$, which exactly implies \eqref{eq:key_est5}.
\end{proof}

\beforesubsec
\subsection{The Proof of Theorem~\ref{th:convergence1}: Convergence Rate of Algorithm~\ref{alg:A1}}\label{apdx:th:convergence1}
\aftersubsec
\textbf{Theorem.~\ref{th:convergence1}.}
\textit{
Assume that Assumptions~\ref{as:A1} and \ref{as:A2} are satisfied.
Let $\set{x_t}$ be generated by Algorithm~\ref{alg:A1} to solve either \eqref{eq:nl_least_squares} or \eqref{eq:finite_sum}.
Then, the following statements hold:
\begin{compactitem}
\item[$\mathrm{(a)}$] If \eqref{eq:a_cond1} holds for some $\varepsilon \geq 0$, then 
\begin{equation*}
\displaystyle\min_{0\leq t \leq T}\norms{\widetilde{G}_M(x_t)}^2  \leq \dfrac{1}{(T+1)}\displaystyle\sum_{t=0}^T \norms{\widetilde{G}_M(x_t)}^2  \leq  \dfrac{2M^2\left[\Psi(x_0) - \Psi^{\star}\right]}{C_g(T+1)} + \dfrac{\varepsilon^2}{2},
\tag{\ref{eq:convergence_bound1}}
\end{equation*}
where $C_g := 2M - M_{\phi}(L_F + \beta_d)$ for $M > \frac{1}{2}M_{\phi}(L_F + \beta_d)$.
\newline
\item[$\mathrm{(b)}$]~If \eqref{eq:a_cond2b} and \eqref{eq:a_cond2} hold for given $C_a > 0$, then
\vspace{-1ex}
\begin{equation*}
\displaystyle\min_{0\leq t \leq T} \norms{\widetilde{G}_M(x_t)}^2  \leq   \dfrac{1}{(T+1)}\displaystyle\sum_{t=0}^T \norms{\widetilde{G}_M(x_t)}^2    \leq  \dfrac{2M^2\left[\Psi(x_0) - \Psi^{\star}\right]}{C_a(T+1)} + \dfrac{\varepsilon^2}{2}.
\tag{\ref{eq:convergence_bound2}}
\vspace{-0.5ex}
\end{equation*}
\end{compactitem}
Consequently, with $\varepsilon > 0$, the total number of iterations $T$ to achieve $\frac{1}{(T+1)}\displaystyle\sum_{t=0}^T \norms{\widetilde{G}_M(x_t)}^2 \leq \varepsilon^2$ is at most 
\begin{equation*}
T := \left\lfloor \frac{4M^2\left[ \Psi(x_0) - \Psi^{\star} \right]}{D\varepsilon^2} \right\rfloor = \BigO{\frac{\left[ \Psi(x_0) - \Psi^{\star} \right]}{\varepsilon^{2}}}, 
\end{equation*}
where $D := C_g$ for the case $\mathrm{(a)}$ and $D :=  C_a$ for the case $\mathrm{(b)}$.
}
\begin{proof}
Using the second inequality of \eqref{eq:key_est5} with $x := x_t$ and $T_M(x) = x_{t+1}$, we have
\begin{equation}\label{eq:key_est5b}
\phi(F(x_{t+1})) \leq  \phi(F(x_t)) - \frac{(2M - M_{\phi}(L_F + \beta_d))}{2}\norms{x_{t+1} - x_t}^2  + 2M_{\phi} \norms{F(x_t) - \widetilde{F}_t}   +  \frac{M_{\phi}\Vert F'(x_t) - \widetilde{J}_t\Vert^2}{2\beta_d}.
\end{equation}
(a)~If \eqref{eq:a_cond1} holds for some $\varepsilon \geq 0$, then using \eqref{eq:a_cond1} into \eqref{eq:key_est5b},   we have
\begin{equation*} 
\phi(F(x_{t+1})) \leq  \phi(F(x_t)) - \frac{C_g}{2}\norms{x_{t+1} - x_t}^2 + 2M_{\phi}\cdot\frac{C_g\varepsilon^2}{16M_{\phi}M^2} + \frac{M_{\phi}}{2\beta_d}\cdot \frac{\beta_dC_g\varepsilon^2}{4M_{\phi}M^2},
\end{equation*}
where $C_g := 2M - M_{\phi}(L_F + \beta_d) > 0$.
Since $\Psi(x) = \phi(F(x))$, the last estimate leads to
\begin{equation*}
\Psi(x_{t+1})   \leq  \Psi(x_t) - \frac{C_g}{2}\norms{x_{t+1} - x_t}^2 + \frac{C_g\varepsilon^2}{4M^2}.
\end{equation*}
By induction, $\widetilde{G}_M(x_t) := M(x_t - \widetilde{T}_M(x_t))$, and $\Psi(x_{T+1}) \geq \Psi^{\star}$, we can show that
\begin{equation}\label{eq:complexity_bound1}
\frac{1}{M^2(T+1)}\sum_{t=0}^T \norms{\widetilde{G}_M(x_t)}^2 = \frac{1}{T+1}\sum_{t=0}^T \norms{x_{t+1} - x_t}^2 \leq \frac{2\left[\Psi(x_0) - \Psi^{\star}\right]}{C_g(T+1)} + \frac{\varepsilon^2}{2M^2},
\end{equation}
which leads to \eqref{eq:convergence_bound1}.

(b)~If \eqref{eq:a_cond2b} and \eqref{eq:a_cond2} are used, then from \eqref{eq:key_est5b} and \eqref{eq:a_cond2}, we have
\begin{equation*} 
\phi(F(x_{t+1})) \leq  \phi(F(x_t)) - \frac{C_1}{2} \norms{x_{t+1} - x_t}^2 + \frac{C_2}{2}\norms{x_t - x_{t-1}}^2, \quad \forall t\geq 1.
\end{equation*}
where $C_1 := 2M - M_{\phi}L_F - \beta_dM_{\phi}$ and $C_2 := 2M_{\phi}\sqrt{C_f }+ \frac{M_{\phi}C_d}{2\beta_d}$.
For $t = 0$, it follows from \eqref{eq:key_est5b} and \eqref{eq:a_cond2b} that
\begin{equation*} 
\phi(F(x_1)) \leq  \phi(F(x_0)) - \frac{C_1}{2} \norms{x_1 - x_0}^2 + \frac{(C_1 - C_2)\varepsilon^2}{4M^2}.
\end{equation*}
Now, note that  $\Psi(x) = \phi(F(x))$, the last two estimates respectively become
\begin{equation*}
\Psi(x_{t+1}) \leq \Psi(x_t) - \frac{C_1}{2} \norms{x_{t+1} - x_t}^2 + \frac{C_2}{2} \norms{x_t - x_{t-1}}^2, \quad \forall t\geq 1,
\end{equation*}
and for $t = 0$, it holds that
\begin{equation*} 
\Psi(x_1) \leq \Psi(x_0) - \frac{C_1}{2} \norms{x_1 - x_0}^2 + \frac{(C_1 - C_2)\varepsilon^2}{4M^2}.
\end{equation*}
By induction and $\Psi^{\star} \leq \Psi(x_{T+1})$, this estimate leads to 
\begin{equation*}
\begin{array}{ll}
\Psi^{\star} \leq  \Psi(x_{T+1}) & \leq \Psi(x_0) - \frac{(C_1 - C_2)}{2}\sum_{t=0}^T \norms{x_{t+1} - x_t}^2 + \frac{(C_1 - C_2)\varepsilon^2}{4M^2} \vspace{1ex}\\
& - {~} \frac{C_2}{2}\norms{x_{T+1} - x_T}^2.
\end{array}
\end{equation*}
Since $C_1 > C_2$, if we define $C_a := C_1 - C_2 > 0$, then the last inequality implies
\begin{equation*}
\frac{1}{M^2(T+1)}\sum_{t=0}^T \norms{\widetilde{G}_M(x_t)}^2 = \frac{1}{(T+1)}\sum_{t=0}^T \norms{x_{t+1} - x_t}^2  \leq \frac{2\left[\Psi(x_0) - \Psi^{\star}\right]}{C_a(T+1)} + \frac{\varepsilon^2}{4M^2},
\end{equation*}
which leads to \eqref{eq:convergence_bound2}.
The last statement of this theorem is a direct consequence of either \eqref{eq:convergence_bound1} or \eqref{eq:convergence_bound2}, and we omit the detailed derivation here.
\end{proof}

\beforesec
\section{High Probability Inequalities and Variance Bounds}\label{apdx:sec:prob_tools}
\aftersec
Since our methods are stochastic, we recall some mathematical tools from high probability and concentration theory, as well as variance bounds that will be used for our analysis.
First, we need the following lemmas to estimate sample complexity of our algorithms.

\begin{lemma}[Matrix Bernstein inequality \cite{tropp2012user}(Theorem 1.6)]\label{le:con_lemma}
Let $X_1, X_2, \cdots, X_n$ be independent random matrices in $\R^{p_1\times p_2}$.
Assume that $\Exp{X_i} = 0$ and $\norms{X_i} \leq R$ a.s. for $i=1,\cdots, n$ and given $R > 0$, where $\norms{\cdot}$ is the spectral norm.
Define $\sigma_X^2 := \max\set{\norm{\sum_{i=1}^n\Exp{X_iX_i^{\top}}}, \norm{\sum_{i=1}^n\Exp{X_i^{\top}X_i}}}$.
Then, for any $\epsilon > 0$, we have
\begin{equation*}
\Prob{\Big\Vert\sum_{i=1}^nX_i \Big\Vert \geq \epsilon} \leq (p_1+p_2)\exp\left(-\frac{3\epsilon^2}{6\sigma_X^2 + 2R \epsilon}\right).
\end{equation*}
As a consequence, if $\sigma_X^2 \leq \bar{\sigma}_X^2$ for a given $\bar{\sigma}_X^2 > 0$, then
\begin{equation*}
\Prob{\Big\Vert\sum_{i=1}^nX_i \Big\Vert \leq \epsilon}  \geq 1 - (p_1+p_2)\exp\left(-\frac{3\epsilon^2}{6\bar{\sigma}_X^2 + 2R \epsilon}\right).
\end{equation*}
\end{lemma}

\begin{lemma}[\citet{lohr2009sampling}]\label{leq:sgd_variance}
Let $\widetilde{F}(x_t)$ and $\widetilde{J}(x_t)$ be the mini-batch stochastic estimators of $F(x_t)$ and $F'(x_t)$ defined by \eqref{eq:msgd_estimators}, respectively, and $\Fc_t := \sigma(x_0, x_1, \cdots, x_{t-1})$ be the $\sigma$-field generated by $\set{x_0, x_1, \cdots, x_{t-1}}$.
Then, these are unbiased estimators, i.e., $\Exp{\widetilde{F}(x_t)\mid \Fc_t} = F(x_t)$ and $\Exp{\widetilde{J}(x_t) \mid \Fc_t} = F'(x_t)$.
Moreover, under Assumption~\ref{as:A2}, we have
\begin{equation}\label{eq:msgd_estimators_var}
\Exp{\norms{\widetilde{F}(x_t) - F(x_t)}^2 \mid \Fc_t} \leq \frac{\sigma_F^2}{b_t}~~~~~~\text{and}~~~~~ \Exp{\norms{\widetilde{J}(x_t) - F'(x_t)}^2 \mid \Fc_t} \leq \frac{\sigma_D^2}{\hat{b}_t}.
\end{equation}
\end{lemma}

\begin{lemma}[\citet{nguyen2017sarah,Pham2019}]\label{le:sarah_estimators}
Let $\widetilde{F}_t$ and $\widetilde{J}_t$ be the mini-batch SARAH estimators of $F(x_t)$ and $F'(x_t)$, respectively defined by \eqref{eq:SARAH_estimators}, and $\Fc_t := \sigma(x_0, x_1, \cdots, x_{t-1})$ be the $\sigma$-field generated by $\set{x_0, x_1, \cdots, x_{t-1}}$.
Then, we have the following estimate
\begin{equation}\label{eq:sarah_F_property2}
\Exp{\norms{\widetilde{F}_t - F(x_t)}^2\mid\Fc_t} = \norms{\widetilde{F}_{t-1} - F(x_{t-1})}^2  +  \rho_t\Exps{\xi}{\norm{\Fb(x_t, \xi) - \Fb(x_{t-1},\xi) }^2}  - \rho_t\norms{F(x_t) - F(x_{t-1})}^2,
\end{equation}
where $\rho_t := \frac{n-b_t}{(n-1)b_t}$ if $F(x) := \frac{1}{n}\sum_{i=1}^nF_i(x)$, and $\rho_t := \frac{1}{b_t}$, otherwise, i.e., $F(x) = \Exps{\xi}{\Fb(x,\xi)}$.

Similarly, we also have
\begin{equation}\label{eq:sarah_J_property}
\Exp{\norms{\widetilde{J}_t - F'(x_t)}^2\mid\Fc_t} = \norms{\widetilde{J}_{t-1} - F'(x_{t-1})}^2  + \hat{\rho}_t\Exps{\xi}{\norm{\Fb'(x_t, \xi) - \Fb'(x_{t-1},\xi) }^2}  -  \hat{\rho}_t\norms{F'(x_t) - F'(x_{t-1})}^2,
\end{equation}
where $\hat{\rho}_t := \frac{n - \hat{b}_t}{(n-1)\hat{b}_t}$ if $F(x) := \frac{1}{n}\sum_{i=1}^nF_i(x)$, and $\hat{\rho}_t := \frac{1}{\hat{b}_t}$, otherwise, i.e., $F(x) = \Exps{\xi}{\Fb(x,\xi)}$.
\end{lemma}

\beforesec
\section{The Proof of Technical Results in Section~\ref{sec:sgn_methods}}\label{apdx:sec:sgn_methods}
\aftersec
This appendix provides the full proof of technical results in Section~\ref{sec:sgn_methods} on our stochastic Gauss-Newton methods.

\beforesubsec
\subsection{The Proof of Theorem~\ref{th:sgd_complexity1}: Convergence of The Stochastic Gauss-Newton Method for Solving \eqref{eq:nl_least_squares}}\label{apdx:th:sgd_complexity}
\aftersubsec
\textbf{Theorem.~\ref{th:sgd_complexity1}.}
\textit{
Suppose that Assumptions~\ref{as:A1} and \ref{as:A2} hold for \eqref{eq:nl_least_squares}.
Let $\widetilde{F}_t$ and $\widetilde{J}_t$ defined by \eqref{eq:msgd_estimators} be mini-batch stochastic estimators  of $F(x_t)$ and $F'(x_t)$, respectively.
Let $\set{x_t}$ be generated by Algorithm~\ref{alg:A1} $($called \textbf{SGN}$)$ to solve \eqref{eq:nl_least_squares}.
For a given tolerance $\varepsilon > 0$, assume that $b_t$ and $\hat{b}_t$ in \eqref{eq:msgd_estimators} are chosen as 
\begin{equation*} 
\left\{\begin{array}{lclcl}
b_t &:= & \left\lfloor \dfrac{256M_{\phi}^2M^4\sigma^2_F}{C_g^2\varepsilon^4} \right\rfloor  & = &  \BigO{\dfrac{\sigma^2_F}{\varepsilon^4}},  \vspace{1ex}\\
\hat{b}_t &:= & \left\lfloor \dfrac{2M_{\phi}M^2\sigma_D^2}{\beta_dC_g\varepsilon^2} \right\rfloor & = &  \BigO{\dfrac{\sigma^2_D}{\varepsilon^2}}.
\end{array}\right.
\tag{\ref{eq:bt_size1}}
\end{equation*}
Furthermore, let $\widehat{x}_T$ be chosen uniformly at random in $\set{x_t}_{t=0}^T$ as the output of Algorithm~\ref{alg:A1} after $T$ iterations.
Then
\begin{equation*}
\Exp{\norms{\widetilde{G}_M(\widehat{x}_T)}^2}   =  \dfrac{1}{(T+1)}\displaystyle\sum_{t=0}^T\Exp{ \norms{\widetilde{G}_M(x_t)}^2}  \leq  \dfrac{2M^2\left[\Psi(x_0) - \Psi^{\star}\right]}{C_g(T+1)} + \dfrac{\varepsilon^2}{2},
\tag{\ref{eq:convergence_bound1_sgd}}
\end{equation*}
where $C_g := 2M - M_{\phi}(L_F + \beta_d)$ with  $M > \frac{1}{2}M_{\phi}(L_F + \beta_d)$.
Moreover, the total number $\Tc_f$ of function evaluations $\Fb(x_t,\xi)$ and the total number $\Tc_d$ of Jacobian evaluations $\Fb'(x_t,\zeta)$ to achieve $\Exp{\norms{\widetilde{G}_M(\widehat{x}_T)}^2} \leq \varepsilon^2$ do not exceed 
\begin{equation*}
\left\{\begin{array}{lclcl}
\Tc_f   &:= &  \left\lfloor \dfrac{1024M^6M_{\phi}^2\sigma_F^2\left[\Psi(x_0) - \Psi^{\star}\right]}{C_g^3\varepsilon^6} \right\rfloor   &= & \BigO{\dfrac{\sigma_F^2}{\varepsilon^6}}, \vspace{1ex}\\
\Tc_d  &:= &  \left\lfloor \dfrac{8M^4M_{\phi}\sigma_D^2\left[\Psi(x_0) - \Psi^{\star}\right]}{\beta_dC_g^2\varepsilon^4} \right\rfloor &= & \BigO{\dfrac{\sigma_D^2}{\varepsilon^4}}.
\end{array}\right.
\tag{\ref{eq:complexity_case1}}
\vspace{-2ex}
\end{equation*}
}
\begin{proof}
Let $\Fc_t := \sigma(x_0, x_1, \cdots, x_{t-1})$ be the $\sigma$-field generated by $\set{x_0, x_1, \cdots, x_{t-1}}$.
By repeating a similar proof as of \eqref{eq:convergence_bound1}, but taking the full expectation overall the randomness with $\Exp{\cdot} = \Exp{\Exp{\cdot} \mid \Fc_{t+1}}$, we have
\begin{equation}\label{eq:convergence_bound1_sgd2b}
\frac{1}{(T+1)}\displaystyle\sum_{t=0}^T \Exp{\norms{\widetilde{G}_M(x_t)}^2} \leq  \dfrac{2M^2\left[\Psi(x_0) - \Psi^{\star}\right]}{C_g(T+1)} + \dfrac{\varepsilon^2}{2},
\end{equation}
where $C_g := 2M - M_{\phi}(L_F + \beta_d)$ with $M > \frac{1}{2}M_{\phi}(L_F + \beta_d)$.
Moreover, by the choice of $\widehat{x}_T$, we have $\Exp{\norms{\widetilde{G}_M(\widehat{x}_T)}^2}   =  \dfrac{1}{(T+1)}\displaystyle\sum_{t=0}^T\Exp{ \norms{\widetilde{G}_M(x_t)}^2}$.
Combining this relation and \eqref{eq:convergence_bound1_sgd2b}, we proves \eqref{eq:convergence_bound1_sgd}.

Next, by Lemma~\ref{leq:sgd_variance}, to guarantee the condition \eqref{eq:a_cond1} in expectation, i.e.:
\begin{equation*}
\left\{\begin{array}{lcl}
\Exp{\norms{\widetilde{F}(x_t) - F(x_t)}^2 \mid \Fc_t} & \leq & \dfrac{C_g^2\varepsilon^4}{256M_{\phi}^2M^4}, \vspace{1ex}\\
\Exp{\norms{\widetilde{J}(x_t) - F'(x_t)}^2 \mid \Fc_t} & \leq & \dfrac{ \beta_dC_g\varepsilon^2}{2M^2M_{\phi}},
\end{array}\right.
\end{equation*}
we have to choose $\frac{\sigma_F^2}{b_t} \leq  \frac{C_g^2\varepsilon^4}{256M_{\phi}^2M^4}$ and $\frac{\sigma_D^2}{\hat{b}_t} \leq \frac{\beta_dC_g\varepsilon^2}{2M_{\phi}M^2}$, which respectively lead to 
\begin{equation*}
b_t \geq  \frac{256M_{\phi}^2M^4\sigma_F^2}{C_g^2\varepsilon^4} ~~~~~~\text{and}~~~~~~\hat{b}_t \geq  \frac{2M_{\phi}M^2\sigma_D^2}{\beta_dC_g\varepsilon^2}.
\end{equation*}
By rounding to the nearest integer, we obtain \eqref{eq:bt_size1}.
Using \eqref{eq:convergence_bound1}, we can see that since $\Exp{\norms{\widetilde{G}_M(\hat{x}_T)}^2}  = \frac{1}{(T+1)}\sum_{t=0}^T\Exp{\norms{\widetilde{G}_M(x_t)}^2}$, to guarantee $\Exp{\norms{\widetilde{G}_M(\hat{x}_T)}^2} \leq \varepsilon^2$, we impose $\frac{2M^2\left[\Psi(x_0) - \Psi^{\star}\right]}{C_g(T+1)} \leq \frac{\varepsilon^2}{2}$, which leads to $T := \left\lfloor \frac{4M^2\left[\Psi(x_0) - \Psi^{\star}\right]}{C_g\varepsilon^2} \right\rfloor$.
Hence, the total number $\Tc_f$ of stochastic function evaluations $\Fb(x_t, \xi)$ can be bounded by
\begin{equation*}
\Tc_f := Tb_t = \left\lfloor \frac{1024M^6M_{\phi}^2\sigma_F^2\left[\Psi(x_0) - \Psi^{\star}\right]}{C_g^3\varepsilon^6} \right\rfloor = \BigO{\frac{\sigma_F^2}{\varepsilon^6}}.
\end{equation*}
Similarly, the total number $\Tc_d$ of stochastic Jacobian evaluations $\Fb'(x_t, \zeta)$ can be bounded by
\begin{equation*}
\Tc_d := T\hat{b}_t = \left\lfloor \frac{8M^4M_{\phi}\sigma_D^2\left[\Psi(x_0) - \Psi^{\star}\right]}{\beta_dC_g^2\varepsilon^4} \right\rfloor = \BigO{\frac{\sigma_D^2}{\varepsilon^4}}.
\end{equation*}
These two last estimates prove \eqref{eq:complexity_case1}.
\end{proof}

\beforesubsec
\subsection{The Proof of Theorem~\ref{th:sgd_complexity2}: Convergence of The Stochastic Gauss-Newton Method for Solving \eqref{eq:finite_sum}}\label{apdx:th:sgd_complexity}
\aftersubsec
\textbf{Theorem.~\ref{th:sgd_complexity2}.}
\textit{
Suppose that Assumptions~\ref{as:A1} and \ref{as:A2} hold for \eqref{eq:finite_sum}.
Let $\widetilde{F}_t$ and $\widetilde{J}_t$ defined by \eqref{eq:msgd_estimators} be mini-batch stochastic estimators  to approximate $F(x_t)$ and $F'(x_t)$, respectively.
Let $\set{x_t}$ be generated by Algorithm~\ref{alg:A1} for solving \eqref{eq:finite_sum}.
Assume that $b_t$ and $\hat{b}_t$ in \eqref{eq:msgd_estimators} are chosen such that $b_t := \min\set{n, \bar{b}_t}$ and $\hat{b}_t := \min\set{n, \hat{\bar{b}}_t}$ for $t\geq 0$, where
\begin{equation*} 
\left\{ \begin{array}{lcl}
\bar{b}_0  &:= & \left\lfloor \dfrac{32M_{\phi}M^2\sigma_F\left[48\sigma_FM_{\phi}M^2 + C_a \varepsilon^2\right]}{3 C_a^2\varepsilon^4} \cdot \log\left(\dfrac{p+1}{\delta}\right) \right\rfloor, \vspace{1ex}\\
\hat{\bar{b}}_0  &:= &  \left\lfloor \dfrac{4M\sqrt{2M_{\phi}}\sigma_D\left(3M\sqrt{2M_{\phi}}\sigma_D + \sqrt{\beta_d C_a}\varepsilon\right)}{\beta_dC_a\varepsilon^2}\cdot \log\left(\dfrac{p+q}{\delta}\right) \right\rfloor, \vspace{1ex}\\
\bar{b}_t  &:= &  \left\lfloor \dfrac{\big(6\sigma_F^2 + 2\sigma_F\sqrt{C_f}\norms{x_t - x_{t-1}}^2\big)}{3C_f^2\norms{x_t - x_{t-1}}^4} \cdot \log\left(\dfrac{p+1}{\delta}\right) \right\rfloor \quad (t\geq 1), \vspace{1ex}\\
\hat{\bar{b}}_t  &:= & \left\lfloor \dfrac{\left(6\sigma_D^2 + 2\sigma_D\sqrt{C_d}\norms{x_t - x_{t-1}}\right)}{3C_d\norms{x_t - x_{t-1}}^2}\cdot \log\left(\dfrac{p+q}{\delta}\right) \right\rfloor \quad (t\geq 1),
\end{array}\right.
\tag{\ref{eq:bt_size2}}
\end{equation*}
for $\delta \in (0, 1)$, and $C_f$ and $C_d$ given in \textbf{Condition 2}, where $\varepsilon > 0$ is a given tolerance.
\newline
Then, we have the following conclusions:
\begin{compactitem}
\item With probability at least $1-\delta$, the bound \eqref{eq:convergence_bound2} in Theorem \ref{th:convergence1} still holds.
\item Moreover, the total number $\Tc_f$ of stochastic function evaluations $\Fb(x_t, \xi)$ and the total number $\Tc_d$ of stochastic Jacobian evaluations $\Fb'(x_t,\zeta)$ to guarantee $\frac{1}{(T+1)}\sum_{t=0}^T\norms{\widetilde{G}_M(x_t)}^2 \leq \varepsilon^2$ do not exceed 
\begin{equation*}
\left\{\begin{array}{lcl}
\Tc_f  &:= & \BigO{\dfrac{\sigma_F^2\left[\Psi(x_0) - \Psi^{\star}\right]}{\varepsilon^6} \cdot \log\left(\dfrac{p+1}{\delta}\right)}, \vspace{1ex}\\
\Tc_d &:=  &  \BigO{\dfrac{\sigma_D^2\left[\Psi(x_0) - \Psi^{\star}\right]}{\varepsilon^4} \cdot \log\left(\dfrac{p+q}{\delta}\right)}.
\end{array}\right.
\tag{\ref{eq:complexity_case2}}
\vspace{-2ex}
\end{equation*}
\end{compactitem}
}
\begin{proof}
We first use Lemma~\ref{le:con_lemma} to estimate the total number of samples for $F(x_t)$ and $F'(x_t)$.
Let $\Fc_t := \sigma(x_0, x_1, \cdots, x_{t-1})$ be the $\sigma$-field generated by $\set{x_0, x_1, \cdots, x_{t-1}}$.
We define $X_i := F_i(x_t) - F(x_t) \in \R^p$ for $i \in \Bc_t$. 
Conditioned on $\Fc_t$, due to the choice of $\Bc_t$, $\set{X_i}_{i\in\Bc_t}$ are independent vector-valued random variables and  $\Exp{X_i} = 0$.
Moreover, by Assumption~\ref{as:A2}, we have $\norms{F_i(x) - F(x)} \leq \sigma_F$ for all $i\in [n]$.  
This implies that $\norms{X_i} \leq \sigma_F$ a.s. and $\Exp{\norms{X_i}^2} \leq \sigma_F^2$.
Hence, the conditions of Lemma~\ref{le:con_lemma} hold.
In addition, we have
\begin{equation*}
\sigma_X^2 := \max\set{\Big\Vert \sum_{i\in\Bc_t}\Exp{X_iX_i^{\top}}\Big\Vert, \Big\Vert\sum_{i\in\Bc_t}\Exp{X_i^{\top}X_i}\Big\Vert} \leq \sum_{i\in\Bc_t}\Exp{\norms{X_i}^2} \leq b_t\sigma_F^2.
\end{equation*}
Since $\widetilde{F}_t := \frac{1}{b_t}\sum_{i\in\Bc_t}F_i(x_t)$, by Lemma~\ref{le:con_lemma}, we have
\begin{equation*}
\begin{array}{lcl}
\Prob{ \norms{\widetilde{F}_t - F(x_t)} \leq  \epsilon } &= & \Prob{ \Big\Vert \sum_{i\in\Bc_t}X_i\Big\Vert \leq  b_t\epsilon } \vspace{1ex}\\
& \geq & 1 - (p+1)\exp\left(-\frac{3b_t^2\epsilon^2}{6b_t\sigma_F^2 + 2\sigma_F b_t \epsilon}\right) \vspace{1ex}\\
& = & 1 - (p+1)\exp\left(-\frac{3b_t\epsilon^2}{6\sigma_F^2 + 2\sigma_F \epsilon}\right).
\end{array}
\end{equation*}
Let us choose $\delta \in (0, 1]$ such that $\delta \geq  (p+1)\exp\left(-\frac{3b_t\epsilon^2}{6\sigma_F^2 + 2\sigma_F \epsilon}\right)$ and $\delta \leq 1$, then $\Prob{ \norms{\widetilde{F}_t - F(x_t)} \leq  \epsilon } \geq 1 - \delta$.
Hence, we have $b_t \geq \left(\frac{6\sigma_F^2 + 2\sigma_F\epsilon}{3\epsilon^2}\right)\cdot\log\left(\frac{p+1}{\delta}\right)$.

To guarantee the first condition of \eqref{eq:a_cond2b}, we choose $\epsilon := \frac{C_a\varepsilon^2}{16M_{\phi}M^2}$.
Then, the condition on $b_0$ leads to $b_0 \geq \frac{32M_{\phi}M^2\sigma_F\left(48\sigma_FM_{\phi}M^2 + C_a\varepsilon^2\right)}{3 C_g^2\varepsilon^4} \cdot \log\left(\frac{p+1}{\delta}\right)$.
To guarantee the first condition of \eqref{eq:a_cond2}, we choose $\epsilon := \sqrt{C_f}\norm{x_t - x_{t-1}}^2$.
Then, the condition on $b_t$ leads to $b_t \geq \frac{(6\sigma_F^2 + 2\sigma_F\sqrt{C_f}\norms{x_t - x_{t-1}}^2}{3C^2\norms{x_t - x_{t-1}}^4} \cdot \log\left(\frac{p+1}{\delta}\right)$.
Rounding both $b_0$ and $b_t$, we obtain
\begin{equation*} 
\left\{\begin{array}{lclcl}
\bar{b}_0 & := & \left\lfloor \frac{32M_{\phi}M^2\sigma_F\left[48\sigma_FM_{\phi}M^2 +  C_a \varepsilon^2\right]}{3 C_a^2\varepsilon^4} \cdot \log\left(\frac{p+1}{\delta}\right) \right\rfloor & = &  \BigO{\frac{\sigma_F^2}{\varepsilon^4} \cdot \log\left(\frac{p}{\delta}\right)}, \vspace{1ex}\\
\bar{b}_t & := & \left\lfloor \frac{(6\sigma_F^2 + 2\sigma_F\sqrt{C_f}\norms{x_t - x_{t-1}}^2}{3C^2\norms{x_t - x_{t-1}}^4} \cdot \log\left(\frac{p+1}{\delta}\right) \right\rfloor & = & \BigO{\frac{\sigma_F^2}{\norms{x_t - x_{t-1}}^4} \cdot \log\left(\frac{p}{\delta}\right)}, \quad \forall t\geq 1.
\end{array}\right.
\end{equation*}
Since $b_t \leq n$ for all $t \geq 0$, we have $b_t := \min\set{n, \bar{b}_t}$ for $t \geq 0$, which proves the first part of \eqref{eq:bt_size2}.

Next, we estimate a sample size for $\widetilde{J}_t$.
Let us define $Y_i := F_i'(x_t) - F'(x_t)$.
Then, similar to the above proof of $X_i$ for $F$, we have $\widetilde{J}_t - F'(x_t) = \frac{1}{\hat{b}_t}\sum_{i\in\hat{\Bc}_t}(F_i'(x_t) - F'(x_t)) = \frac{1}{\hat{b}_t}\sum_{i\in\hat{\Bc}_t}Y_i$.
Under Assumption~\ref{as:A2}, the sequence $\set{Y_i}$ satisfies all conditions of Lemma~\ref{le:con_lemma}.
Hence, we obtain
\begin{equation*}
\Prob{\norms{\widetilde{J}_t - F'(x_t)} \leq \epsilon} \geq 1 - (p + q)\exp\left(\frac{-3\hat{b}_t\epsilon^2}{6\sigma_D^2 + 2\sigma_D\epsilon}\right).
\end{equation*}
Hence, we can choose $\hat{b}_t \geq \left[\frac{6\sigma_D^2 + 2\sigma_D\epsilon}{3\epsilon^2}\right]\cdot \log\left(\frac{p+q}{\delta}\right)$.
From the second condition of  \eqref{eq:a_cond2b}, if we choose $\epsilon := \frac{\sqrt{\beta_dC_a}\varepsilon}{M\sqrt{2M_{\phi}}}$, then we have $\hat{b}_0 \geq \frac{4M\sqrt{2M_{\phi}}\sigma_D\left(3M\sqrt{2M_{\phi}}\sigma_D + \sqrt{\beta_dC_a}\varepsilon\right)}{\beta_dC_a\varepsilon^2}\cdot \log\left(\frac{p+q}{\delta}\right)$.
From the second condition of \eqref{eq:a_cond2}, if we choose $\epsilon := \sqrt{C_d}\norms{x_t - x_{t-1}}$, then we have $\hat{b}_t \geq \frac{\left(6\sigma_D^2 + 2\sigma_D\sqrt{C_d}\norms{x_t - x_{t-1}}\right)}{3C_d\norms{x_t - x_{t-1}}^2}\cdot \log\left(\frac{p+q}{\delta}\right)$.
Rounding $\hat{b}_t$, we obtain
\begin{equation*}
\begin{array}{lclcl}
\hat{\bar{b}}_0 & := & \left\lfloor \frac{4M\sqrt{2M_{\phi}}\sigma_D\left(3M\sqrt{2M_{\phi}}\sigma_D + \sqrt{\beta_dC_a}\varepsilon\right)}{\beta_d C_a\varepsilon^2} \cdot \log\left(\frac{p+q}{\delta}\right) \right\rfloor & = &  \BigO{\frac{\sigma_D^2}{\varepsilon^2}\cdot \log\left(\frac{p+q}{\delta}\right)},\vspace{1ex}\\
\hat{\bar{b}}_t & := &  \left\lfloor \frac{\left(6\sigma_D^2 + 2\sigma_D\sqrt{C_d}\norms{x_t - x_{t-1}}\right)}{3C_d\norms{x_t - x_{t-1}}^2}\cdot \log\left(\frac{p+q}{\delta}\right) \right\rfloor & = &  \BigO{\frac{\sigma_D^2}{\norms{x_t - x_{t-1}}^2}\cdot \log\left(\frac{p+q}{\delta}\right)}, \quad t \geq 1.
\end{array}
\end{equation*}
Since $\hat{b}_t \leq n$ for all $t\geq 0$, combining these conditions, we obtain $\hat{b}_t := \min\sets{n, \hat{\bar{b}}_t}$ for $t\geq 0$, which proves the second part of \eqref{eq:bt_size2}.

For $t\geq 1$, we have $\norms{\widetilde{G}_M(x_{t-1})} = M\norms{x_t - x_{t-1}} > \varepsilon$.
Otherwise, the algorithm has been terminated.
Therefore, we can even bound $b_t$ and $\hat{b}_t$ as
\begin{equation*}
b_t \leq  \frac{2M^2\sigma_F(3M^2\sigma_F + \sqrt{C_f}\varepsilon^2)}{3C^2\varepsilon^4} \cdot \log\left(\frac{p+1}{\delta}\right) ~~~~\text{and}~~\hat{b}_t \leq \frac{M\left(6M\sigma_D^2 + 2\sigma_D\sqrt{C_d}\varepsilon\right)}{3C_d\varepsilon^2} \cdot \log\left(\frac{p+q}{\delta}\right). 
\end{equation*}
From \eqref{eq:convergence_bound2}, to guarantee $\frac{1}{(T+1)}\sum_{t=0}^T \norms{\widetilde{G}_M(x_t)}^2 \leq \varepsilon^2$, we impose $\frac{2M^2\left[\Psi(x_0) - \Psi^{\star}\right]}{C_a(T+1)} \leq \frac{\varepsilon^2}{2}$, which leads to $T := \left\lfloor \frac{4M^2\left[\Psi(x_0) - \Psi^{\star}\right]}{C_a\varepsilon^2} \right\rfloor$.
Hence, the total number $\Tc_f$ of stochastic function evaluations $\Fb(\cdot, \xi)$ can be bounded by
\begin{equation*}
\begin{array}{lcl}
\Tc_f & := &  b_0 + (T-1)b_t \vspace{1ex}\\
& \leq &  \left[ \frac{32M_{\phi}M^2\sigma_F\left(48\sigma_FM_{\phi}M^2 + C_a\varepsilon^2\right)}{3C_a^2\varepsilon^4}  +  \frac{8M^4\sigma_F(3M^2\sigma_F + \sqrt{C_f}\varepsilon^2)\left[\Psi(x_0) - \Psi^{\star}\right]}{3C^2C_a\varepsilon^6}\right] \cdot \log\left(\frac{p+1}{\delta}\right) \vspace{1ex}\\
& = & \BigO{\frac{\sigma^2_F\left[\Psi(x_0) - \Psi^{\star}\right]}{\varepsilon^6} \cdot \log\left(\frac{p+1}{\delta}\right)}.
\end{array}
\end{equation*}
Similarly, the total number $\Tc_d$ of stochastic Jacobian evaluations $\Fb'(\cdot, \zeta)$ can be bounded by
\begin{equation*}
\begin{array}{lcl}
\Tc_d & := & \hat{b}_0 + (T-1)\hat{b}_t \vspace{1ex}\\
&\leq & \left[ \frac{4M\sqrt{2M_{\phi}}\sigma_D\left(3M\sqrt{2M_{\phi}}\sigma_D + \sqrt{\beta_d C_a}\varepsilon\right)}{\beta_d C_a\varepsilon^2} +  \frac{4M^3\left[\Psi(x_0) - \Psi^{\star}\right]\left(6M\sigma_D^2 + 2\sigma_D\sqrt{C_d}\varepsilon\right)}{3C_d C_a \varepsilon^4}  \right] \cdot \log\left(\frac{p+q}{\delta}\right) \vspace{1ex}\\
& = & \BigO{\frac{\sigma_D^2\left[\Psi(x_0) - \Psi^{\star}\right]}{\varepsilon^4} \cdot \log\left(\frac{p+q}{\delta}\right)}.
\end{array}
\end{equation*}
Taking the upper bounds, these two last estimates prove \eqref{eq:complexity_case2}.
\end{proof}

\beforesubsec
\subsection{The Proof of Theorem~\ref{th:convergence_of_Sarah_GN}: Convergence and Complexity Analysis of Algorithm~\ref{alg:A2} for \eqref{eq:nl_least_squares}}\label{apdx:th:convergence_of_Sarah_GN}
\aftersubsec
\textbf{Theorem.~\ref{th:convergence_of_Sarah_GN}.}
\textit{
Suppose that Assumptions~\ref{as:A1} and \ref{as:A2}, and \ref{as:A3} are satisfied for \eqref{eq:nl_least_squares}.
Let $\sets{x_t^{(s)}}_{t=0\to m}^{s=1\to S}$ be  generated by Algorithm~\ref{alg:A2} to solve \eqref{eq:nl_least_squares}.
Let $\theta_F$ and $m$ be chosen by \eqref{eq:rho_quantity}, and  the mini-batches $b_s$, $\hat{b}_s$, $b_t^{(s)}$, and $\hat{b}_t^{(s)}$ be set as in \eqref{eq:mini_batch}.
Assume that the output $\widehat{x}_T$ of Algorithm~\ref{alg:A2} is chosen uniformly at random in $\sets{x_t^{(s)}}_{t=0\to m}^{s=1\to S}$.
Then:
\begin{compactitem}
\item[$\mathrm{(a)}$]~For a given tolerance $\varepsilon > 0$, the following bound holds
\begin{equation*} 
\begin{array}{ll}
\Exp{\norms{\widetilde{G}_M(\widehat{x}_T)}^2}  &=  \dfrac{1}{S(m+1)}\displaystyle\sum_{s=1}^S\sum_{t=0}^m \Exp{\norms{\widetilde{G}_M(x_t)}^2}  \leq  \varepsilon^2.
\end{array}\tag{\ref{eq:convergence2_bound1}}
\end{equation*}
\item[$\mathrm{(b)}$]~The total number of iterations $T$ to obtain $\Exp{\norms{\widetilde{G}_M(\widehat{x}_T)}^2} \leq \varepsilon^2$ is at most 
\begin{equation*}
T := S(m+1) =  \left\lfloor \frac{8M^2\left[\Psi(\widetilde{x}^0) - \Psi^{\star}\right]}{\theta_F\varepsilon^2} \right\rfloor = \BigO{\frac{1}{\varepsilon^{2}}}.
\end{equation*}
Moreover, the total numbers $\Tc_f$ and $\Tc_d$ of stochastic function evaluations $\Fb(x_t,\xi)$ and  stochastic Jacobian evaluations $\Fb'(x_t,\zeta)$, respectively do not exceed:
\begin{equation*} 
\left\{\begin{array}{lcl}
\Tc_f & := &  \BigO{ \dfrac{M_{\phi}^2\sigma_F^2}{\theta_F^2\varepsilon^4} + \dfrac{M^4M_{\phi}^2\left[\Psi(\widetilde{x}^{0}) - \Psi^{\star}\right]}{\theta_F^2\varepsilon^5}}, \vspace{1ex}\\
\Tc_d & := & \BigO{ \dfrac{M_{\phi}\sigma_D^2}{\theta_F\varepsilon^2} + \dfrac{M^2M_{\phi} \left[\Psi(\widetilde{x}^{0}) - \Psi^{\star}\right]}{\theta_F \varepsilon^3}}.
\end{array}\right.
\tag{\ref{eq:total_or_complexity}}
\vspace{-2ex}
\end{equation*}
\end{compactitem}
}
\begin{proof}
We first analyze the inner loop.
Using \eqref{eq:key_est5} with $x := x_t^{(s)}$ and $T_M(x) = x^{(s)}_{t+1}$, and then taking the expectation conditioned on $\Fc_{t+1}^{(s)} := \sigma(x_0^{(s)}, x_1^{(s)},\cdots, x_{t}^{(s)})$, we have
\begin{equation*}
\begin{array}{lcl}
\Exp{\phi(F(x^{(s)}_{t+1})) \mid \Fc_{t+1}^{(s)}} & \leq  & \phi(F(x_t^{(s)})) - \frac{(2M - M_{\phi}(L_F + \beta_d))}{2}\Exp{\norms{x^{(s)}_{t+1} - x^{(s)}_t}^2 \mid \Fc_{t+1}^{(s)}}  \vspace{1ex}\\
&& + {~}  \frac{L_\phi}{\xi_t^s}\Exp{\norms{F(x^{(s)}_t) - \widetilde{F}(x^{(s)}_t)}^2 \mid \Fc^{(s)}_{t+1}} \vspace{1ex}\\
&& + {~}  \frac{M_{\phi}}{2\beta_d}\Exp{\Vert F'(x^{(s)}_t) - \widetilde{J}(x^{(s)}_t)\Vert^2 \mid \Fc^{(s)}_{t+1}} + M_{\phi}\xi_t^s,
\end{array}
\end{equation*}
for any $\xi_t^s > 0$, 
where we use $2ab \leq a^2 + b^2$ and the Jensen inequality $\left(\Exp{\norms{F(x^{(s)}_t) - \widetilde{F}(x^{(s)}_t)} \mid \Fc^{(s)}_{t+1}}\right)^2 \leq \Exp{\norms{F(x^{(s)}_t) - \widetilde{F}(x^{(s)}_t)}^2 \mid \Fc^{(s)}_{t+1}}$ in the second line.
Taking the full expectation both sides of the last inequality, and noting that $\Psi(x) = \phi(F(x))$, we obtain
\begin{equation}\label{eq:sarah_proof1}
\begin{array}{lcl}
\Exp{\Psi(x^{(s)}_{t+1})} & \leq &   \Exp{\Psi(x_t^{(s)})} - \frac{C_g}{2}\Exp{\norms{x^{(s)}_{t+1} - x^{(s)}_t}^2}  + \frac{L_\phi}{\xi_t^s} \Exp{\norms{F(x^{(s)}_t) - \widetilde{F}(x^{(s)}_t)}^2} + M_{\phi}\xi_t^s   \vspace{1ex}\\
&& + {~}  \frac{M_{\phi}}{2\beta_d}\Exp{\Vert F'(x^{(s)}_t) - \widetilde{J}(x^{(s)}_t)\Vert^2},
\end{array}{\!\!\!\!\!\!\!}
\end{equation}
where $C_g := 2M - M_{\phi}(L_F + \beta_d) > 0$, and $\beta_d > 0$ and $\xi_t^s > 0$ are given.

Next, from Lemma~\ref{le:sarah_estimators}, using the Lipschitz continuity of $F'$ in Assumption~\ref{as:A2}, we have
\begin{equation}\label{eq:sarah_proof3}
\Exp{\norms{\widetilde{J}^{(s)}_t - F'(x^{(s)}_t)}^2} \leq \Exp{\norms{\widetilde{J}^{(s)}_{t-1} - F'(x^{(s)}_{t-1})}^2} + \frac{L_F^2}{\hat{b}^{(s)}_t}\Exp{\norms{x^{(s)}_t - x^{(s)}_{t-1}}^2}.
\end{equation}
Similarly, using Lemma~\ref{le:sarah_estimators}, we also have
\begin{equation*} 
\Exp{\norms{\widetilde{F}_t^{(s)} - F(x^{(s)}_t)}^2 \mid \Fc^{(s)}_{t+1}} \leq \norms{\widetilde{F}^{(s)}_{t-1} - F(x^{(s)}_{t-1})}^2 +  \frac{1}{b_t}\Exps{\xi}{\norm{\Fb(x_t, \xi) - \Fb(x_{t-1},\xi) }^2}.
\end{equation*}
Taking the full expectation both sides of this inequality, and using Assumption~\ref{as:A3}, we obtain
\begin{equation}\label{eq:sarah_proof2}
\Exp{\norms{\widetilde{F}_t^{(s)} - F(x^{(s)}_t)}^2 } \leq \Exp{\norms{\widetilde{F}^{(s)}_{t-1} - F(x^{(s)}_{t-1})}^2 } + \frac{M_F^2}{b^{(s)}_t}\Exp{\norms{x^{(s)}_t - x^{(s)}_{t-1}}^2}.
\end{equation}
Let us define a Lyapunov function as 
\begin{equation}\label{eq:Lyapunov_func}
\Lc(x^{(s)}_t) := \Exp{\Psi(x_t^{(s)})} + \frac{a_t^s}{2}\Exp{\norms{\widetilde{F}_t^{(s)} - F(x_t^{(s)}}^2} +  \frac{c_t^s}{2}\Exp{\norms{\widetilde{J}_t^{(s)} - F'(x_t^{(s)}}^2},
\end{equation}
for some $a_t^s > 0$ and $c_t^s > 0$.

Combining \eqref{eq:sarah_proof1}, \eqref{eq:sarah_proof3}, and \eqref{eq:sarah_proof2}, and then using the definition of $\Lc$ in \eqref{eq:Lyapunov_func}, we have
\begin{equation}\label{eq:sarah_proof_new1}
\begin{array}{lcl}
\Lc(x^{(s)}_{t+1}) & = &  \Exp{\Psi(x_{t+1}^{(s)})} + \frac{a_{t+1}^s}{2}\Exp{\norms{\widetilde{F}_{t+1}^{(s)} - F(x_{t+1}^{(s)}}^2} +  \frac{c_{t+1}^s}{2}\Exp{\norms{\widetilde{J}_{t+1}^{(s)} - F'(x_{t+1}^{(s)}}^2} \vspace{1ex}\\
& \leq &  \Exp{\Psi(x_t^{(s)})}  - \left[\frac{C_g}{2} - \frac{M_F^2a_{t+1}^s}{2b_{t+1}^{(s)}} - \frac{L_F^2c_{t+1}^s}{2\hat{b}_{t+1}^{(s)}} \right]\Exp{\norms{x^{(s)}_{t+1} - x^{(s)}_t}^2}   + M_{\phi}\xi_t^s   \vspace{1ex}\\
&& + {~} \left(\frac{a_{t+1}^s}{2} + \frac{L_\phi}{\xi_t^s}\right) \Exp{\norms{F(x^{(s)}_t) - \widetilde{F}(x^{(s)}_t)}^2}   + \left(\frac{c_{t+1}^s}{2} + \frac{M_{\phi}}{2\beta_d}\right)\Exp{\Vert F'(x^{(s)}_t) - \widetilde{J}(x^{(s)}_t)\Vert^2}.
\end{array}
\end{equation}
If we assume that 
\begin{equation}\label{eq:condition3}
a^s_t \geq a^s_{t+1} + \frac{M_{\phi}}{\xi_t^s}~~~~~\text{and}~~~~~c^s_t \geq c^s_{t+1} + \frac{M_{\phi}}{\beta_d},
\end{equation}
then, from \eqref{eq:sarah_proof_new1}, we have
\begin{equation}\label{eq:sarah_proof_new2}
\Lc(x^{(s)}_{t+1}) \leq \Lc(x^{(s)}_t)  - \frac{\rho_{t+1}^s}{2}\Exp{\norms{x^{(s)}_{t+1} - x^{(s)}_t}^2}   + M_{\phi}\xi_t^s,
\end{equation}
where $\rho_{t+1}^s := C_g - \frac{M_F^2a_{t+1}^s}{b_{t+1}^{(s)}} - \frac{L_F^2c_{t+1}^s}{\hat{b}_{t+1}^{(s)}}$.

Let us first fix $\xi^s_t := \xi > 0$.
Next, we choose  $a^s_t := (m+1 - t)\frac{M_{\phi}}{\xi}$ and $c^s_t := (m+1-t)\frac{M_{\phi}}{\beta_d}$.
Clearly, $a^s_{m+1} = c^s_{m+1} = 0$ and they both satisfy the condition \eqref{eq:condition3}.
Then, we choose $b^{(s)}_t := \frac{1}{\gamma_1}a^s_t = (m+1 - t)\frac{M_{\phi}}{\gamma_1\xi}$ and  $\hat{b}^{(s)}_t = \frac{1}{\gamma_2}c^s_t = \frac{M_{\phi}}{\beta_d\gamma_2}(m+1 - t)$ for some $\gamma_1 > 0$ and $\gamma_2 > 0$.
In this case, we have $\rho^s_t = C_g - M_F^2\gamma_1 - L_F^2 \gamma_2  \equiv \theta_F > 0$  due to  \eqref{eq:rho_quantity} by appropriately choosing $\gamma_1$ and $\gamma_2$.
Consequently, \eqref{eq:sarah_proof_new2} reduces to
\begin{equation*} 
\Lc(x^{(s)}_{t+1}) \leq \Lc(x^{(s)}_t)  - \frac{\theta_F}{2}\Exp{\norms{x^{(s)}_{t+1} - x^{(s)}_t}^2}   + M_{\phi}\xi.
\end{equation*}
Summing up this inequality from $t=0$ to $t=m$, we obtain
\begin{equation*} 
\frac{\theta_F}{2}\sum_{t=0}^m\Exp{\norms{x^{(s)}_{t+1} - x^{(s)}_t}^2}  \leq \Lc(x^{(s)}_0) - \Lc(x^{(s)}_{m+1}) + (m+1)M_{\phi}\xi.
\end{equation*}
Using the fact that $\widetilde{x}^{s-1} = x^{(s)}_0$ and $\widetilde{x}^{s} = x^{(s)}_{m+1}$, we have
\begin{equation*} 
\frac{\theta_F}{2}\sum_{t=0}^m\Exp{\norms{x^{(s)}_{t+1} - x^{(s)}_t}^2}  \leq \Lc(\widetilde{x}^{s-1}) - \Lc(\widetilde{x}^{s}) + (m+1)M_{\phi}\xi.
\end{equation*}
Summing up this inequality from $s=1$ to $S$ and multiplying the result by $\frac{2}{\theta_F S(m+1)}$, we obtain
\begin{equation}\label{eq:sarah_proof_new5}
\frac{1}{S(m+1)}\sum_{s=1}^S\sum_{t=0}^m\Exp{\norms{x^{(s)}_{t+1} - x^{(s)}_t}^2}  \leq \frac{2\big[\Lc(\widetilde{x}^{0}) - \Lc(\widetilde{x}^{S})\big]}{\theta_F S(m+1)} + \frac{2M_{\phi}\xi}{\theta_F}.
\end{equation}
Since $\Lc(\widetilde{x}^{0}) = \Psi(\widetilde{x}^0) + \frac{(m+1)M_{\phi}}{2\xi}\Exp{\norms{\widetilde{F}_0 - F(\widetilde{x}^0)}^2} + \frac{(m+1)M_{\phi}}{2\beta_d}\Exp{\norms{\widetilde{J}_0 - F'(\widetilde{x}^0)}^2}$ and $\Lc(\widetilde{x}^{S}) = \Exp{\Psi(\widetilde{x}^{S}} \geq \Phi^{\star}$, we obtain from \eqref{eq:sarah_proof_new5} that
\begin{align}\label{eq:sarah_proof_new6}
\frac{1}{S(m+1)}\sum_{s=1}^S\sum_{t=0}^m\Exp{\norms{x^{(s)}_{t+1} - x^{(s)}_t}^2}  &\leq \frac{2\big[\Psi(\widetilde{x}^{0}) - \Psi^{\star}\big]}{\theta_F S(m+1)} + \frac{M_{\phi}}{\xi\theta_F S}\Exp{\norms{\widetilde{F}_0 - F(\widetilde{x}^0)}^2} \nonumber\\
& + \frac{M_{\phi}}{\theta_F \beta_dS}\Exp{\norms{\widetilde{J}_0 - F'(\widetilde{x}^0)}^2} + \frac{2M_{\phi}\xi}{\theta_F}.
\end{align}
Note that $\Exp{\norms{\widetilde{F}_0 - F(\widetilde{x}^0)}^2} \leq \frac{\sigma_F^2}{b}$ and $\Exp{\norms{\widetilde{J}_0 - F'(\widetilde{x}^0)}^2} \leq \frac{\sigma_D^2}{\hat{b}}$ due to the choice of $b_s = b > 0$ and $\hat{b}_s = \hat{b} > 0$ at Step~\ref{step:o2} of Algorithm~\ref{alg:A2}.
Hence, we can further bound \eqref{eq:sarah_proof_new6} as
\begin{align*} 
\frac{1}{S(m+1)}\sum_{s=1}^S\sum_{t=0}^m\Exp{\norms{x^{(s)}_{t+1} - x^{(s)}_t}^2}  \leq \frac{2\big[\Psi(\widetilde{x}^{0}) - \Psi^{\star}\big]}{\theta_F S(m+1)} + \frac{M_{\phi}\sigma_F^2}{\xi\theta_F Sb} + \frac{M_{\phi}\sigma_D^2}{\theta_F\beta_dS\hat{b}} + \frac{2M_{\phi}\xi}{\theta_F}.
\end{align*}
Since $\norms{\widetilde{G}_M(x^{(s)}_t)} = M\norms{x^{(s)}_{t+1} - x_t^{(s)}}$, to guarantee $\frac{1}{S(m+1)}\sum_{s=1}^S\sum_{t=0}^m\Exp{\norms{\widetilde{G}_M(x^{(s)}_t)}^2} \leq \varepsilon^2$ for a given tolerance $\varepsilon > 0$, we need to set 
\begin{equation*}
\frac{2\big[\Psi(\widetilde{x}^{0}) - \Psi^{\star}\big]}{\theta_F S(m+1)} + \frac{M_{\phi}\sigma_F^2}{\xi \theta_F Sb} + \frac{M_{\phi}\sigma_D^2}{\theta_F \beta_dS\hat{b}} + \frac{2M_{\phi}\xi}{\theta_F} = \frac{\varepsilon^2}{M^2}.
\end{equation*}
Let us break this condition into
\begin{equation*}
\frac{2\big[\Psi(\widetilde{x}^{0}) - \Psi^{\star}\big]}{\theta_F S(m+1)} = \frac{\varepsilon^2}{4M^2} ~~~~\text{and}~~~~\frac{M_{\phi}\sigma_F^2}{\xi\theta_F Sb} =  \frac{M_{\phi}\sigma_D^2}{\theta_F \beta_dS\hat{b}} = \frac{2M_{\phi}\xi}{\theta_F} = \frac{\varepsilon^2}{4M^2}.
\end{equation*}
Hence, we can choose $\xi := \frac{\theta_F\varepsilon^2}{8M^2M_{\phi}}$, $\hat{b} := \frac{4M_{\phi}\sigma_D^2}{\theta_F\beta_dM^2S\varepsilon^2}$, $b := \frac{2M_{\phi}^2\sigma_F^2}{\theta_F^2M^2 S\varepsilon^4}$, and $S(m+1) = \frac{8M^2\big[\Psi(\widetilde{x}^{0}) - \Psi^{\star}\big]}{\theta_F\varepsilon^2}$.

Now, let us choose $m + 1 := \frac{\hat{C}}{\varepsilon}$ for some constant $\hat{C} > 0$.
Then, we can estimate the total number $\Tc_f$ of stochastic function evaluations $\Fb(x_t^{(s)},\xi)$  as follows:
\begin{equation*} 
\begin{array}{lcl}
\Tc_f & :=  & \sum_{s=1}^Sb_s + \sum_{s=1}^S\sum_{t=0}^mb^{(s)}_t = Sb + \frac{M_{\phi}}{\gamma_1\xi}\sum_{s=1}^S\sum_{t=0}^m(m+1-t)\vspace{1ex}\\
& = &  \frac{2M_{\phi}^2\sigma_F^2}{\theta_F^2M^2 \varepsilon^4} + \frac{8M^2M_{\phi}^2}{\gamma_1 \theta_F \varepsilon^2}\cdot \frac{S(m+1)(m+2)}{2} \vspace{1ex}\\
& = &  \frac{2M_{\phi}^2\sigma_F^2}{\theta_F^2M^2 \varepsilon^4} + \frac{8M^2M_{\phi}^2}{\gamma_1 \theta_F\varepsilon^2}\cdot \frac{8M^2\left[\Psi(\widetilde{x}^{0}) - \Psi^{\star}\right]}{\theta_F \varepsilon^2} \cdot \frac{\hat{C}+\varepsilon}{2\varepsilon}\vspace{1ex}\\
& = & \BigO{\frac{M_{\phi}^2\sigma_F^2}{\theta_F^2\varepsilon^4} + \frac{M^4M_{\phi}^2\big[\Psi(\widetilde{x}^{0}) - \Psi^{\star}\big]}{\theta_F^2\varepsilon^5}}.
\end{array}
\end{equation*}
Similarly, the total number $\Tc_d$ of stochastic Jacobian evaluations $\Fb'(x_t^{(s)}, \zeta)$ can be bounded as
\begin{equation*}
\begin{array}{lcl}
\Tc_d & := &  \sum_{s=1}^S\hat{b}_s  + \sum_{s=1}^S\sum_{t=0}^m\hat{b}_t^{(s)} = S\hat{b} +  \frac{M_{\phi}S}{\beta_d\gamma_2}\sum_{t=0}^m (m+1 - t) \vspace{1ex}\\
& \leq & \frac{4M_{\phi}\sigma_D^2}{\theta_F \beta_dM^2\varepsilon^2} + \frac{8M^2M_{\phi}\left[\Psi(\widetilde{x}^{0}) - \Psi^{\star}\right]}{\beta_d\gamma_2 \theta_F \varepsilon^2} \cdot \frac{\hat{C}+\varepsilon}{2\varepsilon} \vspace{1ex}\\
& = & \BigO{\frac{M_{\phi}\sigma_D^2}{\theta_F \varepsilon^2} + \frac{M^2M_{\phi} \big[\Psi(\widetilde{x}^{0}) - \Psi^{\star}\big]}{\theta_F \varepsilon^3}}.
\end{array} 
\end{equation*}
Hence, taking the upper bounds, we have proven \eqref{eq:total_or_complexity}.
\end{proof}

\beforesec
\section{Solution Routines for Computing Gauss-Newton Search Directions}\label{apdx:sec:subsolver}
\aftersec
One main step of SGN methods is to compute the Gauss-Newton direction by solving the subproblem \eqref{eq:surg_model}.
This subproblem is also called a \textbf{prox-linear operator}, which can be rewritten as
\begin{equation}\label{eq:subprob}
\min_{d\in\R^p}\set{ \phi(\widetilde{F}_t + \widetilde{J}_td) + \hat{g}(d) + \tfrac{M}{2}\norms{d}_2^2 },
\end{equation}
where $\widetilde{F}_t \approx F(x_t)$, $\widetilde{J}_t\approx F'(x_t)$, $d := x - x_t$, $\phi$ is convex, $\hat{g}(d) := g(x_t + d)$, and $M > 0$ is given.
This is a basic convex problem, and we can apply different methods to solve it.
Here, we describe two methods for solving \eqref{eq:subprob}.

\beforesubsec
\subsection{Accelerated Dual Proximal-Gradient Method}\label{apdx:subsec:ADPGM}
\aftersubsec
For accelerated dual proximal-gradient method, we consider the case $\hat{g}(d)= 0$ for simplicity.
Using Fenchel's conjugate of $\phi$, we can write $\phi(\widetilde{F}_t + \widetilde{J}_td) = \max\set{ \iprods{\widetilde{F}_t + \widetilde{J}_td, u} - \phi^{*}(u)}$.
Assume that strong duality holds for \eqref{eq:subprob}, then using this expression, we can write it as
\begin{equation*}
\min_{d}\max_{u}\Big\{ \iprods{\widetilde{F}_t + \widetilde{J}_td, u} - \phi^{*}(u) + \frac{M}{2}\norms{d}_2^2\Big\} ~~~\Leftrightarrow~~~
\max_{u}\Big\{ \min_{d}\Big\{ \iprods{\widetilde{F}_t + \widetilde{J}_td, u}  + \frac{M}{2}\norms{d}_2^2\Big\} - \phi^{*}(u) \Big\}.
\end{equation*}
Solving the inner problem $\min_{d}\Big\{ \iprods{\widetilde{F}_t + \widetilde{J}_td, u}  + \frac{M}{2}\norms{d}_2^2\Big\}$, we obtain $d^{*}(u) := -\frac{1}{M}\widetilde{J}^{\top}u$.
Substituting it into the objective, we eventually obtain the dual problem as follows:
\begin{equation}\label{eq:dual_subprob}
\min_{u}\Big\{ \frac{1}{2M}\norms{\widetilde{J}_t^{\top}u}^2_2 - \iprods{\widetilde{F}_t, u} + \phi^{*}(u) \Big\}.
\end{equation}
We can solve this problem by an accelerated proximal-gradient method \cite{Beck2009,Nesterov2004}, which is described as follows.

\begin{algorithm}[ht!]\caption{(Accelerated Dual Proximal-Gradient~\textbf{(ADPG)})}\label{alg:subA1}
\normalsize
\begin{algorithmic}[1]
   \STATE{\bfseries Initialization:} Choose $u_0\in\R^m$. Set $\tau_0 := 1$ and $\hat{u}_0 := u_0$. Evaluate $L := \frac{1}{M}\norms{\widetilde{J}_t^{\top}\widetilde{J}_t}$.
   \STATE\hspace{0ex}\label{step:o1}{\bfseries For $k := 0,\cdots, k_{\max}$ do}
   \STATE\hspace{2ex}\label{step:o2}  $u_{k+1} := \prox_{(1/L)\phi^{*}}\left(\hat{u}_k - \frac{1}{L}(\frac{1}{M}\widetilde{J}_t\widetilde{J}_t^{\top}\hat{u}_k - \widetilde{F}_t)\right)$.
   \STATE\hspace{2ex}\label{step:o2} $\tau_{k+1} := \frac{1 + \sqrt{1 + 4\tau_k^2}}{2}$.
   \STATE\hspace{2ex}\label{step:o3} $\hat{u}_{k+1} := u_{k+1} + \left(\frac{\tau_k-1}{\tau_{k+1}}\right)(u_{k+1} - u_k)$.
   \vspace{0.5ex}   
   \STATE\hspace{0ex}{\bfseries End For}
   \STATE\hspace{0ex}{\bfseries Output:} Reconstruct $d^{*} := -\frac{1}{M}\widetilde{J}_t^{\top}u_{k}$ as an approximate solution of \eqref{eq:subprob}.
\end{algorithmic}
\end{algorithm}

Note that in Algorithm~\ref{alg:subA1}, we use the proximal operator $\prox_{\lambda\phi^{*}}$ of $\phi^{*}$.
However, by Moreau's identity, $\prox_{\lambda\phi^{*}}(v) + \lambda\prox_{\phi/\lambda}(v/\lambda) = v$, we can again use the proximal operator $\prox_{\phi/\lambda}$ of $\phi$.

\beforesubsec
\subsection{Primal-Dual First-Order Methods}\label{apdx:subsec:PD_method}
\aftersubsec
We can apply any primal-dual algorithm from the literature \cite{Bauschke2011,Chambolle2011,Esser2010a,Goldstein2013,TranDinh2015b,tran2017proximal} to solve \eqref{eq:subprob}.
Here, we describe the well-known Chambolle-Pock's primal-dual method \cite{Chambolle2011} to solve \eqref{eq:subprob}.

Let us define $\hat{\phi}(z) := \phi(z + F_k)$ and $\hat{\psi}(d) := \hat{g}(d) + \frac{M}{2}\norms{d}^2$.
Since \eqref{eq:subprob} is strongly convex with the strong convexity parameter $\mu_{\hat{\psi}} := M$, we can apply the strongly convex primal-dual variant as follows.

Choose $\sigma_0 > 0$ and $\tau_0 > 0$ such that $\tau_0\sigma_0 \leq \frac{1}{\norms{\widetilde{J}_t^{\top}\widetilde{J}_t}}$.
For example, we can choose $\sigma_0 = \tau_0 = \frac{1}{\norms{\widetilde{J}_t}}$, or we choose $\sigma_0 > 0$ first, and choose $\tau_0 := \frac{1}{\sigma_0\norms{\widetilde{J}_t^{\top}\widetilde{J}_t}}$.
Choose $d_0\in\R^p$ and $u_0\in\R^m$ and set $\bar{d}_0 := d_0$.
Then, at each iteration $k\geq 0$, we update
\begin{equation}\label{eq:pd_scheme}
\left\{\begin{array}{lcl}
u_{k+1} & := & \prox_{\sigma_k\hat{\phi}^{*}}\left(u_k + \sigma_k\widetilde{J}_t\bar{d}_k \right), \vspace{1ex}\\
d_{k+1} & := & \prox_{\tau_k\hat{\psi}}\left(d_k - \tau_k\widetilde{J}_t^{\top}u_{k+1}\right), \vspace{1ex}\\
\theta_k & := & 1/\sqrt{1 + 2M\tau_k}, \vspace{1ex}\\
\tau_{k+1} & := & \theta_k\tau_k, \vspace{1ex}\\
\sigma_{k+1} & := & \sigma_k/\theta_k, \vspace{1ex}\\
\bar{d}_{k+1} & := & d_{k+1} + \theta_k(d_{k+1} - d_k).
\end{array}\right.
\end{equation}
Alternatively to the Accelerated Dual Proximal-Gradient  and the primal-dual methods, we can also apply the alternating direction method of multipliers (ADMM) to solve \eqref{eq:subprob}.
However, this method requires to solve a linear system, that may not scale well when the dimension $p$ is large.

\beforesec
\section{Details of The Experiments in Section~\ref{sec:num_exp}}\label{sec:apdx:add_experiments}
\aftersec
In this supplementary document, we provide the details of our experiments in Section \ref{sec:num_exp}, including modeling, data generating routines, and experiment configurations.
We also provide more experiments for both examples.
All algorithms are implemented in Python 3.6 running on a Macbook Pro with 2.3 GHz Quad-Core, 8 GB RAM and on a Linux-based computing node, called Longleaf, where each node has 24 physical cores, 2.50 GHz processors, and 256 GB RAM.

\beforesubsec
\subsection{Stochastic Nonlinear Equations}\label{subsec:exm1}
\aftersubsec
Our goal is to solve the following nonlinear equation in expectation as described in Subsection~\ref{subsec:parameter_estimation}:
\begin{equation}\label{eq:exam1_b}
F(x) = 0,~~\text{where}~~F(x) := \Exps{\xi}{\Fb(x, \xi)}.
\end{equation}
Here, $\Fb$ is a stochastic vector function from $\R^p\times \Omega\to\R^q$.
As discussed in the main text, \eqref{eq:exam1_b} covers the first-order optimality condition $\Exps{\xi}{\nabla_x\mathbf{G}(x, \xi)} = 0$ of a stochastic optimization problem $\min_x\Exps{\xi}{\mathbf{G}(x, \xi)}$ as a special case.
More generally, it also covers the KKT condition of a stochastic optimization problem with equality constraints.
However, these problems may not have stationary point, which leads to an inconsistency of \eqref{eq:exam1_b}.
As a remedy, we can instead consider 
\begin{equation}\label{eq:exam1_b2}
\min_{x}\set{ \Psi(x) := \norms{\Exps{\xi}{\Fb(x, \xi)}}},
\end{equation}
for a given norm $\norms{\cdot}$ (e.g., $\ell_1$-norm or $\ell_2$-norm).
Problem \eqref{eq:exam1_b} also covers the expectation formulation of stochastic nonlinear equations such as stochastic ODEs or PDEs.

In our experiment from Subsection~\ref{subsec:parameter_estimation}, we only consider one instance of \eqref{eq:exam1_b2} by choosing $q = 4$ and $\Fb_j$ ($j=1,\cdots, q$) as
\begin{equation}\label{eq:Fi_func}
\left\{ \begin{array}{lcl}
 \Fb_1(x, \xi_i) & := & (1 - \tanh(y_i(a_i^{\top}x + b_i)), \vspace{1ex}\\
 \Fb_2(x, \xi_i) & := &  \left(1 - \frac{1}{1 + \exp(- y_i(a_i^{\top}x + b_i))}\right)^2, \vspace{1ex}\\
 \Fb_3(x, \xi_i) & := & \log(1 + \exp({-y_i(a_i^{\top}x + b_i)})) - \log(1 + \exp({-y_i(a_i^{\top}x + b_i)-1})), \vspace{1ex}\\
 \Fb_4(x, \xi_i) & := & \log(1 +(y_i(a_i^{\top}x + b_i) - 1)^2),
\end{array} \right.
\end{equation}
where $a_i$ is the  $i$-row  of an input matrix $A \in\R^{n\times p}$, $y \in \set{-1, 1}^n$ is a vector of labels, $b \in \R^n$ is a bias vector in binary classification, and $\xi_i := (a_i, b_i, y_i)$.
Note that the binary classification problem with nonconvex loss has been widely studied in the literature, including \citet{zhao2010convex}, where one aims at solving:
\begin{equation}\label{eq:binary_class}
\min_{x\in\R^p}\set{ H(x) := \frac{1}{n}\sum_{i=1}^n\ell(y_i(a_i^Tx + b_i))},
\end{equation}
for a given loss function $\ell$.
If $\ell$ is nonnegative, then instead of solving \eqref{eq:binary_class}, we can solve $\min_{x}\vert H(x)\vert$.
If we have $q$ different losses $\ell_j$ for $j=1,\cdots, q$ and we want to solve $q$ problems of the form \eqref{eq:binary_class} for different losses simultaneously, then we can formulate such a problem into \eqref{eq:exam1_b2} to have $\min_x\norms{\Hb(x)}$, where $\Hb(x) := (H_1(x), H_2(x), \cdots, H_q(x))^{\top}$.
Since we use different losses, under the formulation \eqref{eq:exam1_b2}, we can view it as a binary classification task with an averaging loss.

\begin{table*}[hpt!]
\caption{Hyper-parameter  configurations for the two algorithms on all datasets when using the $\Vert\cdot\Vert_2$ loss.}
\label{tab:parameters_l2}
\vspace{2ex}
\centering
\resizebox{1\textwidth}{!}{%
\begin{tabular}{|c|c|c|c|c|c|c|c|c|c|c|c|c|}
\hline
\multirow{2}{*}{Algorithm} & \multicolumn{3}{c|}{w8a} & \multicolumn{3}{c|}{ijcnn1} & \multicolumn{3}{c|}{covtype} & \multicolumn{3}{c|}{url\_combined}  \\ \cline{2-13} 
 & $\hat{b}_t$ & $b_t$ & Inner Iterations & $\hat{b}_t$ & $b_t$ & Inner Iterations & $\hat{b}_t$ & $b_t$ & Inner Iterations & $\hat{b}_t$ & $b_t$ & Inner Iterations \\ \hline
SGN & 256 & 512 &  & 512 & 1,024 &  & 1,024 & 4,096 &  & 20,000 & 50,000 &  \\ \hline
SGN2 & 64 & 128 & 2,000 & 128 & 256 & 1,000 & 256 & 512 & 2000 & 5,000 & 10,000 & 5,000  \\ \hline
\multirow{2}{*}{} & \multicolumn{3}{c|}{a9a} & \multicolumn{3}{c|}{rcv1\_train.binary} & \multicolumn{3}{c|}{real-sim} & \multicolumn{3}{c|}{skin\_nonskin} \\ \cline{2-13} 
 & $\hat{b}_t$ & $b_t$ & Inner Iterations & $\hat{b}_t$ & $b_t$ & Inner Iterations & $\hat{b}_t$ & $b_t$ & Inner Iterations & $\hat{b}_t$ & $b_t$ & Inner Iterations \\ \hline
SGN & 512 & 1,024 &  & 512 & 1,024 &  & 1,024 & 4,096 & & 512 & 1024 &  \\ \hline
SGN2 & 64 & 128 & 2000 & 128 & 256 & 1,000 & 256 & 512 & 2,000 & 128 & 256 & 5,000 \\ \hline
\end{tabular}%
}
\end{table*}

\begin{table*}[hpt!]
\vspace{-2ex}
\caption{Hyper-parameter configurations for the four algorithms on 4 datasets when using the Huber loss.}
\label{tab:parameters_huber}
\vspace{2ex}
\centering
\resizebox{1\textwidth}{!}{%
\begin{tabular}{|c|c|c|c|c|c|c|c|c|c|c|c|c|}
\hline
\multirow{2}{*}{Algorithm} & \multicolumn{3}{c|}{w8a} & \multicolumn{3}{c|}{ijcnn1} & \multicolumn{3}{c|}{covtype} & \multicolumn{3}{c|}{url\_combined}  \\ \cline{2-13} 
 & $\hat{b}_t$ & $b_t$ & Inner Iterations & $\hat{b}_t$ & $b_t$ & Inner Iterations & $\hat{b}_t$ & $b_t$ & Inner Iterations & $\hat{b}_t$ & $b_t$ & Inner Iterations \\ \hline
SGN & 256 & 512 &  & 512 & 1,024 &  & 512 & 1,024 & & 20,000 & 50,000 &  \\ \hline
SCGD & 256 & 512 &   & 512 & 1,024 & & 512 & 1,024 &  & 20,000 & 50,000 &  \\ \hline
SGN2 & 64 & 128 & 5,000 & 128 & 256 & 2,000 & 128 & 256 & 5,000 & 5,000 & 10,000 & 5,000 \\ \hline
N-SPIDER & 64 & 128 & 5,000 & 128 & 256 & 2,000 & 128 & 256 & 5,000 & 5,000 & 10,000 & 5,000 \\ \hline
\multirow{2}{*}{} & \multicolumn{3}{c|}{a9a} & \multicolumn{3}{c|}{rcv1\_train.binary} & \multicolumn{3}{c|}{real-sim} & \multicolumn{3}{c|}{news20.binary} \\ \cline{2-13} 
 & $\hat{b}_t$ & $b_t$ & Inner Iterations & $\hat{b}_t$ & $b_t$ & Inner Iterations & $\hat{b}_t$ & $b_t$ & Inner Iterations & $\hat{b}_t$ & $b_t$ & Inner Iterations \\ \hline
SGN & 128 & 256 &  & 128 & 512 &  & 256 & 512 & & 128 & 512 &  \\ \hline
SCGD & 1,024 & 2,048 &   & 128 & 512 & & 256 & 512 &  & 128 & 512 &  \\ \hline
SGN2 & 64 & 128 & 2,000 & 64 & 128 & 5,000 & 64 & 128 & 5,000 & 64 & 128 & 5,000 \\ \hline
N-SPIDER & 64 & 128 & 2,000 & 64 & 128 & 5,000 & 64 & 128 & 5,000 & 64 & 128 & 5,000 \\ \hline
\end{tabular}%
}
\vspace{-2ex}
\end{table*}

\textbf{Datasets.}
We test three algorithms: GN, SGN, and SGN2 on four real datasets: \texttt{w8a} ($\boldsymbol{n=49,749;p=300}$), \texttt{ijcnn1} ($\boldsymbol{n=91,701;p=22}$), \texttt{covtype} ($\boldsymbol{n=581,012;p=54}$), and \texttt{url\_combined} ($\boldsymbol{n=2,396,130;p=3,231,961}$) from LIBSVM.

\textbf{Parameter configuration.}
We can easily check that $F$ defined by \eqref{eq:Fi_func} satisfies Assumption~\ref{as:A1} and Assumption~\ref{alg:A2}.
However, we do not accurately estimate the Lipschitz constant of $F'$ since it depends on the dataset.
We were instead experimenting with different choices of the parameter $M$ and $\rho$, and eventually fix $\rho := 1$ and $M := 1$ for our tests. We also choose the mini-batch sizes for both $\widetilde{F}$ and $\widetilde{J}$ in SGN and SGN2 by sweeping over the set of $\sets{ 64, 128, 256, 512,1024, 2048, 4096, 8192}$ to estimate the best ones. 
Table~\ref{tab:parameters_l2} presents the chosen parameters for the instance when $\phi = ||\cdot||_2$.

In the case of smooth $\phi$, i.e., using Huber loss, we add two competitors: N-SPIDER  \citep[Algorithm 3]{yang2019multilevel} and SCGD \citet[Algorithm 1]{wang2017stochastic}.
The learning rates of N-SPIDER and SCGD are tuned from a set of different values: $\{0.01,0.05,0.1,0.5,1,2\}$. Eventually we obtain $\eta:= 1.0$ and set $\varepsilon:=10^{-1}$ for N-SPIDER, see \citep[Algorithm 3]{yang2019multilevel}. 
For SCGD, we use $\beta_k:=1$ and $\alpha_k := 1$, see \citet[Algorithm 1]{wang2017stochastic}. The mini-batch sizes of these algorithm are chosen using similar search as in the previous case. 
Table~\ref{tab:parameters_huber} reveals the parameter configuration of the algorithms when using the Huber loss.

\begin{figure}[hpt!]
\begin{center}
    \includegraphics[width = 0.45\textwidth]{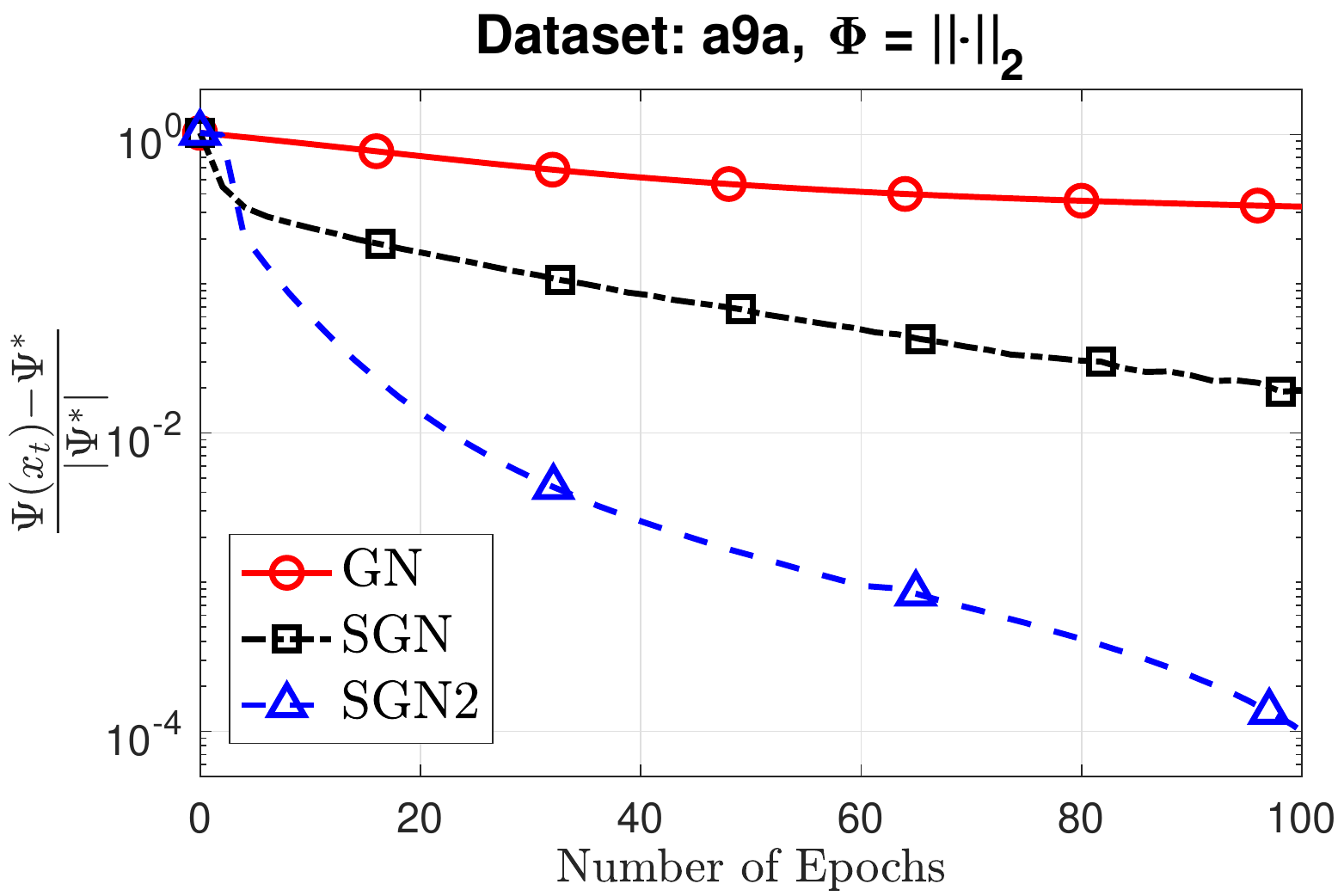}
    \includegraphics[width = 0.45\textwidth]{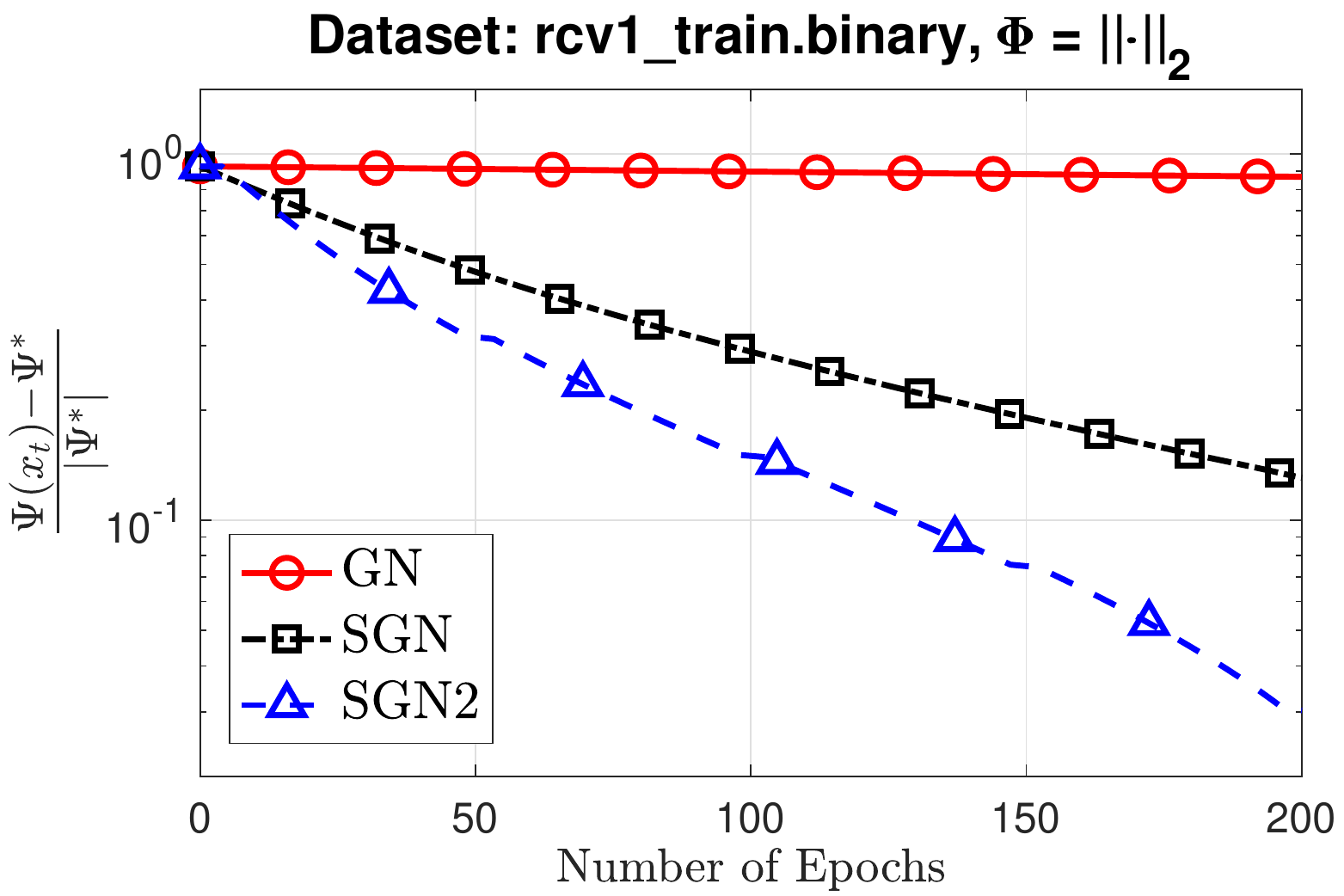}\vspace{2ex}\\
     \includegraphics[width = 0.45\textwidth]{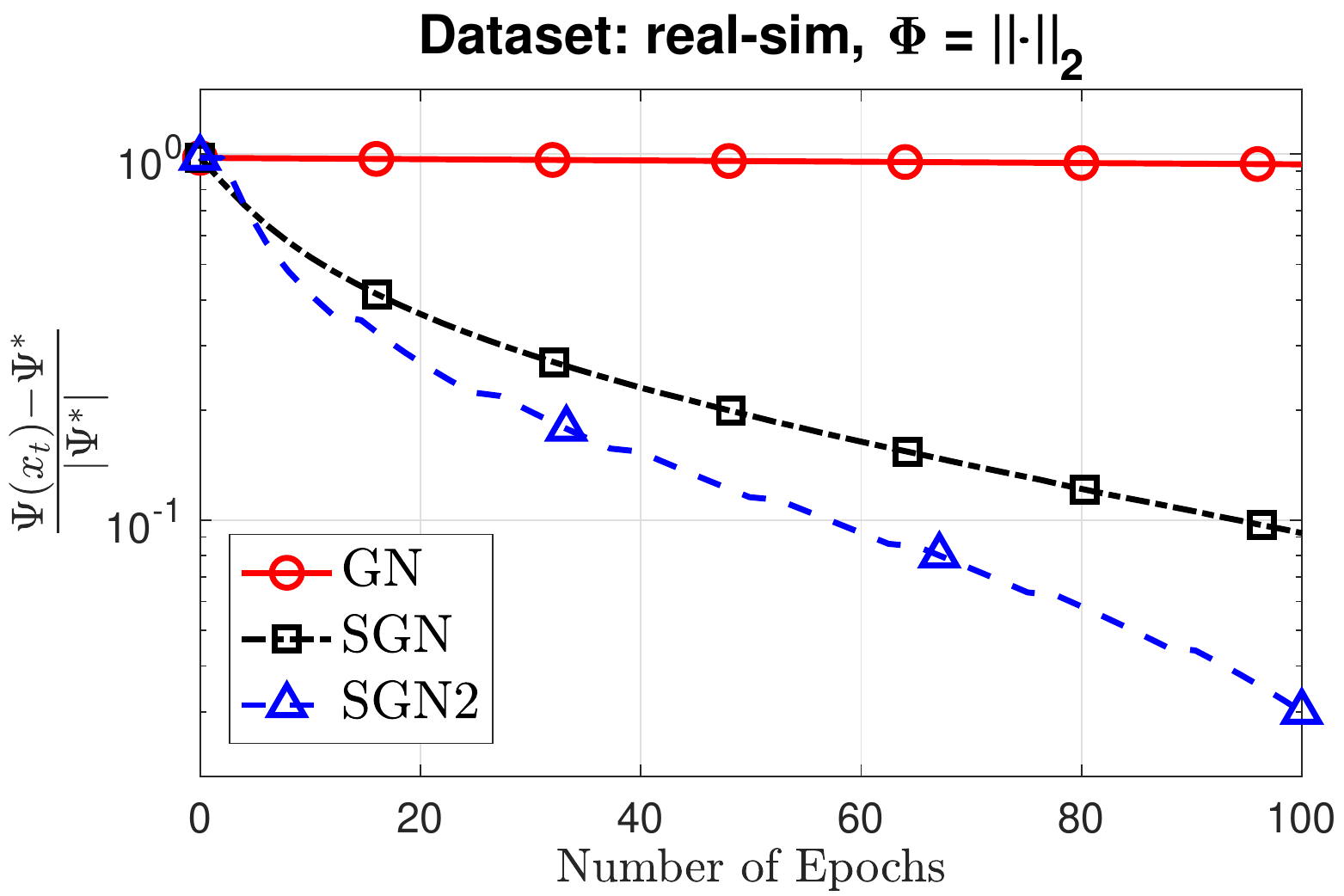}
     \includegraphics[width = 0.45\textwidth]{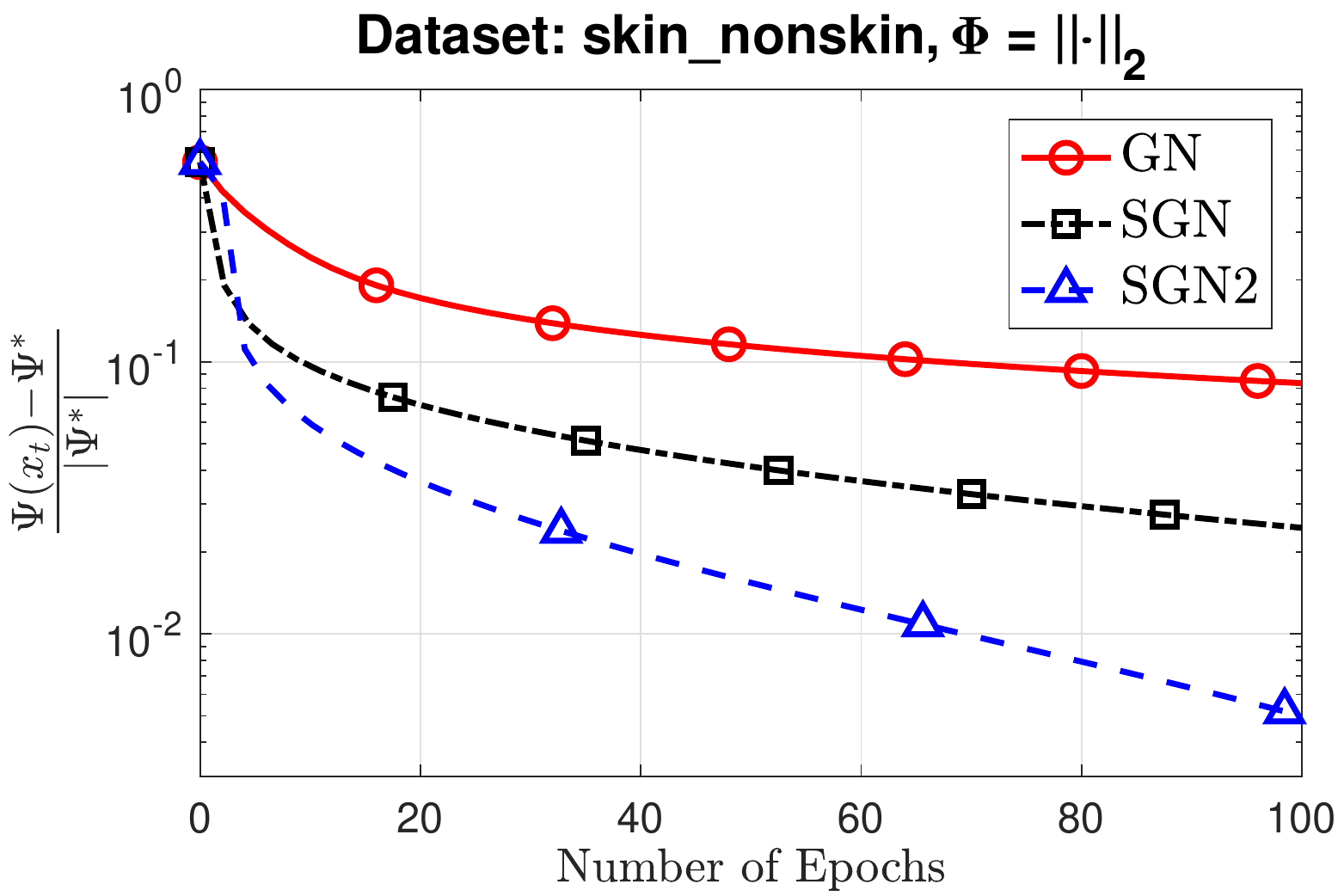}
    \caption{The performance of three algorithms on additional real datasets when $\phi(\cdot) = \norm{\cdot}_2$.}\label{fig:exp_1_add}
\end{center}
\end{figure}

\textbf{Additional Experiments.}
When $\phi(\cdot) = \norm{\cdot}_2$, we also run these algorithms on other classification datasets from LIBSVM: \texttt{a9a} ($\boldsymbol{n=32,561;p=123}$), \texttt{rcv1\_train.binary} ($\boldsymbol{n=20,242;p=47,236}$), \texttt{real-sim} ($\boldsymbol{n=72,309;p=20,958}$), and \texttt{skin\_nonskin} ($\boldsymbol{n=245,057;p=3}$). We set $M := 1$ and $\rho := 1$ for three datasets. Other parameters are obtained via grid search and the results are shown in Table~\ref{tab:parameters_l2}. The performance of three algorithms on these datasets are presented in Figure~\ref{fig:exp_1_add}.

SGN2 appears to be the best among the 3 algorithms while SGN is much better than the baseline GN. SGN appears to have advantage in the early stage but SGN2 makes better progress later on.

\begin{figure}[ht!]
\begin{center}
    \includegraphics[width = 0.45\textwidth]{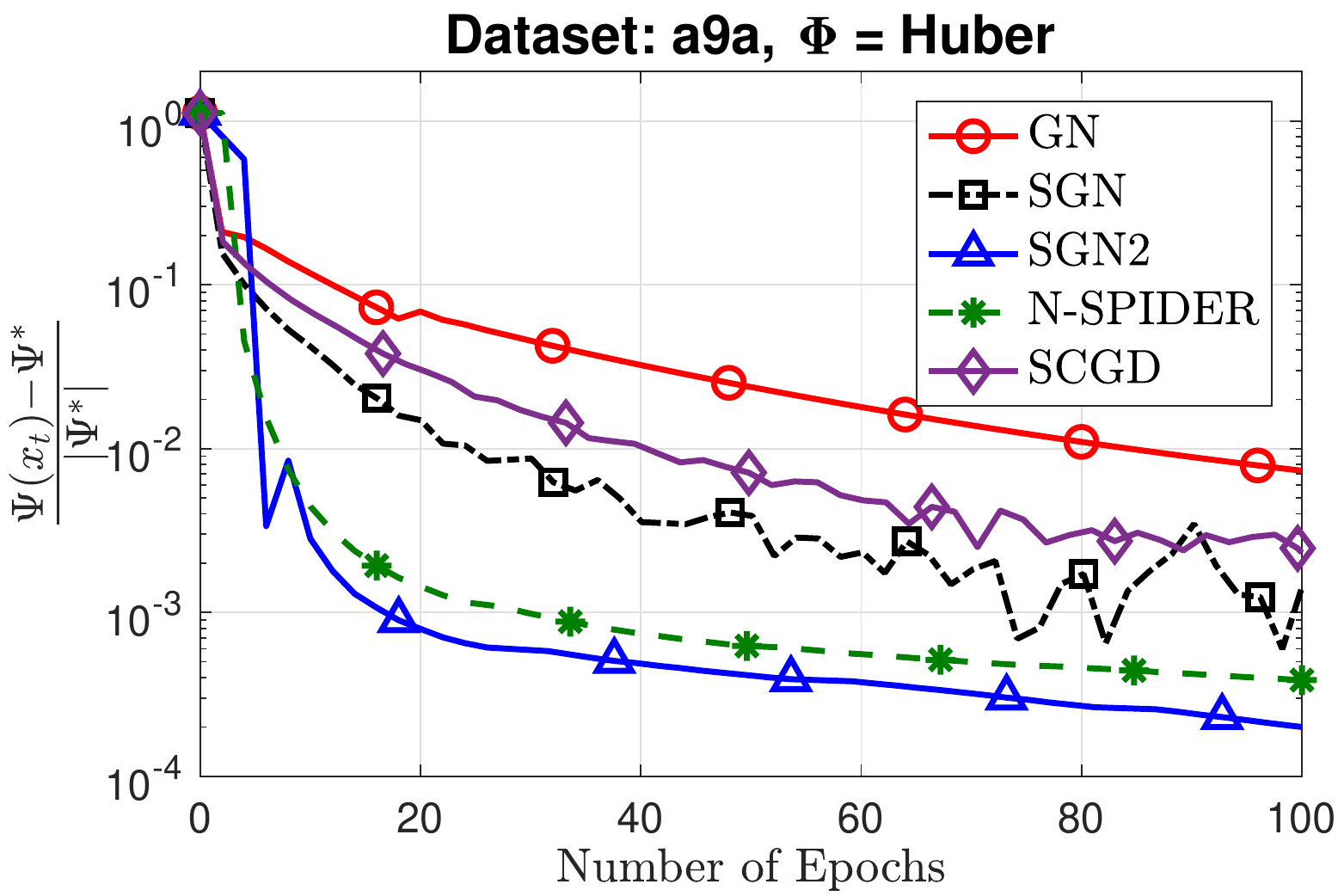}
    \includegraphics[width = 0.45\textwidth]{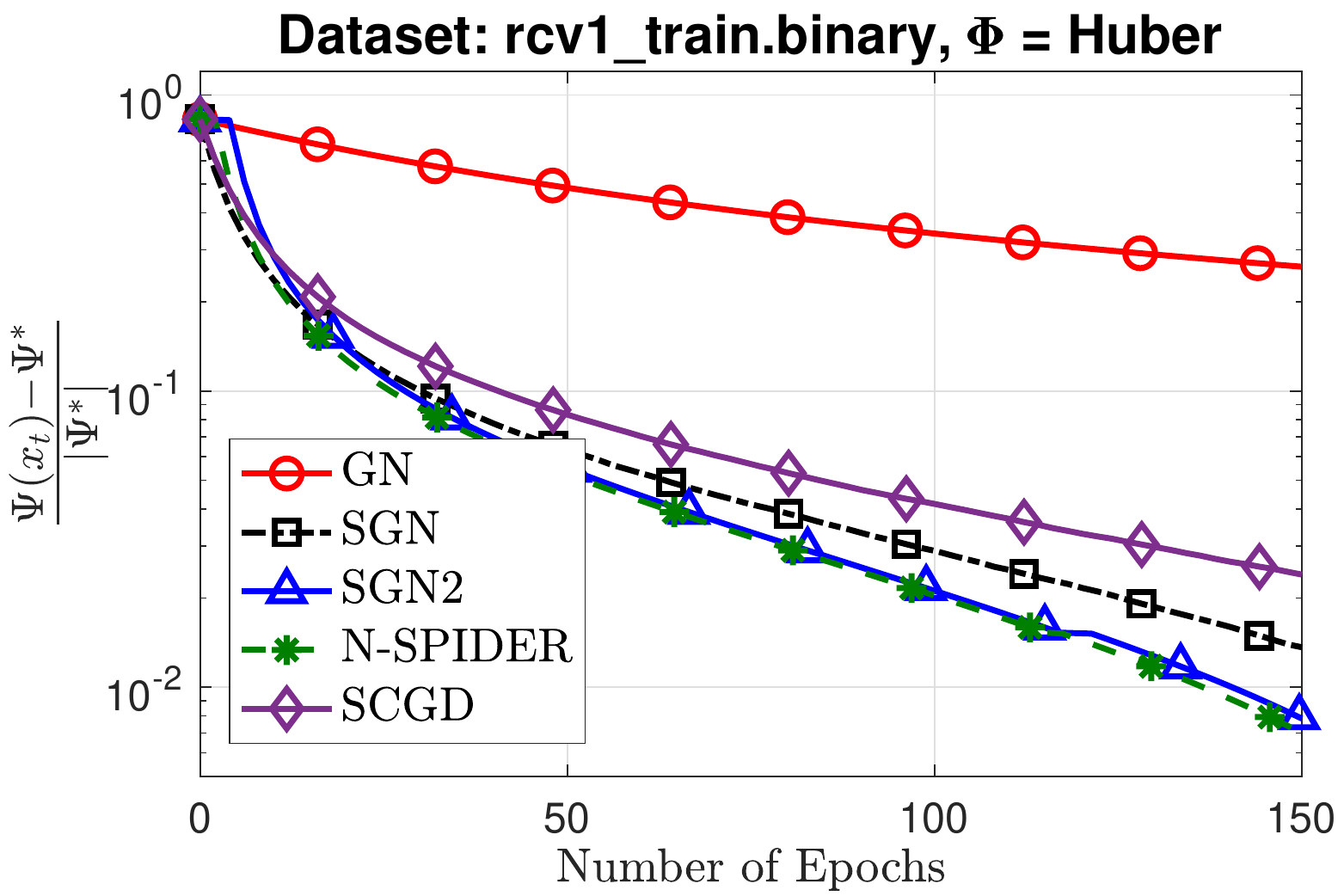}\vspace{2ex}\\
     \includegraphics[width = 0.45\textwidth]{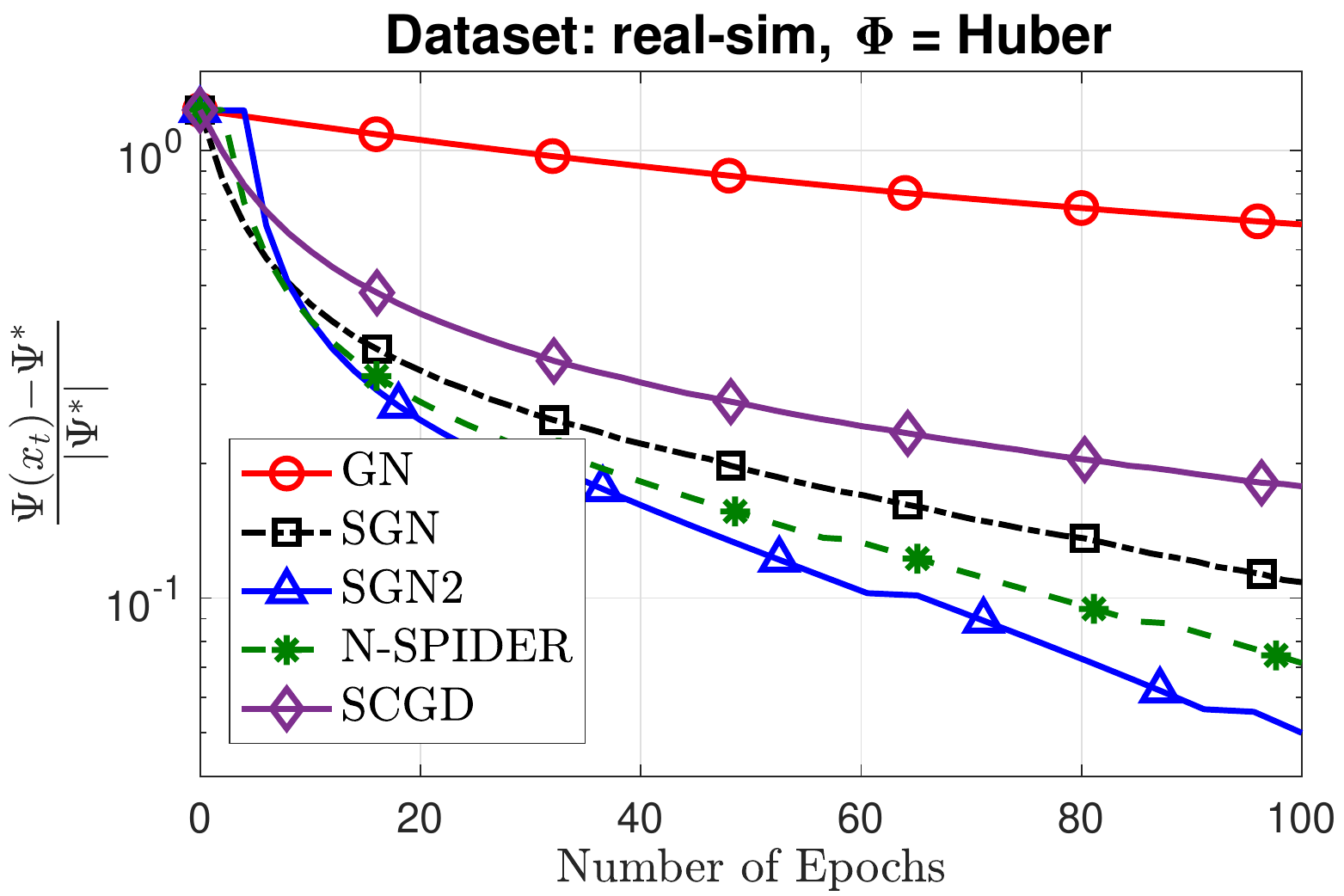}
     \includegraphics[width = 0.45\textwidth]{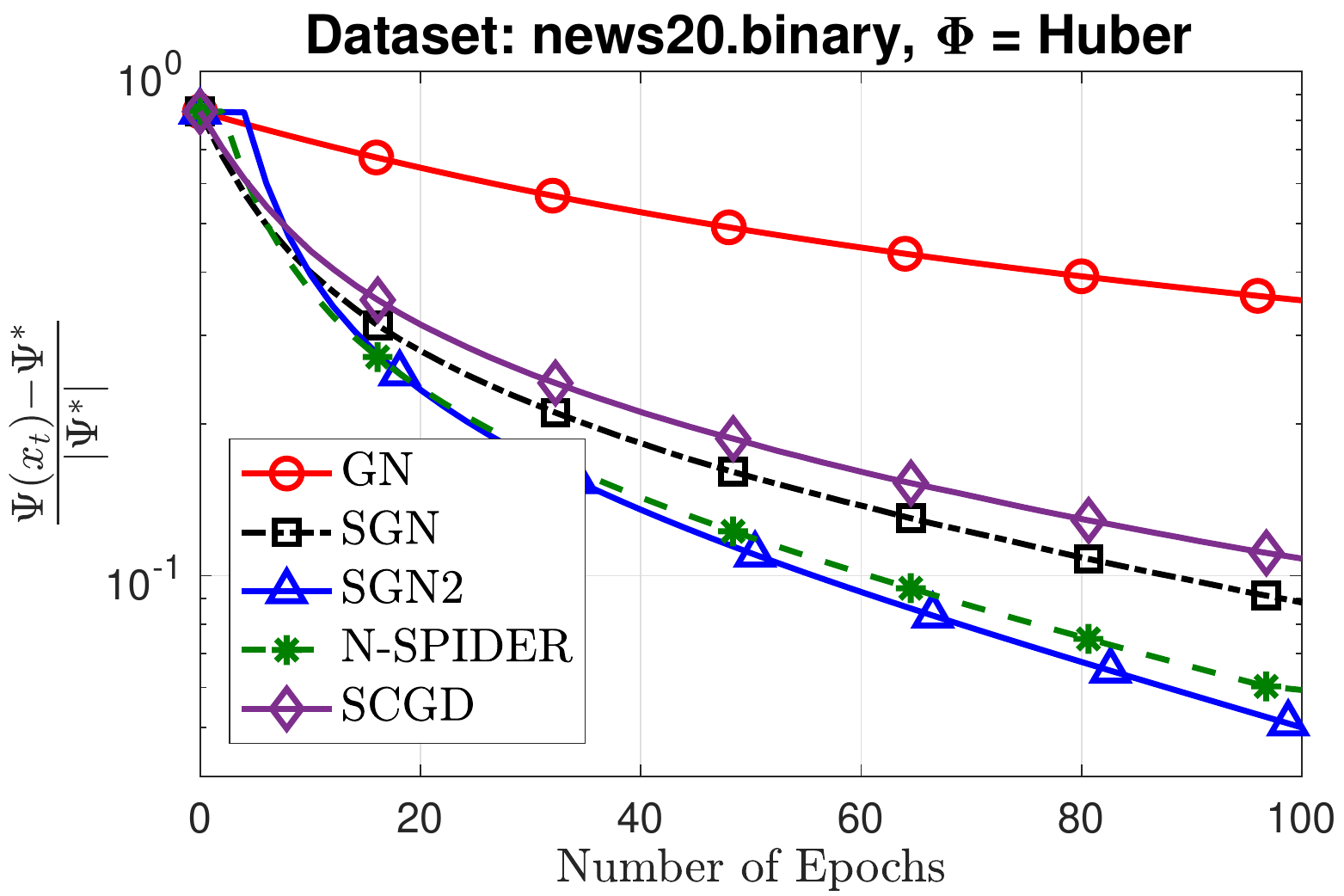}
    \caption{The performance of three algorithms on additional real datasets when using Huber loss.}\label{fig:exp_1_add_huber}
\end{center}
\end{figure}

In addition, we also run 5 algorithms on these datasets in the smooth case when using the Huber loss. 
We still tune the parameters for these algorithms and obtain the learning rate of $1.0$ for both N-SPIDER and SCGD. 
We again use $\varepsilon = 10^{-1}$ for N-SPIDER. 
More details about other parameters selection are presented in Table~\ref{tab:parameters_huber} and the performance of these algorithms are shown in Figure~\ref{fig:exp_1_add_huber}.

From Figure~\ref{fig:exp_1_add_huber}, SGN2 performs better than other algorithms in most cases while N-SPIDER is better than SGN and somewhat comparable with SGN2 in the \texttt{rcv1\_train.binary} and \texttt{news20.binary} datasets. SGN and SCGD appear to have similar behavior, but SGN is slightly better than SCGD in these datasets.

\beforesubsec
\subsection{Optimization Involving Expectation Constraints}\label{subsec:exam2_new}
\aftersubsec
We consider an optimization problem involving expectation constraints as described in \eqref{eq:min_with_e_constr}.
As mentioned, this problem has various applications in different fields, including optimization with conditional value at risk (CVaR) constraints and metric learning, see, e.g., \citet{lan2016algorithms} for detailed discussion.

Instead of solving the constrained setting \eqref{eq:min_with_e_constr}, we consider its exact penalty formulation \eqref{eq:penalty_form}:
\begin{equation*}
\min_{x\in\R^p}\Big\{ \Psi(x) := g(x) + \phi(\Exps{\xi}{\Fb(x, \xi)}) \Big\},
\tag{\ref{eq:penalty_form}}
\end{equation*}
where $\phi(u) := \rho\sum_{i=1}^q[u_i]_{+} $ with  $[u]_{+} :=\max\set{0, u}$ is a penalty function, and $\rho > 0$ is a given penalty parameter.
It is well-known that under mild conditions and $\rho$ sufficiently large (e.g., $\rho > \norms{y^{\star}}^{*}$, the dual norm of the optimal Lagrange multiplier $y^{\star}$), if $x^{\star}$ is a stationary point of \eqref{eq:penalty_form} and it is feasible to \eqref{eq:min_with_e_constr}, then it is also a stationary point of \eqref{eq:min_with_e_constr}.

As a concrete instance of \eqref{eq:min_with_e_constr}, we solve the following asset allocation problem studied in \citet{rockafellar2000optimization,lan2016algorithms}:
\myeq{eq:cvar0}{
\left\{\begin{array}{ll}
\displaystyle\min_{z\in\R^p, \tau \in [\underline{\tau}, \bar{\tau}]} &-c^{\top}z  \vspace{1ex}\\
\mathrm{s.t}~& \tau + \frac{1}{\beta n}\sum_{i=1}^n[-\xi_i^{\top}z - \tau]_{+} \leq 0, \vspace{1ex}\\
& z \in \Delta_p := \set{ \hat{z} \in \R^p_{+} \mid \sum_{i=1}^p\hat{z}_i = 1}.
\end{array}\right.
}
Here, $\Delta_p$ denotes the standard simplex in $\R^p$, and $[\underline{\tau}, \bar{\tau}]$ is a given range of $\tau$.
The  exact penalty formulation of \eqref{eq:cvar0} is given by \eqref{eq:cvar}:
\begin{equation*}
\min_{z\in\Delta^p,  \tau \in [\underline{\tau}, \bar{\tau}]} \left\{ -c^{\top}z + \phi\left( \tau + \frac{1}{\beta n}\sum_{i=1}^n[-\xi_i^{\top}z - \tau]_{+}\right) \right\},
\tag{\ref{eq:cvar}}
\end{equation*}
where $\phi(u) := \rho[u]_{+}$ with given $\rho > 0$.
However, since $[-\xi_i^{\top}z - \tau]_{+}$ is nonsmooth, we smooth it by $\sqrt{(\xi_i^{\top}z + \tau)^2 + \gamma^2} - \gamma -\xi_i^{\top}z - \tau$ for sufficiently small value of $\gamma > 0$.
Hence, \eqref{eq:cvar} can be approximated by 
\begin{equation}\label{eq:cvar_sm}
\min_{z\in\Delta_p, \tau \in [\underline{\tau}, \bar{\tau}]} \left\{ -c^{\top}z + \phi\left( \tau + \frac{1}{\beta n}\sum_{i=1}^n\left[\sqrt{(\xi_i^{\top}z + \tau)^2 + \gamma^2} - \gamma -\xi_i^{\top}z - \tau\right]\right) \right\}.
\end{equation}
If we introduce $x := (z, \tau)$, $\Fb(x, \xi) := \tau +  \frac{1}{2\beta}\left( \sqrt{(\xi_i^{\top}z + \tau)^2+\gamma^2} - \gamma - \xi_i^{\top}z - \tau\right)$ for $i=1,\cdots, n$, and $g(x) = -c^{\top}z + \delta_{\Delta_p 
\times [\underline{\tau},\bar{\tau}]}(x)$, where $\delta_{\Xc}$ is the indicator of $\Xc$, then we can reformulate \eqref{eq:cvar_sm}  into \eqref{eq:composite_form}.
It is obvious to check that $\Fb(\cdot, \xi)$ is Lipschitz continuous with $M_{i} := 1 + \frac{\norms{\xi_i}+1}{\beta\gamma}$ and its gradient $\Fb'(\cdot,\zeta)$ is also Lipschitz continuous with $L_{i} := \frac{\norms{\xi_i}^2}{2\beta\gamma}$.
Hence, Assumptions~\ref{as:A1} and \ref{as:A3} hold.

\textbf{Datasets.}
We consider both synthetic and US stock datasets.
For the synthetic datasets, we follow the procedures from \citet{lan2012validation} to generate the data with $n = 10^5$ and $p \in \set{300, 500, 700}$. 
We obtain real datasets of US stock prices for $889$, $865$, and $500$ types of stocks as described, e.g., \citet{SunTran2017gsc}.
Then, we apply a bootstrap strategy to resample in order to obtain three corresponding new datasets of sizes $n = 10^5$.

\begin{table}[hpt!]
\vspace{-2ex}
\caption{Hyper-parameter configuration of the two algorithms on 6 datasets in the asset allocation example.}
\label{tab:parameters_asset}
\vspace{2ex}
\centering
\resizebox{0.9\textwidth}{!}{%
\begin{tabular}{|c|c|c|c|c|c|c|c|c|c|}
\hline
\multirow{2}{*}{Algorithm} & \multicolumn{3}{c|}{Synthetic: p = 300} & \multicolumn{3}{c|}{Synthetic: p = 500} & \multicolumn{3}{c|}{Synthetic: p = 700} \\ \cline{2-10} 
 & $\hat{b}_t$ & $b_t$ & Inner Iterations & $\hat{b}_t$ & $b_t$ & Inner Iterations & $\hat{b}_t$ & $b_t$ & Inner Iterations \\ \hline
SGN & 1,024 & 2,048 &  & 1,024 & 2,048 &  & 1,024 & 2,048 &  \\ \hline
SGN2 & 128 & 256 & 5,000 & 128 & 256 & 2,000 & 256 & 512 & 2,000 \\ \hline
\multirow{2}{*}{Algorithm} & \multicolumn{3}{c|}{US Stock 1: p = 889} & \multicolumn{3}{c|}{US Stock 1: p = 865} & \multicolumn{3}{c|}{US Stock 1: p = 500} \\ \cline{2-10} 
 & $\hat{b}_t$ & $b_t$ & Inner Iterations & $\hat{b}_t$ & $b_t$ & Inner Iterations & $\hat{b}_t$ & $b_t$ & Inner Iterations \\ \hline
SGN & 512 & 1,024 &  & 512 & 1,024 &  & 512 & 1,024 &  \\ \hline
SGN2 & 128 & 256 & 5,000 & 128 & 256 & 5,000 & 128 & 256 & 5,000 \\ \hline
\end{tabular}%
}
\end{table}

\textbf{Parameter selection.}
We fix the smoothness parameter $\gamma := 10^{-3}$ and choose the range $[\underline{\tau}, \bar{\tau}]$ to be $[0, 1]$.
The parameter $\beta := 0.1$ as discussed in \citet{lan2016algorithms}.
Note that we do not use the theoretical values for $M$ as in our theory since that value is obtained in the worst-case.
We were instead experimenting different values for the penalty parameter $\rho$ and $M$, and eventually get  $\rho := 5$ and $M := 5$ as default values for this example.

\textbf{Experiment setup.}
We implement our algorithms: SGN and SGN2, and also a baseline variant, the deterministic GN scheme (i.e., we exactly evaluate $F$ and its Jacobian using the full batches) as in the first example.
Similar to the first example, we sweep over the same set of possible mini-batch sizes, and the chosen parameters are reported in Table~\ref{tab:parameters_asset}.

\begin{figure}[H]
\begin{center}
    \includegraphics[width = 0.45\textwidth]{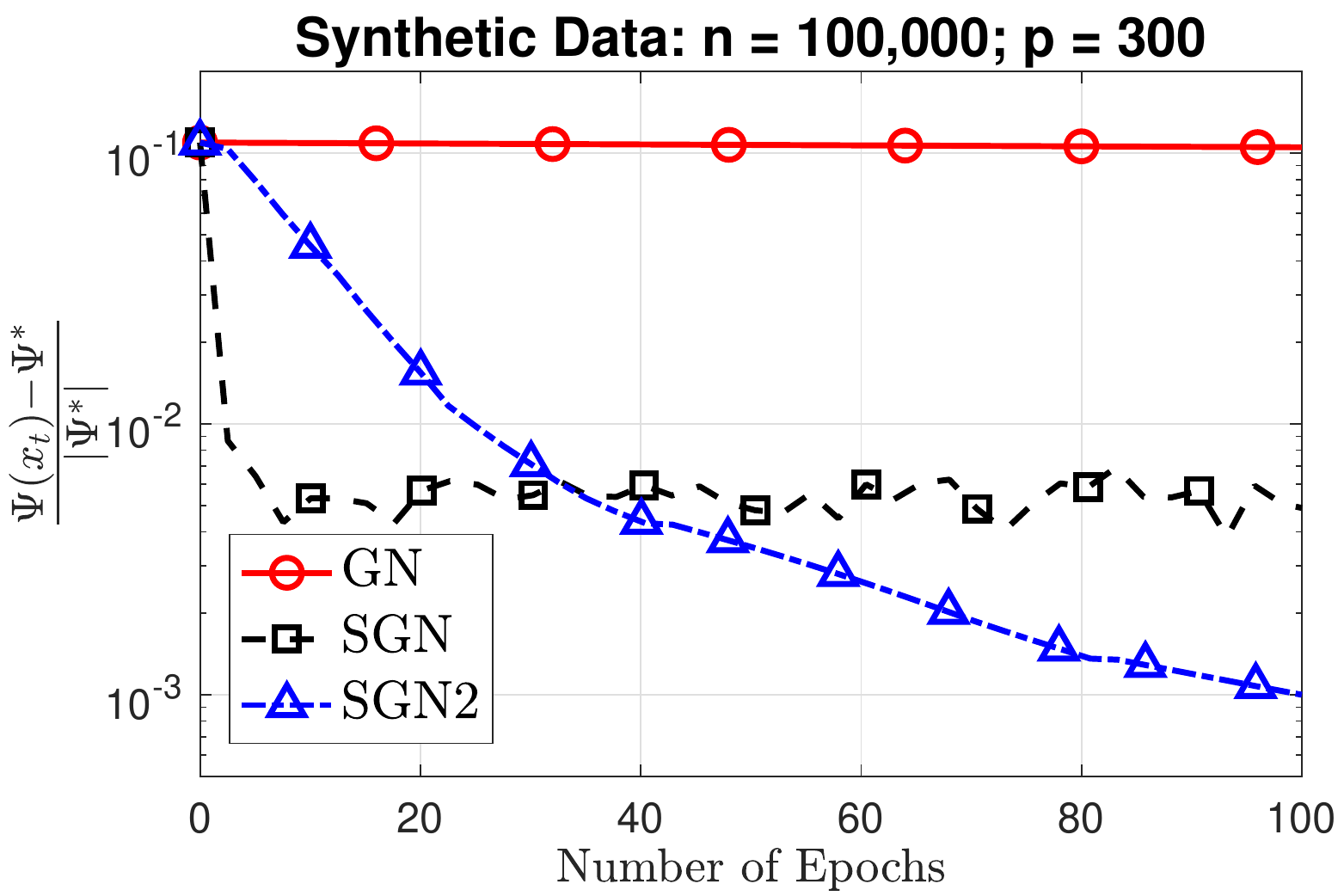}
    \includegraphics[width = 0.45\textwidth]{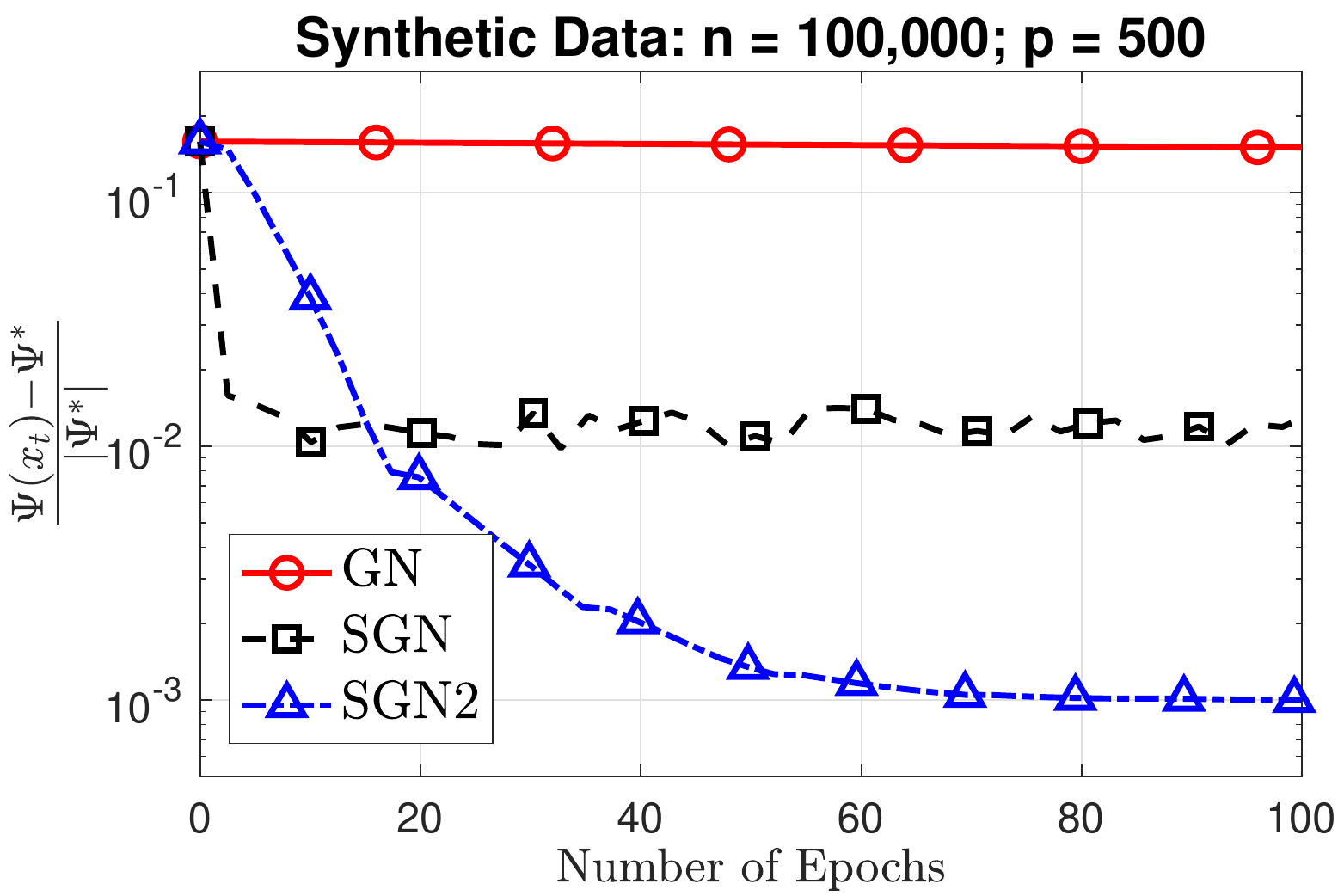}
    \vspace{2ex}\\
     \includegraphics[width = 0.45\textwidth]{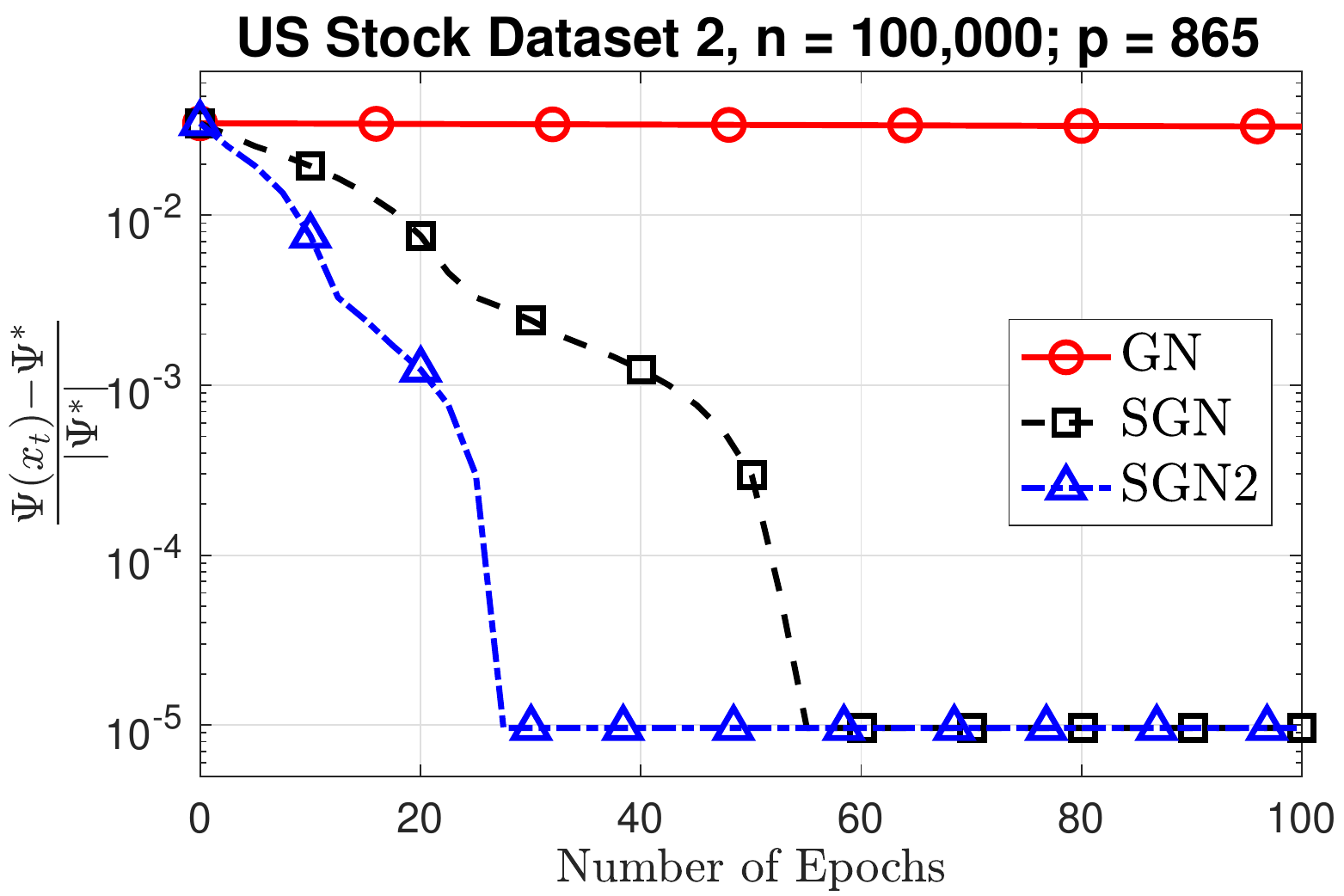}
    \includegraphics[width = 0.45\textwidth]{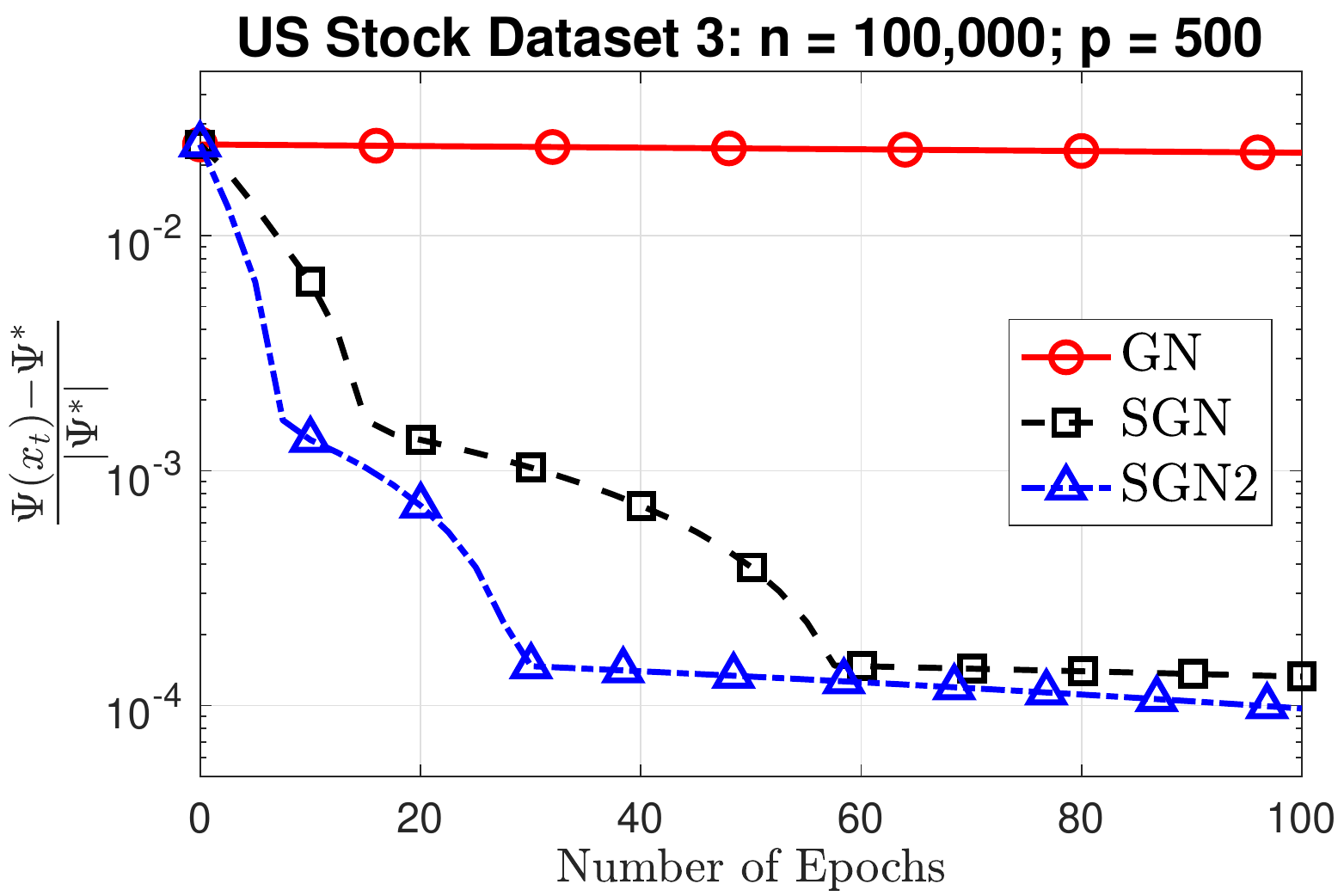}
    \caption{The performance of the three algorithms on two synthetic and two real datasets.}\label{fig:exp_2b}
\end{center}
\end{figure}

\textbf{Additional experiments.}
We run three algorithms: GN, SGN, and SGN2 with $3$ synthetic datasets, where the first one was reported in Figure~\ref{fig:exp_2} of the main text.
We also use two other US Stock datasets and the performance of three algorithms on these synthetic and real datasets are revealed in Figure~\ref{fig:exp_2b}.

Clearly, SGN2 is the best, while SGN still outperforms GN in these two datasets.
We believe that this experiment confirms our theoretical results presented in the main text.

\beforesec
\section{Convergence of Algorithm~\ref{alg:A2} for the finite-sum case \eqref{eq:finite_sum} without Assumption~\ref{as:A3}}\label{apdx:th:convergence_of_Sarah_GN2}
\aftersec
Although Theorem~\ref{eq:nl_least_squares} significantly improves stochastic oracle complexity of Algorithm~\ref{alg:A2} compared to Theorem~\ref{th:sgd_complexity1}, it requires additional assumption, Assumption~\ref{as:A3}.
Assumption~\ref{as:A3} is usually used in compositional models such as neural network and parameter estimations.
However, we still attempt to establish a convergence and complexity result for Algorithm~\ref{alg:A2} to solve \eqref{eq:finite_sum} without Assumption~\ref{as:A3} in the following theorem.

\begin{theorem}\label{th:convergence_of_Sarah_GN2}
Suppose that Assumptions~\ref{as:A1} and \ref{as:A2} are satisfied for \eqref{eq:finite_sum}.
Let $\sets{x_t^{(s)}}_{t=0\to m}^{s=1\to S}$ be  generated by Algorithm~\ref{alg:A2} to solve \eqref{eq:finite_sum}.
Let the mini-batches $b_s$, $\hat{b}_s$, $b_t^{(s)}$, and $\hat{b}_t^{(s)}$ be set as follows:
\begin{equation}\label{eq:mini_batches}
\left\{\begin{array}{lcl}
b_s &:= & \BigO{\frac{\sigma_F^2}{\varepsilon^4} \cdot \log\left(\frac{p+1}{\delta}\right)}, \vspace{1ex}\\
b_t^{(s)} &:= &  \BigO{\frac{m^2}{\varepsilon^2} \cdot \log\left(\frac{p+1}{\delta}\right)}, \vspace{1ex}\\
\hat{b}_s &:= & \BigO{\frac{\sigma_D^2}{\varepsilon^2} \cdot \log\left(\frac{p+q}{\delta}\right)},\vspace{1ex}\\
\hat{b}_t^{(s)} &:= & \BigO{m^2 \cdot \log\left(\frac{p+1}{\delta}\right)}.
\end{array}\right.
\end{equation}
Then, with probability at least $1-\delta$, the following statements hold:

$\mathrm{(a)}$~The following bound holds
\begin{equation*} 
{\!\!\!\!\!\!\!}\begin{array}{ll}
\dfrac{1}{S(m+1)}\displaystyle\sum_{s=1}^S\sum_{t=0}^m  \norms{\widetilde{G}_M(x_t)}^2   \leq  \BigO{\varepsilon^2}.
\end{array}{\!\!\!\!\!\!}
\end{equation*}
$\mathrm{(b)}$~The total number of iterations $T$ to achieve 
\begin{equation*}
\dfrac{1}{S(m+1)}\displaystyle\sum_{s=1}^S\sum_{t=0}^m  \norms{\widetilde{G}_M(x_t)}^2 \leq \varepsilon^2
\end{equation*}
is at most $T := S(m+1) = \BigO{ \frac{\left[\Psi(\widetilde{x}^0) - \Psi^{\star}\right]}{\varepsilon^2}}$.
Moreover, the total stochastic oracle calls  $\Tc_f$ and $\Tc_d$ to approximate $F$ and its Jacobian $F'$, respectively do not exceed
\begin{equation*} 
\left\{\begin{array}{lcl}
\Tc_f &:= & \BigO{\left(\frac{\sigma_F^2}{\varepsilon^5} + \frac{1}{\varepsilon^6}\right)\left[\Psi(\widetilde{x}^0) - \Psi^{\star}\right]\cdot \log\left(\frac{p+1}{\delta}\right)}, \vspace{1ex}\\
\Tc_d &:= & \BigO{\left(\frac{\sigma_D^2}{\varepsilon^3} + \frac{1}{\varepsilon^4} \right)  \left[\Psi(\widetilde{x}^0) - \Psi^{\star}\right] \cdot \log\left(\frac{p+q}{\delta}\right)}. 
\end{array}\right.
\end{equation*}
\end{theorem}

\begin{remark}\label{re:downside_of_this_result}
Although we do not gain an improvement on the worst-case oracle complexity through Theorem~\ref{th:convergence_of_Sarah_GN2}, we observe in our experiment that Algorithm~\ref{as:A2} highly outperforms \textbf{SGN}.
There could be an artifact in our proof of Theorem~\ref{th:convergence_of_Sarah_GN2}.
\end{remark}

\begin{proof}
We first analyze the inner loop of Algorithm~\ref{alg:A2}.
For simplicity of notation, we drop the superscript $^{(s)}$ in the following derivations until it is recalled.
We first verify the conditions \eqref{eq:a_cond2} if we use the SARAH estimators \eqref{eq:SARAH_estimators} for $F'(x_t)$ and $F(x_t)$.
Let $\Fc_t := \sigma(x_0, x_1, \cdots, x_{t-1})$ be the $\sigma$-field generated by $\set{x_0, x_1, \cdots, x_{t-1}}$.
We define $X_i := F_i'(x_t) - F_i'(x_{t-1}) - [F'(x_t) - F'(x_{t-1})]$. 
Then, clearly, conditioned on $\Fc_{t}$, we have $\set{ X_i }_{i\in\hat{\Bc}_t}$ is mutually independent and $\Exp{X_i \mid \Fc_{t-1}} = 0$. 
Moreover, by Assumption~\ref{as:A2}, we have
\begin{equation*}
\begin{array}{lcl}
\norms{X_i} &= &  \norms{F_i'(x_t) - F_i'(x_{t-1}) - [F'(x_t) - F'(x_{t-1})]} \vspace{1ex}\\
& \leq & \norms{F_i'(x_t) - F_i'(x_{t-1})} + \norms{F'(x_t) - F'(x_{t-1})} \vspace{1ex}\\
& \leq &  2L_F\norms{x_t - x_{t-1}} =: \hat{\sigma}_t.
\end{array}
\end{equation*}
We consider $Z_t := \frac{1}{\hat{b}_t}\sum_{i\in\hat{\Bc}_t}[F_i'(x_t) - F_i'(x_{t-1}) - F'(x_t) + F'(x_{t-1})] = \frac{1}{\hat{b}_t}\sum_{i\in\hat{\Bc}_t}X_i$.
We have
\begin{equation*}
\sigma_X^2 := \max\set{\Big\Vert\sum_{i\in\hat{\Bc}_t}\Exp{X_iX_i^{\top}\mid\Fc_{t-1}}\Big\Vert, \Big\Vert\sum_{i\in\hat{\Bc}_t}\Exp{X_i^{\top}X_i\mid\Fc_{t-1}}\Big\Vert} \leq \sum_{i\in\hat{\Bc}_t}\Exp{\norms{X_i}^2\mid\Fc_{t-1}} \leq \hat{b}_t\hat{\sigma}_t^2.
\end{equation*}
For any $\hat{\epsilon}  > 0$, we can apply Lemma~\ref{le:con_lemma} to obtain the following bound
\begin{equation*}
\begin{array}{lcl}
\Prob{\norms{Z_t} \leq \hat{\epsilon}\hat{\sigma}_t} &= & \Prob{\norms{\sum_{i\in\hat{\Bc}_t}X_i} \leq \hat{b}_t\hat{\epsilon}\hat{\sigma}_t} \vspace{1ex}\\
& \geq & 1 - (p + q)\exp\left(-\frac{3\hat{b}_t^2\hat{\epsilon}^2\hat{\sigma}_t^2}{6\hat{b}_t\hat{\sigma}_t^2 + 2\hat{\sigma}_t\hat{b}_t\hat{\epsilon}\hat{\sigma}_t}\right) \vspace{1ex}\\
& = & 1 - (p + q)\exp\left(-\frac{3\hat{b}_t\hat{\epsilon}^2}{6 + 2\hat{\epsilon}}\right).
\end{array}
\end{equation*}
Hence, if we choose  $\delta \geq (p + q)\exp\left(-\frac{3\hat{b}_t\hat{\epsilon}^2}{6 + 2\hat{\epsilon}}\right)$ and $\delta \leq 1$, we obtain $\Prob{\norms{Z_t} \leq \hat{\epsilon}\hat{\sigma}_t} \geq 1 - \delta$ for all $t\geq 0$.
The condition in $\hat{b}_t$ leads to 
\begin{equation*}
\hat{b}_t \geq \frac{6 + 2 \hat{\epsilon}}{3\hat{\epsilon}^2} \cdot \log\left(\frac{p+q}{\delta}\right).
\end{equation*}
By the update \eqref{eq:SARAH_estimators}, we have $\widetilde{J}_t - F(x_t) = [\widetilde{J}_{t-1} - F'(x_{t-1})] + \frac{1}{\hat{b}_t}\sum_{i\in\hat{\Bc}_t}[F_i'(x_t) - F_i'(x_{t-1}) - F'(x_t) + F'(x_{t-1})] = [\widetilde{J}_{t-1} - F'(x_{t-1})] + Z_t$.
Hence, by the triangle inequality, we get
\begin{equation*}
\norms{\widetilde{J}_t - F'(x_t)} = \norms{\widetilde{J}_0 - F'(x_0) + \sum_{j=1}^tZ_j} \leq \norms{\widetilde{J}_0 - F'(x_0)} + \sum_{j=1}^t\norms{Z_j}.
\end{equation*}
On the other hand, by the update \eqref{eq:msgd_estimators} of $\widetilde{J}_0$ as $\widetilde{J}_0 := \frac{1}{\hat{b}}\sum_{i\in\hat{\Bc}}F_i'(x_0)$, where $\hat{b} := \hat{b}_s$ and $\hat{\Bc} := \hat{\Bc}_s$, with a similar proof as of Theorem~\ref{th:sgd_complexity2}, we can show that if we choose $\hat{b} \geq \frac{6\sigma_D^2 + 2\sigma_D\hat{\epsilon}_0}{3\hat{\epsilon}_0^2}\log\left(\frac{p+q}{\delta}\right)$ then $\Prob{\norms{\widetilde{J}_0 - F'(x_0)} \leq \hat{\epsilon}_0} \geq 1 - \delta$.
Then with probability at least $1-\delta$, we have
\begin{equation*} 
\norms{\widetilde{J}_t - F'(x_t)} \leq \norms{\widetilde{J}_0 - F'(x_0)} + \sum_{j=1}^t\norms{Z_j} \leq \hat{\epsilon}_0 + \hat{\epsilon}\sum_{s=1}^t\hat{\sigma}_s = \hat{\epsilon}_0 + 2L_F\hat{\epsilon}\sum_{s=1}^t\norms{x_s - x_{s-1}}.
\end{equation*}
This inequality implies
\begin{equation}\label{eq:Jt_upper_bound}
\norms{\widetilde{J}_t - F'(x_t)}^2 \leq  2\hat{\epsilon}_0^2 + 8L_F^2\hat{\epsilon}^2t\sum_{s=1}^t\norms{x_s - x_{s-1}}^2.
\end{equation}
Our next step is to estimate the $\norms{\widetilde{F}_t - F(x_t)}$.
We define $Y_i := F_i(x_t) - F_i(x_{t-1}) - [F(x_t) - F(x_{t-1})]$ and $U_j := F_j(x_t) - F_j(x_{t-1}) - F_j'(x_{t-1})(x_t - x_{t-1})$ for $j\in [n]$.
In this case, $\set{Y_i}_{i\in\Bc_t}$ is mutually independent and $\Exp{Y_i} = 0$.
We also have 
\begin{equation*}
\begin{array}{lcl}
\norms{Y_i} &=&  \Big\Vert F_i(x_t) - F_i(x_{t-1}) - \frac{1}{n}\sum_{j=1}^n[F_j(x_t) - F_j(x_{t-1})] \Big\Vert \vspace{1ex}\\
&= & \Big\Vert F_i(x_t) - F_i(x_{t-1}) - F_i'(x_{t-1})(x_t - x_{t-1}) + \frac{1}{n}\sum_{j=1}^n[F_i'(x_{t-1}) - F_j'(x_{t-1})](x_t - x_{t-1}) \vspace{1ex}\\
&& - {~} \frac{1}{n}\sum_{j=1}^n[F_j(x_t) - F_j(x_{t-1}) - F_j'(x_{t-1})(x_t - x_{t-1})] \Big\Vert \vspace{1ex}\\
&\leq & \frac{1}{n}\Big\Vert \sum_{j=1,j\neq i}^n[U_i  - U_j] \Big\Vert + \Vert [F'_i(x_{t-1}) - F'(x_{t-1})](x_t - x_{t-1})\Vert \vspace{1ex}\\
&\leq & \frac{1}{n}\sum_{j=1,j\neq i}\norm{U_j} + \frac{n-1}{n}\norms{U_i} + \Vert F'_i(x_{t-1}) - F'(x_{t-1})\Vert \norm{x_t - x_{t-1}} \vspace{1ex}\\
&\leq & \frac{(n-1)L_F}{n}\norms{x_t - x_{t-1}}^2 + \sigma_D\norms{x_t - x_{t-1}}.
\end{array}
\end{equation*}
Here, we use the facts that $\norms{U_j} = \Vert F_j(x_t) - F_j(x_{t-1}) - F_j'(x_{t-1})(x_t - x_{t-1})\Vert \leq \frac{L^2}{2}\norms{x_t - x_{t-1}}^2$ for $j\in [n]$ and $\norms{F'_i(x_{t-1}) - F'(x_{t-1})} \leq \sigma_D$ from Assumption~\ref{as:A2} into the last inequality.
Moreover, we have
\begin{equation*}
\sigma_Y^2 := \max\set{\Big\Vert\sum_{i\in\Bc_t}\Exp{Y_iY_i^{\top}}\Big\Vert, \Big\Vert\sum_{i\in\Bc_t}\Exp{Y_i^{\top}Y_i}\Big\Vert} \leq \sum_{i\in\Bc_t}\Exp{\norms{Y_i}^2} \leq b_t\sigma_t^2,
\end{equation*}
where $\sigma_t := L_F\norms{x_t - x_{t-1}}^2 + \sigma_D\norms{x_t - x_{t-1}}$.

Now, we consider $W_t := \frac{1}{b_t}\sum_{i\in\Bc_t}Y_i = \frac{1}{b_t}\sum_{i\in\Bc_t}[F_i(x_t) - F_i(x_{t-1}) - F(x_t) + F(x_{t-1})]$.
For any $\epsilon > 0$, we can apply Lemma~\ref{le:con_lemma} to obtain the following bound
\begin{equation*}
\begin{array}{lcl}
\Prob{\norms{W_t} \leq \epsilon\sigma_t} &= & \Prob{\norms{\sum_{i\in\Bc_t}Y_i} \leq \epsilon b_t\sigma_t} \geq 1 - (p+1)\exp\left(-\frac{3b_t^2\epsilon^2\sigma_t^2}{6 b_t \sigma_t^2 + 2\sigma_tb_t\epsilon\sigma_t}\right) \vspace{1ex}\\
& = & 1 - (p+1)\exp\left(-\frac{3b_t\epsilon^2}{6  + 2\epsilon}\right).
\end{array}
\end{equation*}
Hence, if we choose  $\delta \geq (p+1)\exp\left(-\frac{3b_t\epsilon^2}{6 + 2\epsilon}\right)$ and $\delta \leq 1$, then we obtain $\Prob{\norms{W_t} \leq \epsilon\sigma_t} \geq 1 - \delta$ for all $t\geq 0$.
The condition in $b_t$ leads to $b_t \geq \frac{6 + 2\epsilon}{3\epsilon^2} \cdot \log\left(\frac{p+1}{\delta}\right)$.

Note that since $\widetilde{F}_0 := \frac{1}{b}\sum_{i\in\Bc}F_i(x_0)$ is updated by \eqref{eq:msgd_estimators}, to guarantee 
\begin{equation*}
\Prob{\norms{\widetilde{F}_0 - F(x_0)} \leq \epsilon_0} \geq 1 - \delta,
\end{equation*}
we choose the mini-batch size $b \geq  \frac{6\sigma_F^2 + 2\sigma_F\epsilon_0}{3\epsilon_0^2}\log\left(\frac{p+1}{\delta}\right)$.

By the update of $\widetilde{F}_t$ from \eqref{eq:SARAH_estimators}, we have $\widetilde{F}_t - F(x_t) = [\widetilde{F}_{t-1} - F(x_{t-1})] + \frac{1}{b_t}[F_i(x_t) - F_i(x_{t-1}) - F(x_t) + F(x_{t-1})] =  [\widetilde{F}_{t-1} - F(x_{t-1})] + W_t$.
Hence, by induction, it implies that $\widetilde{F}_t - F(x_t) = [\widetilde{F}_0 - F(x_0)] + \sum_{s=1}^tW_s$, which leads to
\begin{equation}\label{eq:Ft_upper_bound} 
\norms{\widetilde{F}_t - F(x_t)} \leq \norms{\widetilde{F}_0 - F(x_0)} + \sum_{s=1}^t\norms{W_s} \leq  \epsilon_0 + \epsilon\sum_{s=1}^t\left[L_F\norms{x_s - x_{s-1}}^2 + \sigma_D\norms{x_s - x_{s-1}}\right].
\end{equation}
Now, we analyze the inner loop of $t = 0$ to $m$.
Using \eqref{eq:key_est5} with $x := x_t^{(s)}$ and $\widetilde{T}_M(x) = x^{(s)}_{t+1}$, we have
\begin{equation}\label{eq:sarah2_proof1}
\phi(F(x^{(s)}_{t+1}))  \leq  \phi(F(x_t^{(s)})) - \frac{C_g}{2}  \norms{x^{(s)}_{t+1} - x^{(s)}_t}^2  +   2M_{\phi} \norms{F(x^{(s)}_t) - \widetilde{F}(x^{(s)}_t)} + \frac{M_{\phi}}{2\beta_d} \Vert F'(x^{(s)}_t) - \widetilde{J}(x^{(s)}_t)\Vert^2,
\end{equation}
where $C_g := 2M - M_{\phi}(L_F + \beta_d) > 0$ and  $\beta_d > 0$ is given.
Combining \eqref{eq:sarah2_proof1}, \eqref{eq:Jt_upper_bound}, and \eqref{eq:Ft_upper_bound}, we have
\begin{equation*} 
\begin{array}{lcl}
\phi(F(x^{(s)}_{t+1}))  &\leq &  \phi(F(x_t^{(s)})) - \frac{C_g}{2}  \norms{x^{(s)}_{t+1} - x^{(s)}_t}^2  +   2M_{\phi}\left[\epsilon_0 +  L_F\epsilon \sum_{j=1}^t\norms{x_j^{(s)} - x_{j-1}^{(s)}}^2 \right] \vspace{1ex}\\
&& + {~} \frac{L_F}{2\beta_d}\left[ 2\hat{\epsilon}_0^2 + 8L_F^2\hat{\epsilon}^2t\sum_{j=1}^t\norms{x_j^{(s)} - x_{j-1}^{(s)}}^2\right] + 2M_{\phi}\sigma_D\epsilon \sum_{j=1}^t\norms{x_j^{(s)} - x_{j-1}^{(s)}}.
\end{array}
\end{equation*}
Summing up this inequality from $t=0$ to $t=m$, we obtain
\begin{equation*} 
\begin{array}{lcl}
\phi(F(x^{(s)}_{m+1}))  & \leq & \phi(F(x_0^{(s)})) - \frac{C_g}{2} \sum_{t=0}^m \norms{x^{(s)}_{t+1} - x^{(s)}_t}^2  +   2M_{\phi}(m+1)\epsilon_0 + \frac{M_{\phi}(m+1)\hat{\epsilon}_0^2}{\beta_d} \vspace{1ex}\\
&& + {~} 2M_{\phi}L_F\epsilon\sum_{t=0}^m\sum_{j=1}^t\norms{x_j^{(s)} - x_{j-1}^{(s)}}^2 
+ \frac{4L_F^3\hat{\epsilon}^2}{\beta_d}\sum_{t=0}^mt\sum_{j=1}^t\norms{x_j^{(s)} - x_{j-1}^{(s)}}^2 \vspace{1ex}\\
&& +  {~} 2M_{\phi}\sigma_D\epsilon \sum_{t=0}^m\sum_{j=1}^t\norms{x_j^{(s)} - x_{j-1}^{(s)}}.
\end{array}
\end{equation*}
Since $\widetilde{x}^{s-1} = x_0^{(s)}$ and $\widetilde{x}^s = x_{m+1}^{(s)}$, the last inequality becomes
\begin{equation}\label{eq:sarah2_proof_new1}
\phi(F(\widetilde{x}^{s}))  \leq  \phi(F(\widetilde{x}^{s-1})) - \frac{C_g}{4} \sum_{t=0}^m \norms{x^{(s)}_{t+1} - x^{(s)}_t}^2  +   2M_{\phi}(m+1)\epsilon_0 + \frac{M_{\phi}(m+1)m\epsilon^2}{2\gamma} + \frac{M_{\phi}(m+1)\hat{\epsilon}_0^2}{\beta_d} + \Tc_m^s,
\end{equation}
where $\Tc_m^s$ is defined as
\begin{equation*}
\begin{array}{lcl}
\Tc_m^s &:= & 2M_{\phi} L_F\epsilon\sum_{t=0}^m\sum_{j=1}^t\norms{x_j^{(s)} - x_{j-1}^{(s)}}^2 + 2M_{\phi} \sigma_D\epsilon\sum_{t=0}^m\sum_{j=1}^t\norms{x_j^{(s)} - x_{j-1}^{(s)}} \vspace{1ex}\\
&& + {~} \frac{4L_F^3\hat{\epsilon}^2}{\beta_d}\sum_{t=0}^mt\sum_{j=1}^t\norms{x_j^{(s)} - x_{j-1}^{(s)}}^2 - \frac{C_g}{4} \sum_{t=0}^m \norms{x^{(s)}_{t+1} - x^{(s)}_t}^2.
\end{array}
\end{equation*}
Let $u_{t-1} := \norms{x_t^{(s)} - x_{t-1}^{(s)}}$.
Then, we can rewrite $\Tc_m^s$ as
\begin{equation*}
\begin{array}{lcl}
\Tc_m^s &= &  2M_{\phi} L_F\epsilon\left[u_0^2 + (u_0^2 + u_1^2) + \cdots + (u_0^2 + u_1^2 + u_{m-1}^2)\right] \vspace{1ex}\\
&& + {~} 2M_{\phi}\sigma_D\epsilon\left[u_0 + (u_0 + u_1) + \cdots + (u_0 + u_1 + u_{m-1})\right] \vspace{1ex}\\
&& + {~}  \frac{4L_F^3\hat{\epsilon}^2}{\beta_d}\left[u_0^2 + 2(u_0^2 + u_1^2) + \cdots + m(u_0^2 + u_1^2 + \cdots + u_{m-1}^2)\right] \vspace{1ex}\\
&& - {~} \frac{C_g}{4}\left[u_0^2 + u_1^2 + \cdots + u_{m}^2\right] \vspace{1ex}\\
&= & \left[2M_{\phi} L_F\epsilon m +  \frac{4L_F^3\hat{\epsilon}^2}{\beta_d}m(m+1) - \frac{C_g}{4}\right]u_0^2 + \left[2M_{\phi} L_F\epsilon(m-1) +  \frac{4L_F^3\hat{\epsilon}^2}{\beta_d}m(m-1) - \frac{C_g}{4}\right]u_1^2 + \cdots \vspace{1ex}\\
&& + {~} \left[ 2M_{\phi} L_F\epsilon +  \frac{4L_F^3\hat{\epsilon}^2}{\beta_d} - \frac{C_g}{4}\right]u_{m-1}^2 - \frac{C_g}{4}u_m^2  +  2M_{\phi}\sigma_D\epsilon\left[m u_0 + (m-1)u_1 + \cdots + u_{m-1}\right] \vspace{1ex}\\
&\leq & \left[2M_{\phi} L_F\epsilon m +  \frac{4L_F^3\hat{\epsilon}^2m(m+1)}{\beta_d} - \frac{C_g}{4}\right]u_0^2 + \left[2M_{\phi} L_F\epsilon(m-1) +  \frac{4L_F^3\hat{\epsilon}^2m(m-1)}{\beta_d} - \frac{C_g}{4}\right]u_1^2 + \cdots \vspace{1ex}\\
&& + {~} \left[2M_{\phi} L_F\epsilon +  \frac{4L_F^3\hat{\epsilon}^2}{\beta_d} - \frac{C_g}{4}\right]u_{m-1}^2 - \frac{C_g}{4}u_m^2 \vspace{1ex}\\
&& + {~} \frac{M_{\phi}\sigma_D m\epsilon}{\gamma} \left[u_0^2 + u_1^2 + \cdots + u_{m-1}^2\right] +  M_{\phi}\sigma_D\epsilon\gamma m^2 \vspace{1ex}\\
&\leq & \left[2M_{\phi} L_F\epsilon m + M_{\phi}\sigma_D\sqrt{\epsilon}m + \frac{4L_F^3\hat{\epsilon}^2m(m+1)}{\beta_d} - \frac{C_g}{4}\right]u_0^2 \vspace{1ex}\\
&& + {~} \left[2M_{\phi} L_F\epsilon(m-1) + M_{\phi}\sigma_D\sqrt{\epsilon}m +  \frac{4L_F^3\hat{\epsilon}^2m(m-1)}{\beta_d} - \frac{C_g}{4}\right]u_1^2 + \cdots \vspace{1ex}\\
&& + {~} \left[2M_{\phi} L_F\epsilon + \frac{M_{\phi}\sigma_D \epsilon m}{\gamma} +  \frac{4L_F^3\hat{\epsilon}^2}{\beta_d} - \frac{C_g}{4}\right]u_{m-1}^2 - \frac{C_g}{4}u_m^2 +  M_{\phi}\sigma_D\epsilon\gamma m^2.
\end{array}
\end{equation*}
If we impose the following condition
\begin{equation}\label{eq:cond_on_params}
M_{\phi}\left(2L_F  + \frac{\sigma_D}{\gamma}\right)\epsilon m + \frac{2L_F^3\hat{\epsilon}^2m(m+1)}{\beta_d} \leq \frac{C_g}{4},
\end{equation}
then $\Tc_m^{s} \leq M_{\phi}\sigma_D\epsilon\gamma m^2$.

Under this condition, \eqref{eq:sarah2_proof_new1} reduces to 
\begin{equation*} 
\phi(F(\widetilde{x}^{s}))  \leq  \phi(F(\widetilde{x}^{s-1})) - \frac{C_g}{4} \sum_{t=0}^m \norms{x^{(s)}_{t+1} - x^{(s)}_t}^2  +   2M_{\phi}(m+1)\epsilon_0 + \frac{M_{\phi}(m+1)\hat{\epsilon}_0^2}{\beta_d} + M_{\phi}\sigma_D\epsilon \gamma m^2.
\end{equation*}
Summing up this inequality from $s=1$ to $S=s$ and rearranging the result, we obtain
\begin{equation*}
\displaystyle\frac{1}{S(m+1)}\sum_{s=1}^S\sum_{t=0}^m \norms{x^{(s)}_{t+1} - x^{(s)}_t}^2  \leq \displaystyle\frac{4}{C_g(m+1)S}\left[\phi(F(\widetilde{x}^0)) - \phi(F(\widetilde{x}^S))\right] +  \frac{4M_{\phi}}{C_g}\left(2\epsilon_0 + \sigma_D m\gamma\epsilon + \frac{\hat{\epsilon}_0^2}{\beta_d}\right).
\end{equation*}
Using $\phi(F(\widetilde{x}^S) \geq \Psi^{\star}$ and $\Psi(x) = \phi(F(x))$, we obtain from the last inequality that
\begin{equation*} 
\displaystyle\frac{1}{S(m+1)}\sum_{s=1}^S\sum_{t=0}^m \norms{x^{(s)}_{t+1} - x^{(s)}_t}^2  \leq \displaystyle\frac{4}{C_g(m+1)S}\left[\Psi(\widetilde{x}^0) - \Psi^{\star}\right] + \frac{4M_{\phi}}{C_g}\left(2\epsilon_0 + \sigma_Dm\gamma\epsilon  + \frac{\hat{\epsilon}_0^2}{\beta_d}\right).
\end{equation*}
Clearly, if we choose $\epsilon_0 := C_0\varepsilon^2$, $\epsilon  :=  \frac{C_1\varepsilon^2}{\gamma m}$, $\hat{\epsilon}_0 := \sqrt{\hat{C}_0}\varepsilon$, and $\hat{\epsilon}^2 := \frac{\hat{C}_1}{m(m+1)}$ for some positive constant $C_0$, $C_1$, $\hat{C}_0$, and $\hat{C}_1$, then we obtain from the last estimate that
\begin{equation}\label{eq:sarah2_key_est1}
\displaystyle\frac{1}{S(m+1)}\sum_{s=1}^S\sum_{t=0}^m \norms{\widetilde{G}_M(x^{(s)}_t)}^2  \leq \displaystyle\frac{4M^2}{C_g(m+1)S}\left[\Psi(\widetilde{x}^0) - \Psi^{\star}\right] + \frac{M^2M_{\phi}}{C_g}\left(2C_0 + \sigma_DC_1 + \frac{\hat{C}_0}{\beta_d} \right)\varepsilon^2,
\end{equation}
where we use the fact that $\widetilde{G}_M(x^{(s)}_t) = M(x^{(s)}_{t+1} - x^{(s)}_t)$.
Now, assume that the condition \eqref{eq:cond_on_params} is tight. 
Using the choice of accuracies, we obtain 
\begin{equation*}
M_{\phi}\left(2L_F + \frac{\sigma_D}{\gamma}\right)\frac{C_1\varepsilon^2}{\gamma} + \frac{2L_F^3\hat{C}_1}{\beta_d} = \frac{C_g}{4}.
\end{equation*}
If we choose $\gamma := \varepsilon$, then this condition becomes $2M_{\phi}\left(L_F\varepsilon  +  \sigma_D \right)C_1 + \frac{2L_F^3\hat{C}_1}{\beta_d} = \frac{C_g}{4}$ and $\epsilon  :=  \frac{C_1\varepsilon}{m}$.

Now, with the choice of $\epsilon_0$, $\epsilon$, $\hat{\epsilon}_0$, and $\hat{\epsilon}$ as above, we can set the mini-batch sizes as follows:
\begin{equation}\label{eq:choice_of_batches}
\left\{\begin{array}{lclcl}
b_s &:= &  \left\lfloor \frac{6\sigma_F^2 + 2\sigma_FC_0\varepsilon^2}{3C_0^2\varepsilon^4} \cdot \log\left(\frac{p+1}{\delta}\right)  \right\rfloor  &= & \BigO{\frac{\sigma_F^2}{\varepsilon^4} \cdot \log\left(\frac{p+1}{\delta}\right)}, \vspace{1ex}\\
b_t^{(s)} &:= & \left\lfloor  \frac{m\left[6m + 2 C_1 \varepsilon\right]}{3C_1^2\varepsilon^2}\cdot \log\left(\frac{p+1}{\delta}\right) \right\rfloor  &= & \BigO{\frac{m^2}{\varepsilon^2} \cdot \log\left(\frac{p+1}{\delta}\right)}, \vspace{1ex}\\
\hat{b}_s &:= & \left\lfloor \frac{\left[6\sigma_D^2 + 2\sigma_D\sqrt{\hat{C}_1}\varepsilon\right]}{3\hat{C}_1\varepsilon^2} \cdot \log\left(\frac{p+q}{\delta}\right) \right\rfloor  &= &  \BigO{\frac{\sigma_D^2}{\varepsilon^2} \cdot \log\left(\frac{p+q}{\delta}\right)}, \vspace{1ex}\\
\hat{b}_t^{(s)} &:= &  \left\lfloor \frac{\sqrt{m(m+1)}\left[6\sqrt{m(m+1)} + 2 \sqrt{\hat{C}_1}\right]}{3\hat{C}_1} \cdot \log\left(\frac{p+q}{\delta}\right) \right\rfloor &= &  \BigO{m^2 \cdot \log\left(\frac{p+1}{\delta}\right)}.
\end{array}\right.
\end{equation}
Since $S(m+1) = \frac{8M^2\left[\Psi(\widetilde{x}^0) - \Psi^{\star}\right]}{C_g\varepsilon^2}$, if we choose $m := \frac{C}{ \varepsilon}$, then $S = \frac{8M^2\left[\Psi(\widetilde{x}^0) - \Psi^{\star}\right]}{C C_g\varepsilon}$.
The total complexity is
\begin{equation*}
\begin{array}{lcl}
\Tc_f &:= & \sum_{s=1}^Sb_s + \sum_{s=1}^S\sum_{t=0}^mb_t^{(s)} = \frac{[6\sigma_F^2 + 2\sigma_FC_0\varepsilon^2]S}{3C_0^2\varepsilon^4} \cdot \log\left(\frac{p+1}{\delta}\right) +   \frac{S(m+1)m\left[6m + 2 C_1 \varepsilon\right]}{3C_1^2\varepsilon^2}\cdot \log\left(\frac{p+1}{\delta}\right) \vspace{1ex}\\
&= & \frac{8M^2[6\sigma_F^2 + 2\sigma_FC_0\varepsilon^2]\left[\Psi(\widetilde{x}^0) - \Psi^{\star}\right]}{3CC_gC_0^2\varepsilon^5} \cdot \log\left(\frac{p+1}{\delta}\right)
+   \frac{ 8M^2\left[\Psi(\widetilde{x}^0) - \Psi^{\star}\right]\left[6C^2 + 2 CC_1 \varepsilon^2\right]}{3C_gC_1^2\varepsilon^6} \cdot \log\left(\frac{p+1}{\delta}\right) \vspace{1ex}\\
& = & \BigO{\frac{\sigma_F^2\left[\Psi(\widetilde{x}^0) - \Psi^{\star}\right]}{\varepsilon^5} \cdot \log\left(\frac{p+1}{\delta}\right)} + \BigO{\frac{\left[\Psi(\widetilde{x}^0) - \Psi^{\star}\right]}{\varepsilon^6} \cdot \log\left(\frac{p+1}{\delta}\right)},  \vspace{1ex}\\
\Tc_d &:= &  \sum_{s=1}^S\hat{b}_s + \sum_{s=1}^S\sum_{t=0}^m\hat{b}_t^{(s)} \vspace{1ex}\\
&= &  \frac{\left[6\sigma_D^2 + 2\sigma_D\sqrt{\hat{C}_1}\varepsilon\right]S}{3\hat{C}_1\varepsilon^2} \cdot \log\left(\frac{p+q}{\delta}\right)  + \frac{S(m+1)\sqrt{m(m+1)}\left[6\sqrt{m(m+1)} + 2 \sqrt{\hat{C}_1}\right]}{3\hat{C}_1} \cdot \log\left(\frac{p+q}{\delta}\right) \vspace{1ex}\\
&= & \BigO{\frac{\sigma_D^2\left[\Psi(\widetilde{x}^0) - \Psi^{\star}\right]}{\varepsilon^3} \cdot \log\left(\frac{p+q}{\delta}\right)} + \BigO{\frac{\left[\Psi(\widetilde{x}^0) - \Psi^{\star}\right]}{\varepsilon^4} \cdot \log\left(\frac{p+q}{\delta}\right)}. 
\end{array}
\end{equation*}
This proves our theorem.
\end{proof}

\end{document}